\documentclass[UTF8,reqno,11pt]{amsart}

\usepackage[top=2.0cm,bottom=2.0cm,left=2cm,right=2cm]{geometry}
\usepackage{amsthm,amsmath,amssymb,dsfont}
\usepackage{mathrsfs,amsfonts,functan,dsfont,bbm,extarrows,mathtools}
\usepackage[colorlinks]{hyperref}
\usepackage{verbatim}

\usepackage{marginnote}
\usepackage{xcolor,soul}

\usepackage{stmaryrd}
\usepackage{esint}
\usepackage{graphicx}
\usepackage{bbm}
\usepackage{tikz}
\usetikzlibrary{positioning}

\newtheorem{proposition}{Proposition}[section]
\newtheorem{lemma}{Lemma}[section]
\newtheorem{theorem}{Theorem}[section]

\newtheorem{remark}{Remark}[section]
\numberwithin{equation}{section}
\allowdisplaybreaks
\def\eps{\varepsilon}
\def\i{\mathrm{in}}
\def\d{\mathrm{d}}
\def\R{\mathbb{R}}
\renewcommand{\S}{\mathbb{S}}
\def\f{\mathfrak{f}}
\def\u{\mathfrak{u}}

\def\A{\mathcal{A}}
\def\B{\mathcal{B}}
\def\C{\mathcal{C}}
\def\E{\mathcal{E}}
\def\e{\mathrm{e}}
\def\H{\mathcal{H}}
\def\p{/\!\!/}

\def\L{\mathcal{L}}
\def\N{\mathcal{N}}
\def\K{\mathcal{K}}
\def\Q{\mathcal{Q}}
\def\T{\mathcal{T}}
\def\P{\mathcal{P}}
\def\h{\{\mathcal{I}-\mathcal{P}\}}

\newcommand{\dis}{\displaystyle}

\newcounter{wronumber}

\begin{document}
\title[The Compressible Euler and Acoustic Limit]{The Compressible Euler and Acoustic Limits from quantum Boltzmann Equation with Fermi-Dirac Statistics}

\author[Ning Jiang]{Ning Jiang}
\address[Ning Jiang]{\newline School of Mathematics and Statistics, Wuhan University, Wuhan, 430072, P. R. China}
\email{njiang@whu.edu.cn}

\author[Kai Zhou]{Kai Zhou}
\address[Kai Zhou]
{\newline School of Mathematics and Statistics, Wuhan University, Wuhan, 430072, P. R. China}
\email{kaizhoucm@whu.edu.cn}

\date{\today}
\maketitle

\begin{abstract}
   This paper justifies the compressible Euler and acoustic limits from quantum Boltzmann equation with Fermi-Dirac statistics (briefly, BFD) rigorously. This limit was formally derived in Zakrevskiy's thesis \cite{Zakrevskiy} by moment method. We employ the Hilbert approach. The forms of the classical compressible Euler system and the one derived from BFD are different. Our proof is based on the analysis of the nonlinear implicit transformation of these two forms, and a few novel nonlinear estimates. \\


\end{abstract}
\tableofcontents

\section{Introduction}
\subsection{The BFD equation} In this paper, we investigate the motion of  the quantum particles following the Fermi-Dirac statistics. Namely, the evolution of rarefied gas with quantum effect described by the Boltzmann-Fermi-Dirac (briefly BFD) equation:
\begin{equation}\label{BFD}
  \partial_t F + v\cdot \nabla_x F =\C(F),
\end{equation}
 is studied. Here $0\leqslant F(t,x,v) \leqslant 1$ is the number density of particles located at position $x\in \R^3$, with velocity $v\in \R^3$,  at time $t\geqslant 0$. The collision integral $\C(F)$ takes the form
\begin{equation}\label{collision-op}
  \C(F)=\iint\limits_{\R^3\times\S^2}b(v_*-v,\omega) \Big[\,F^\prime F_*^\prime (1-F)(1-F_*)-FF_*(1-F^\prime)(1- F_*^\prime)\,\Big]\d\omega\d v_*.
\end{equation}
Here $F$, $F_*$, $F'$ and $F'_*$ appearing in the integrand are understood to mean $F(t,x,\cdot)$ evaluated at the velocities $v$, $v_*$, $v'$ and $v'_*$ respectively. We will always use this kind of shorthands throughout the paper. The primed velocities $v'$ and $v'_*$ are velocities after collision of any pair of colliding molecules with velocities $v$ and $v_*$ before. In the process of collision, the system of equations
\begin{equation}\label{Vel-cons}
  \left\{
  \begin{aligned}
  v+v_*=&v'+v'_*,\\
  |v|^2+|v_*|^2=&|v'|^2+|v'_*|^2,
  \end{aligned}
  \right.
\end{equation}
 is that of the only natural conservation laws, namely, the conservation laws of momentum and kinetic energy. The solutions to \eqref{Vel-cons} are represented by
\begin{equation}\label{v-after}
  v^\prime=v+[(v_*-v)\cdot \omega]\omega,\quad v_*^\prime=v_*-[(v_*-v)\cdot \omega]\omega,\quad \omega\in \S^2.
\end{equation}
In \eqref{collision-op}, the collision kernel $b(v_*-v,\omega)$ takes the classical factored form
\begin{equation*}
  b(v_*-v,\omega)=|v_*-v|^\gamma b_0\left(\left|\omega\cdot\frac{v_*-v}{|v_*-v|}\right|\right),\;-3< \gamma \leqslant 1,
\end{equation*}
when the molecular interaction is given by a long-range potential. Here we call it the hard potential kernel for $0\leqslant \gamma \leqslant 1$, while soft potential kernel for $-3<\gamma < 0$. However, throughout the entire paper, we consider the hard sphere collision model, i.e.,
\begin{equation*}
  b(v_*-v,\omega)=|(v_*-v)\cdot \omega|.
\end{equation*}

The BFD equation is a modification of the classical Boltzmann equation, when we account for the quantum effects of molecular encounters. In some other literatures, BFD equation is also frequently called Uehling-Uhlenbeck equation or Nordheim equation. At variance with the derivation of the classical Boltzmann equation \cite{GST-2014,Lanford-1975}, the rigorous derivation of the BFD equation is yet to be established. Since it was presented in Nordheim \cite{Nordheim}, Uehling and Uhlenbeck \cite{UU-1933} by heuristic arguments, works on the derivation of the BFD equation have been taking lots of time.

Mathematically, for the well-posedness of BFD equation, early results were obtained by Dolbeault \cite{Dolbeault} and Lions \cite{Lions}. They studied the global weak existence of solutions in mild or distributional sense for the whole space $\R^3$ under some assumptions on the collision kernel. Lu systematically studied the global existence and stability of weak solutions for general initial data in \cite{Lu-2001JSP,Lu-2008JDE,Lu-Wennberg-2003ARMA}. In particular, he considered the collision kernels of very soft potential with a weak angular cutoff. On the other hand, in our paper \cite{Jiang-Xiong-Zhou-2021}, we prove the global-in-time existence of classical solution for hard sphere potential with an assumption on the smallness of the initial data. 

In the other direction, He-Lu-Pulvirenti \cite{He-Lu-Pulvirenti-2021CMP} mathematically justified the semi-classical limit from homogeneous quantum Boltzmann equation to the homogeneous Fokker-Planck-Landau equation as the Planck constant $\hbar$ tends to zero.

\subsection{The hydrodynamic limits for kinetic equations}
In recent decades, there has a great interest on the connection between kinetic equations and their fluid limits in different scalings. 
In the late 1980s, the so-called BGL program was initialized by Bardos-Golse-Levermore \cite{BGL-1991JSP,BGL-1993CPAM}, which aimed at obtaining rigorously the Leray's solutions to incompressible Navier-Stokes equations from the Boltzmann equation with large initial data in the framework of DiPerna-Lions renormalized solutions  \cite{DiPerna-Lions-1989Annals}. This program was finally completed by Golse and Saint-Raymond \cite{GSR-2004Invention,GSR-2009JMPA}. For incompressible Euler limits, Saint-Raymond made the main contributions in \cite{Saint-Raymond-2003ARMA,Saint-Raymond-2009HPoincare}, which connected the DiPern-Lions solutions to the dissipative solutions. As for the hyperbolic scaling, say, acoustic limit,  Golse-Levermore \cite{Golse-Levermore-2002CPAM} established the acoustic limit with an assumption on the size of the fluctuations. Later on, the restriction has been relaxed to the borderline in \cite{JLM-2010CPDE}.

In the classical solutions framework, there are mainly two approaches: one is based on the spectral analysis of semigroup and another is the nonlinear energy method. For the former approach, Bardos and Ukai \cite{Bardos-Ukai-1991MMMAS} took the first step to obtain the uniform in $\eps\in (0,1)$ (the Knusden number) global existence of classical solution to scaled Boltzmann equation for hard cutoff potential. As a result, they also established the incompressible Navier-Stokes limit from Boltzmann equation. More contributions refer to the literatures \cite{Briant-2015JDE,BMAM-2019AA,Gallagher-Tristani-2020,Nishida-1978CMP}. For the nonlinear energy method developed by Guo \cite{Guo-2003,Guo-2004}, we focus on its application to the compressible Euler limits although there are many results (c.f. \cite{Guo-2006CPAM,JK-2020,Jiang-Xu-Zhao-2018}) of incompressible fluid limits. As an improvement of Caflisch's work \cite{Caflisch-1980CPAM} where the truncated Hilbert expansion with a remainder was used, Guo-Jang-Jiang \cite{GJJ-2010CPAM} made the first step to prove the compressible Euler limit and acoustic limit in optimal scaling from Boltzmann equation by employing the nonlinear energy method. This method relies on our knowledge of the local well-posedness of compressible Euler equation in advance.

Although there are many results of hydrodynamic limits from the classical Boltzmann equation, the research on fluid limits from quantum Boltzmann equation is much limited. The compressible Euler and Navier-Stokes limits and incompressible Navier-Stokes limits from BFD equation for both hard potential and soft potential are formally derived by Zakrevskiy \cite{Zakrevskiy}. We also mention that, Filbet, Hu and Jin \cite{FHJ-2012} introduced a new scheme for quantum Boltzmann equation to capture the Euler limit by numerical computations. However, due to the absence of rigorous mathematical theory on fluid limits from quantum Boltzmann equation, an attempt on rigorous proof of incompressible Navier-Stokes-Fourier limit from BFD equation is made by Jiang-Xiong-Zhou \cite{Jiang-Xiong-Zhou-2021} recently. In \cite{Jiang-Xiong-Zhou-2021} , the authors prove the global existence of classical solutions near equilibrium and in addition, they also obtain the uniform in $\eps$ energy estimates, by which and the structure of equations of conservation laws they establish the incompressible Navier-Stokes-Fourier limit from BFD equation.

\subsection{Formal Analysis of Fluid Limits from BFD Equation}\label{subsec-FormalAnalysis}
Our interests are the compressible Euler limit and the acoustic limit from the BFD equation. Therefore, the dimensionless number $\eps>0$ called the Knudsen number is introduced to rescale the BFD equation \eqref{BFD} as:
\begin{equation}\label{SBFD}
  \partial_t F_\eps +v\cdot \nabla_x F_\eps =\frac{1}{\eps}\C(F_\eps),
\end{equation}
on which the initial data is further imposed:
\begin{equation}\label{In-Data-SBFD}
  F_\eps(0,x,v)=F_\eps^\i(x,v).
\end{equation}

Next we  exhibit some basic properties of BFD equation. In the case of classical Boltzmann equation \cite{Cercignani-1988}, the conservation laws and Boltzmann's $H$ Theorem are statements that bear exclusively on the collision integral. As a analogy, the collision operator $\C$ satisfies the conservation laws.
\begin{proposition}[Proposition 2, Chapter I in \cite{Zakrevskiy}]\label{Prop-C-conser-law}
  For any measurable function
  \begin{equation}\label{F}
    F(t,x,v):\,\R_+\times \R^3 \times \R^3 \longrightarrow \R,\quad 0 \leqslant F \leqslant 1\; a.e.
  \end{equation}
  rapidly decaying on the infinity, there hold
  \begin{equation*}
    \int_{\R^3} \C(F) \d v=0,\quad \int_{\R^3} v\,\C(F) \d v=0,\quad \int_{\R^3} |v|^2\C(F) \d v=0.
  \end{equation*}
\end{proposition}
On the other hand, $H$ Theorem is also hold for the BFD equation (c.f. Proposition 4, Chapter I in \cite{Zakrevskiy}).
\begin{proposition}[$H$-Theorem]
  For every measurable rapidly decaying $F$ with at most polynomially increasing $\left|\ln\frac{1-F}{F}\right|$ satisfying \eqref{F}, the following properties are equivalent:
  \begin{itemize}
    \item[(1)] $\C(F)=0$;
    \item[(2)] The entropy production rate is zero,
  \begin{equation*}
    \int_{\R^3} \C(F)\ln\frac{1-F}{F}\d v=0;
  \end{equation*}
    \item[(3)] $F$ is a (local) Fermi-Dirac distribution,
  \begin{equation}\label{Local-FD}
    F_{\f,\u,T}\equiv\mu(t,x,v)=\frac{1}{1+e^{\frac{|v-\u|^2}{2T}-\f}},
  \end{equation}
  \end{itemize}
\end{proposition}
Here in (3) above, $T=T(t,x)>0$ is the temperature, $\u=\u(t,x)\in \R^3$ is the bulk velocity, and $\f(t,x)/T(t,x)$ is the total chemical potential. When the parameters $(\f,\u,T)$ are constants, the corresponding distributions are called {\it global} Fermi-Dirac distributions.

\subsubsection{Hilbert Expansion}
By utilizing the above conservation laws and $H$-Theorem for BFD equation, Zakrevskiy \cite{Zakrevskiy} formally shows that the solutions to the scaled BFD equation \eqref{SBFD} converge to the local Fermi-Dirac distribution $\mu$ (c.f. Theorem I.5 in \cite{Zakrevskiy}). As a consequence, the following Euler system for compressible fluid
\begin{equation}\label{Com-Euler}
\left\{
  \begin{aligned}
  &\partial_t \rho+\nabla_x\cdot (\rho \u)=0,\\
  &\partial_t(\rho \u)+\nabla_x\cdot(\rho \u\otimes \u)+\nabla_x\mathcal{E}=0,\\
  &\partial_t \left(\frac{1}{2}\rho|\u|^2+\frac{3}{2}\E\right)+\nabla_x\cdot\left[\,\left(\frac{1}{2}\rho |\u|^2+\frac{5}{2}\E\right)\u\,\right]=0,
  \end{aligned}
\right.
\end{equation}
is satisfied by the fluid variables $(\rho,\u,\E)$ which are given by
\begin{equation}\label{rho-u-E}
  \begin{aligned}
  \rho&=\int_{\R^3}\mu\d v=T^{3/2}\int_{\R^3}\frac{1}{1+e^{\frac{|v|^2}{2}-\f}}\d v,\\
  \rho \u&=\int_{\R^3}v\mu\d v=T^{3/2}\u\int_{\R^3}\frac{1}{1+e^{\frac{|v|^2}{2}-\f}}\d v,\\
  \E&=\frac{1}{3}\int_{\R^3}|v-\u|^2\mu\d v=\frac{1}{3}T^{5/2}\int_{\R^3}\frac{|v|^2}{1+e^{\frac{|v|^2}{2}-\f}}\d v.
  \end{aligned}
\end{equation}

Although by using the so-called moment method, the literature \cite{Zakrevskiy} formally explained how to acquire the compressible Euler equations \eqref{Com-Euler} from scaled BFD equation \eqref{SBFD}, such computations are not rigorous proofs. However, the process obtaining the Euler equations would play a role of valuable intuition.

Now we will apply the Hilbert expansion (c.f. \cite{Cercignani-1988}) to the scaled BFD equation \eqref{SBFD}. In the process, the arguments are largely formal but rational, hence highly instructive. To this end, the bilinear operator
\begin{equation}\label{Q-op}
  \Q(F,G)=\iint_{\R^3\times\S^2}|(v_*-v)\cdot\omega|\left(F_*^\prime G^\prime-F_* G\right)\d \omega\d v_*,
\end{equation}
and the trilinear operator
\begin{equation}\label{T-op}
    \T(F,G,H)=\iint_{\R^3\times\S^2}|(v_*-v)\cdot\omega|\Big[F_*^\prime G^\prime\left(H_*+H\right)-F_* G\left(H_*^\prime+H^\prime\right)\Big]\d \omega\d v_*,
\end{equation}
are designated respectively for the sake of simplicity. In terms of the local Fermi-Dirac distribution $\mu$, we additionally denote the local linearized collision operator $\L$ by
\begin{equation}\label{L-op}
  \begin{aligned}
  \L g= & \frac{1}{\sqrt{\mu(1-\mu)}}\iint_{\R^3\times\S^2}|(v_*-v)\cdot\omega|\N\Bigg\{\left(\frac{g}{\sqrt{\mu(1-\mu)}}\right) \\
  &\qquad\qquad+\left(\frac{g}{\sqrt{\mu(1-\mu)}}\right)_*-\left(\frac{g}{\sqrt{\mu(1-\mu)}}\right)_*^\prime-\left(\frac{g}{\sqrt{\mu(1-\mu)}}\right)^\prime\Bigg\}\d \omega\d v_*.
  \end{aligned}
\end{equation}
Here, the conservation laws of momentum and kinetic energy \eqref{Vel-cons} allow us to denote $\N$ appearing above by
\begin{equation}\label{N}
  \N=\mu_*^\prime \mu^\prime(1-\mu_*)(1-\mu)=\mu_*\mu(1-\mu_*^\prime)(1-\mu^\prime).
\end{equation}
By inspection, the operator $\C$ and $\L$ can be equivalently formulated with respect to $\Q$ and $\T$ as
\begin{equation}\label{CL-relations-QT}
  \begin{aligned}
  \C(F)=&\Q(F,F)-\T(F,F,F),\\
  \L \left(\frac{F}{\sqrt{\mu(1-\mu)}}\right)=&-\frac{1}{\sqrt{\mu(1-\mu)}}[\,\Q(\mu,F)+\Q(F,\mu)-\T(F,\mu,\mu)-\T(\mu,F,\mu)-\T(\mu,\mu,F)\,].
  \end{aligned}
\end{equation}

In general, Hilbert expansion means to seek the solutions as the power series of $\eps$:
\begin{equation*}
  F_\eps(t,x,v) \simeq \sum_{n\geqslant 0} \eps^n F_n(t,x,v),
\end{equation*}
Plugging the form above into the scaled BFD equation \eqref{SBFD} and collecting the coefficients of the same orders of $\eps$ on both sides, one can obtain
\begin{equation}\label{Coeff-SameOrder}
  \begin{array}{rrl}
  O(\frac{1}{\eps}):&0=&\Q(F_0,F_0)-\T(F_0,F_0,F_0),\\
  &\phantom{he}\\
  O(1):&\qquad (\partial_t+v\cdot\nabla_x)F_0=&\Q(F_0,F_1)+\Q(F_1,F_0)\\
  &&\quad-\T(F_0,F_1,F_0)-\T(F_1,F_0,F_0)-\T(F_0,F_0,F_1),\\
  &\phantom{he}\\
  O(\eps):& (\partial_t+v\cdot\nabla_x)F_1=&\underline{\Q(F_0,F_2)+\Q(F_2,F_0)}+\Q(F_1,F_1)\\
  &&\quad\underline{-\T(F_0,F_2,F_0)-\T(F_2,F_0,F_0)}-\T(F_1,F_1,F_0)\\
  &&\qquad\quad-\T(F_0,F_1,F_1)-\T(F_1,F_0,F_1)-\underline{\T(F_0,F_0,F_2)},\\
  &\phantom{he}\\
  O(\eps^2):& (\partial_t+v\cdot\nabla_x)F_2=&\underline{\Q(F_0,F_3)+\Q(F_3,F_0)}+\Q(F_1,F_2)+\Q(F_2,F_1)\\ &&\quad\underline{-\T(F_0,F_3,F_0)-\T(F_3,F_0,F_0)}-\T(F_1,F_2,F_0)-\T(F_2,F_1,F_0)\\
  &&\quad-\T(F_0,F_2,F_1)-\T(F_2,F_0,F_1)-\T(F_1,F_1,F_1)-\T(F_0,F_1,F_2)\\
  &&\qquad\quad-\T(F_1,F_0,F_2)-\underline{\T(F_0,F_0,F_3)};\\
  &\cdots\hfill\cdots \hfill \cdots\\
  O(\eps^n):& (\partial_t+v\cdot\nabla_x)F_n=&\Q(F_0,F_{n+1})+\Q(F_{n+1},F_0)-\T(F_0,F_{n+1},F_0)-\T(F_{n+1},F_0,F_0)\\
  & &\dis-\T(F_0,F_0,F_{n+1})+\sum_{\substack{i+j=n+1\\i,j<n+1}}\Q(F_i,F_j) -\sum_{\substack{i+j+k=n+1\\i,j,k<n+1}}\T(F_i,F_j,F_k),\\
  &\cdots\hfill\cdots \hfill \cdots
  \end{array}
\end{equation}

We firstly expect to recover the compressible Euler equations \eqref{Com-Euler} from \eqref{Coeff-SameOrder}. Corresponding to $O(\frac{1}{\eps})$, the $H$-theorem mentioned before combining with \eqref{CL-relations-QT} indicates that the leading order $F_0$ must be a local Fermi-Dirac distribution: $F_0=\mu$. Consequently, $\mu$ is expected to be the limiting distribution function. Then, giving a sight of Proposition \ref{Prop-L-proper} and equalities \eqref{CL-relations-QT}, we derive the compressible Euler system \eqref{Com-Euler} by taking inner products in $L^2(\d v)$ of the second equation in \eqref{Coeff-SameOrder} and $[1,\,v,\,|v|^2]$.

On the other hand, $\mu=F_{\f,\u,T}$ is a function of $\f,\,\u,\,T$. The compressible Euler system \eqref{Com-Euler} thereby can be recasted in terms of $\f,\,\u,\,T$ into
\begin{equation}\label{f-u-T-eq}
  \left\{
  \begin{aligned}
  &\partial_t \f+ \u\cdot\nabla_x \f=0,\\
  &\partial_t\u+ \u\cdot \nabla_x \u + T\nabla_x \f+\frac{5\E\nabla_x T}{2\rho T}=0,\\
  &\partial_t T+ \u\cdot\nabla_x T +\frac{2}{3}T\nabla_x\cdot\u=0.
  \end{aligned}
  \right.
\end{equation}

Once the leading order $\mu$ is determined, we turn to solve the coefficients $F_n\,(n\geqslant 1)$ from \eqref{Coeff-SameOrder}. For this purpose, we define the operator $\P$ as the projection from $L^2(\d v)$ to the null space $Null(\L)$ of $\L$:
\begin{equation}\label{P-op}
  \P:\,L^2(\d v)\, \longrightarrow \,Null(\L)=\mathrm{Span}\Big\{1,\,v-\u,\,|v-\u|^2\Big\}\sqrt{\mu(1-\mu)}.
\end{equation}
By Proposition \ref{Prop-L-proper}, we designate
\begin{equation}\label{PF-n-Exp}
  \P\left(\frac{F_n}{\sqrt{\mu(1-\mu)}}\right)=\left\{\frac{\rho_n}{\varrho}+u_n\cdot\frac{v-\u}{T} +\frac{\theta_n}{6T}\left(\frac{|v-\u|^2}{T}-\frac{3\rho}{\varrho}\right)\right\}\sqrt{\mu(1-\mu)},\quad n\geqslant 1
\end{equation}
for the decomposition for each $n\geqslant 1$:
\begin{equation*}
  \begin{aligned}
  \frac{F_n}{\sqrt{\mu(1-\mu)}}=&\P\left(\frac{F_n}{\sqrt{\mu(1-\mu)}}\right)+\h\left(\frac{F_n}{\sqrt{\mu(1-\mu)}}\right)\\
  =&\left\{\frac{\rho_n}{\varrho}+u_n\cdot\frac{v-\u}{T} +\frac{\theta_n}{6T}\left(\frac{|v-\u|^2}{T}-\frac{3\rho}{\varrho}\right) \right\}\sqrt{\mu(1-\mu)} +\h\left(\frac{F_n}{\sqrt{\mu(1-\mu)}}\right),
  \end{aligned}
\end{equation*}
where $(\rho_n,u_n,\theta_n)(t,x)$ are the coefficients of projection $\P$ and $\varrho(t,x) = \int_{\R^3}\mu(1-\mu)\d v$. From $O(\eps^{n-1})$ in \eqref{Coeff-SameOrder}, we deduce that the microscopic part of $\frac{F_n}{\sqrt{\mu(1-\mu)}}$ is
\begin{equation}\label{I-P-F-n}
    \begin{aligned}
    \h & \left(\frac{F_n}{\sqrt{\mu(1-\mu)}}\right)\\
    =&\L^{-1}\left(-\frac{(\partial_t+v\cdot\nabla_x)F_{n-1}}{\sqrt{\mu(1-\mu)}} +\sum_{\substack{i+j=n\\i,j<n}}\frac{\Q(F_i,F_j)}{\sqrt{\mu(1-\mu)}} -\sum_{\substack{i+j+k=n\\i,j,k<n}}\frac{\T(F_i,F_j,F_k)}{\sqrt{\mu(1-\mu)}}\right),\quad n\geqslant 1.
    \end{aligned}
\end{equation}
However, the hydrodynamic part of $\frac{F_n}{\sqrt{\mu(1-\mu)}}$ is determined by the equation corresponding to $O(\eps^n)$. For each $n \geqslant 1$, applying $\P$ to the equation corresponding to $O(\eps^n)$ divided by $\sqrt{\mu(1-\mu)}$, we compute directly to declare the fluid variables $(\rho_n,u_n,\theta_n)$ subject to
\begin{equation}\label{rho-u-theta-n}
  \left\{
  \begin{aligned}
  &\partial_t\rho_n+\nabla_x\cdot(\rho u_n+\rho_n \u)=0,\\
  &\rho(\partial_t u_n+\u\cdot\nabla_x u_n+ u_n\cdot\nabla_x\u)\\
  &\hspace{3cm}-\frac{\nabla_x\E}{\rho}\rho_n+\nabla_x\left(\frac{\rho T}{\varrho}\rho_n+\frac{\theta_n}{6T}\left(5\E - \frac{3\rho^2 T}{\varrho}\right)\right)=\mathcal{F}_u^\bot(F_n),\\
  &\left(\frac{5\E}{2T} - \frac{3\rho^2 }{2\varrho}\right)\Big(\partial_t\theta_n+\u\cdot\nabla_x \theta_n+\frac{2}{3}(\theta_n\nabla_x\cdot\u+3T\nabla_x\cdot u_n)\\
  &\hspace{9cm}+3u_n\cdot\nabla_x T\Big) =\mathcal{F}_\theta^\bot(F_n),
  \end{aligned}
  \right.
\end{equation}
where
\begin{equation}
  \mathcal{F}_u^\bot(F_n)=-\sum_{j=1}^{3}\partial_{x_j}\int_{\R^3}T \frac{F_n}{\sqrt{\mu(1-\mu)}}\B_{ij}\left(\frac{v-\u}{\sqrt{T}}\right)\d v,
\end{equation}
and
\begin{equation}
  \begin{aligned}
  \mathcal{F}_\theta^\bot(F_n)=&-\sum_{i=1}^{3}\partial_{x_i}\int_{\R^3}2T^{3/2}\frac{F_n}{\sqrt{\mu(1-\mu)}} \A_i\left(\frac{v-\u}{\sqrt{T}}\right)\d v\\
  &\qquad-\sum_{i=1}^{3}\partial_{x_i}\sum_{j=1}^{3}2T\u_j\int_{\R^3}\frac{F_n}{\sqrt{\mu(1-\mu)}} \B_{ij}\left(\frac{v-\u}{\sqrt{T}}\right)\d v\\
  &\qquad\qquad-2\u\cdot\mathcal{F}_u^\bot(F_n).
  \end{aligned}
\end{equation}
Here $\A\in\R^3$ and $\B\in\R^{3\times 3}$ are given by
\begin{equation}\label{AB}
  \begin{aligned}
  &\A_{i}\left(\frac{v-\u}{\sqrt{T}}\right)=\left(\frac{|v-\u|^2}{2T}-\frac{5}{2}\frac{\E}{\rho T}\right)\frac{v_i-\u_i}{\sqrt{T}}\sqrt{\mu(1-\mu)},\\
  &\B_{ij}\left(\frac{v-\u}{\sqrt{T}}\right)=\left(
  \frac{v_i-\u_i}{\sqrt{T}} \frac{v_j-\u_j}{\sqrt{T}} -\frac{|v-\u|^2}{3T}\delta_{ij}\right)\sqrt{\mu(1-\mu)}.
  \end{aligned}
\end{equation}
Finally, the initial data
\begin{equation}\label{In-data-n}
  (\rho_n,u_n,\theta_n)(0,x)=(\rho_n^\i,u_n^\i,\theta_n^\i)(x)\in \R\times\R^3\times\R,\quad n=1,2,\cdots
\end{equation}
are imposed on \eqref{rho-u-theta-n}.

\subsubsection{The acoustic limits from BFD equation} Since the acoustic equations are the linearization of compressible Euler equations \eqref{Com-Euler} about the constant state $(\f,\u,T)=(1,0,1)$, we will formally acquire the acoustic limit from BFD equation \eqref{SBFD} when $F_\eps$ is close to the global Fermi-Dirac distribution
\begin{equation}\label{mu-0}
  F_{1,0,1}\equiv\mu_0=\frac{1}{1+e^{\frac{|v|^2}{2}-1}}.
\end{equation}

To this end, we apply Taylor expansion to $\mu=F_{(\f,\u,T)}$ with
\begin{equation*}
  \f=1+\delta \f_1, \quad \u=\delta u, \quad T=1+\delta \theta,\,\delta>0
\end{equation*}
to get
\begin{equation}\label{mu-Taylor-Exp}
  \mu =\mu_0\Big[\,1+\delta(1-\mu_0)\Big(\f_1+v\cdot u+\frac{|v|^2}{2}\theta\Big)\,\Big]+O(\delta^2),
\end{equation}
Same as the course of \cite{BGL-2000ARMA}, the expansion \eqref{mu-Taylor-Exp} indicates that we should better introduce $G_\eps$ and $g_\eps$ defined by
\begin{equation}\label{Global-Exp}
  F_{\eps}=\mu_0+\delta G_{\eps}\equiv\mu_0+\delta \sqrt{\mu_0(1-\mu_0)}g_\eps.
\end{equation}
So that the scaled BFD equation \eqref{SBFD} can be rewritten as
\begin{equation}\label{g-eps-Eq}
  \begin{aligned}
  &\partial_t g_\eps + v\cdot \nabla_x g_\eps + \frac{1}{\eps}\L_0 g_\eps\\
  &\quad=\frac{\delta}{\eps\sqrt{\mu_0(1-\mu_0)}}\left[\Q(\widetilde{g}_\eps,\widetilde{g}_\eps)-\T(\widetilde{g}_\eps,\widetilde{g}_\eps,\mu_0) -\T(\mu_0,\widetilde{g}_\eps,\widetilde{g}_\eps)-\T(\widetilde{g}_\eps,\mu_0,\widetilde{g}_\eps) \right]\\
  &\qquad\quad-\frac{\delta^2}{\eps\sqrt{\mu_0(1-\mu_0)}}\T(\widetilde{g}_\eps,\widetilde{g}_\eps,\widetilde{g}_\eps),
  \end{aligned}
\end{equation}
where the operator $\L_0$ is defined by \eqref{L-op} with $\mu$ replaced by $\mu_0$ and $G_\eps\equiv\widetilde{g}_\eps=\sqrt{\mu_0(1-\mu_0)}g_\eps$. Then the Proposition \ref{Prop-C-conser-law} immediately implies the local conservation laws of mass, momentum and energy:
\begin{equation*}
\left\{
  \begin{aligned}
  \partial_t\langle  G_\eps\rangle+\nabla_x\cdot\langle  v \, G_\eps\rangle&=0,\\
  \partial_t\langle  v \, G_\eps\rangle+\nabla_x\cdot\langle v\otimes v\, G_\eps\rangle&=0,\\
  \partial_t\langle |v|^2\, G_\eps\rangle+\nabla_x\cdot\langle  v|v|^2 \, G_\eps\rangle&=0,
  \end{aligned}
\right.
\end{equation*}
where for any integrable function $\xi=\xi(v)$, we denote the integral by $\langle \xi \rangle = \int_{\R^3} \xi \d v$.

Now we are ready to give a formal theorem in the style of \cite{BGL-1991JSP} for the acoustic equations formulated as
\begin{equation}\label{AC-Sys}
\left\{
  \begin{aligned}
  \partial_t \sigma +\frac{2}{3}(\overline{K}_A-1)\nabla_x\cdot u&=0,\\
  \partial_t u +\nabla_x(\sigma+\theta)&=0,\\
  \partial_t\theta+\frac{2}{3}\nabla_x\cdot u&=0,
  \end{aligned}
\right.
\end{equation}
by assuming the relevant moments of $g_\eps$ converge in the sense of distributions as $\eps \to 0$, where $\overline{K}_A=\frac{5\overline{\E}}{2\overline{\rho}}$ with $\overline{\E}$ and $\overline{\rho}$ corresponding to $(\f,T)=(1,1)$ in \eqref{rho-u-E}.

\newtheorem*{thm}{Theorem}
\begin{thm}[Formal Acoustic Limit Theorem]
  Let $G_\eps$ be as in \eqref{Global-Exp} and the fluctuations $g_\eps$ given by \eqref{Global-Exp} be a family of weak solutions to equation \eqref{g-eps-Eq}. Assume $\{g_\eps\}_{\eps>0}$ converge in the sense of distributions to a function $g\in L^\infty\left(\d t; L^2(\d v)\right)$ as $\eps \to 0$. Furthermore, assume that the moments
  \begin{equation*}
    \langle G_\eps \rangle,\qquad \langle v G_\eps \rangle,\qquad \langle v\otimes v G_\eps \rangle,\qquad \langle v |v|^2 G_\eps \rangle
  \end{equation*}
  satisfy the local conservation laws and converge in the sense of distributions as $\eps \to 0$ to the corresponding moments
  \begin{equation*}
    \langle G \rangle,\qquad \langle v G \rangle,\qquad \langle v\otimes v G \rangle,\qquad \langle v |v|^2 G \rangle
  \end{equation*}
  with $G=\sqrt{\mu_0(1-\mu_0)}g$ and that
  \begin{equation*}
    \begin{aligned}
    &\L_0 g_\eps \to \L_0 g, \qquad \frac{\delta^2}{\sqrt{\mu_0(1-\mu_0)}}\T(G_\eps,G_\eps,G_\eps) \to 0,\\
    &\frac{\delta}{\sqrt{\mu_0(1-\mu_0)}}\left[\Q(G_\eps,G_\eps)-\T(G_\eps,G_\eps,\mu_0) -\T(\mu_0,G_\eps,G_\eps)-\T(G_\eps,\mu_0,G_\eps)\right]\to 0,
  \end{aligned}
  \end{equation*}
  in the sense of distributions as $\eps \to 0$. Then $g$ has the form
  \begin{equation*}
    g=\left\{\sigma+ u\cdot v+ \left(\frac{|v|^2}{2}-(\overline{K}_A-1)\right)\theta\right\}\sqrt{\mu_0(1-\mu_0)},
  \end{equation*}
  where $(\sigma, u, \theta)$ solve the acoustic equations \eqref{AC-Sys}.
\end{thm}

Generally, the fluctuation amplitude $\delta$ is a function of $\eps$ satisfying
\begin{equation*}
  \delta\to 0,\;\text{ as }\eps\to 0.
\end{equation*}
On can take the form
\begin{equation*}
  \delta=\eps^m, \quad \text{ for any }\,m>0.
\end{equation*}
Based on the Formal Acoustic Limit Theorem, by comparing \eqref{Global-Exp} with the Taylor expansion of $\mu$ \eqref{mu-Taylor-Exp}, one can formally understand that $\{G_\eps\}_{\eps>0}$ converge to
\begin{equation}\label{G}
  G=\left\{\sigma+v\cdot u+\left(\frac{|v|^2}{2}-(\overline{K}_A-1)\right)\theta\right\}\mu_0(1-\mu_0),
\end{equation}
as $\eps\to 0$, while $(\sigma,u,\theta)$ satisfy the acoustic system \eqref{AC-Sys}.

\subsection{Truncated Hilbert Expansion}
Motivated by Caflish's work \cite{Caflisch-1980CPAM} and the improvement work Guo-Jang-Jiang \cite{GJJ-2010CPAM}, in the present paper we expect to employ Hilbert expansion approach to rigorously justify the higher order hydrodynamic corrections to the compressible Euler system from \eqref{SBFD}-\eqref{In-Data-SBFD}, as $\eps\to 0$. However, the more terms are expanded, the more special the solutions are. In this sense, the ansatz of truncating Hilbert expansion is taking the following form
\begin{equation}\label{Trun-HilExp}
  F_\eps=\mu+\sum_{n=1}^{3}\eps^n F_n+\eps^3 F_{R,\eps},
\end{equation}
where the parameters $(\f,\u,T)$ of $\mu$ are characterized by Proposition \ref{Prop-Relation-rhoEfT} once we determine the solution $(\rho,\u,\E)$ to the compressible Euler system \eqref{Com-Euler}. In addition, the coefficients $F_n$ ($n=1,2,3$) are of the forms
\begin{equation*}
  \begin{aligned}
  F_n=&\left[\,\P\left(\frac{F_n}{\sqrt{\mu(1-\mu)}}\right)+\h\left(\frac{F_n}{\sqrt{\mu(1-\mu)}}\right)\,\right]\sqrt{\mu(1-\mu)},\;n=1,2,\\
  F_3=&\h\left(\frac{F_3}{\sqrt{\mu(1-\mu)}}\right)\sqrt{\mu(1-\mu)},
  \end{aligned}
\end{equation*}
and $F_{R,\eps}$ is the remainder, where $\P\left(\tfrac{F_n}{\sqrt{\mu(1-\mu)}}\right)$ and $\h\left(\tfrac{F_n}{\sqrt{\mu(1-\mu)}}\right)$ are given by \eqref{PF-n-Exp} and \eqref{I-P-F-n} respectively. Precisely, the fluid variables $(\rho_n,u_n,\theta_n)$ ($n=1,2$) corresponding to $\P\left(\tfrac{F_n}{\sqrt{\mu(1-\mu)}}\right)$ are solutions of the linear hyperbolic system \eqref{rho-u-theta-n}-\eqref{In-data-n}. Therefore, $F_n$ ($n=1,2$) are determined specifically. For $F_3$, we take the form such that its fluid part vanishes. Therefore $F_3$ is not completely known.

Putting the expansion \eqref{Trun-HilExp} into the scaled BFD equation \eqref{SBFD} to acquire the remainder equation
\begin{equation}\label{Rem-eq}
  \begin{aligned}
    (\partial_t +v\cdot \nabla_x)&F_{R,\eps}+\frac{1}{\eps}\sqrt{\mu(1-\mu)}\L\left(\frac{F_{R,\eps}}{\sqrt{\mu(1-\mu)}}\right)\\
    =&\sum_{i=1}^{3}\eps^{i-1}\Big(\,\Q(F_i,F_{R,\eps})+\Q(F_{R,\eps},F_i)\,\Big) +\eps^2\Q(F_{R,\eps},F_{R,\eps})\\
    &-\sum_{\substack{i+j=1\\i,j\leqslant 3}}^{6} \eps^{i+j-1}\Big(\,\T(F_{R,\eps},F_i,F_j)+\T(F_i,F_{R,\eps},F_j)+\T(F_i,F_j,F_{R,\eps})\,\Big)\\
    &\qquad-\sum_{i=0}^{3}\eps^{i+2}\Big(\,\T(F_i,F_{R,\eps},F_{R,\eps})+\T(F_{R,\eps},F_i,F_{R,\eps})+\T(F_{R,\eps},F_{R,\eps},F_i)\,\Big)\\
    &\qquad\qquad-\eps^5\T(F_{R,\eps},F_{R,\eps},F_{R,\eps})+R_\eps.
  \end{aligned}
\end{equation}
where
\begin{equation*}
  R_\eps=\sum_{\substack{i+j\geqslant 4\\i,j\leqslant 3}}\eps^{i+j-4}\Q(F_i,F_j)-\sum_{\substack{i+j+k\geqslant 4\\i,j,k\leqslant 3}}\eps^{i+j+k-4}\T(F_i,F_j,F_k)-(\partial_t +v\cdot \nabla_x)F_3.
\end{equation*}
Moreover, the remainder equation \eqref{Rem-eq} is given the initial data
\begin{equation}\label{In-data-ReEq}
  F_{R,\eps}(0,x,v)=F_{R,\eps}^\i(x,v).
\end{equation}

\subsection{Main Results}
Our goal is to rigorously justify the compressible Euler limit and the acoustic limit from the BFD equation. Then proving that the remainders $F_{R,\eps}$ of expansion \eqref{Trun-HilExp} will go to zero as $\eps\to 0$ is the key point of this paper. Once this point is showing mathematically, a class of special solutions of BFD equation \eqref{SBFD}-\eqref{In-Data-SBFD} are found according to the same discussions as remarkable work in Caflisch \cite{Caflisch-1980CPAM}.

{\it Notations.} For the simplicity of presentation, some conventions shall be introduced. Firstly, we use the notation $H^s$ to denote the standard Sobolev space $W^{s,2}(\R^3,\d x)$ with corresponding norm $\|\cdot\|_{H^s}$. We use $\langle \cdot,\cdot \rangle_{L_{x,v}^2}$ and $\langle \cdot,\cdot \rangle_{L_v^2}$ to denote the $L_{x,v}^2\equiv L^2(\R^3\times \R^3,\d x \d v)$ and $L_v^2\equiv L^2(\R^3, \d v)$ inner products respectively. Correspondingly, the notations $\|\cdot\|_2$ and $|\cdot|_{L_v^2}$ represent the $L_{x,v}^2$ norm and the $L_v^2$ norm. Similarly, the notations $\|\cdot\|_\infty$ and $|\cdot|_{L_x^\infty}$ represent the $L_{x,v}^\infty\equiv L^\infty(\R^3\times \R^3,\d x \d v)$ norm and $L_x^\infty\equiv L^\infty(\R^3, \d x)$ norm respectively. Particularly, $\|\cdot\|_{\nu(\mu)}$ and $|\cdot|_{\nu(\mu)}$ are designated respectively for the norms
\begin{equation*}
  \|g_1\|_{\nu(\mu)}^2=\iint_{\R^3\times \R^3} g_1^2(x,v) \nu(v) \d x \d v\quad\text{ and } \quad |g_2|_{\nu(\mu)}^2=\int_{\R^3} g_2^2(v) \nu(v)\d v,
\end{equation*}
where $\nu(v)\equiv \nu(\mu)(v)$ is the collision frequency of the form \eqref{nu}. Then a series of integrals are introduced in terms of $\f,\,\u,\,T$ as follows:
\begin{equation}\label{ints}
  \begin{aligned}
  &p_0^0(t,x)=\int_{\R^3}\frac{1}{1+e^{\frac{|v|^2}{2}-\f}}\d v, \quad p_2^0(t,x)=\int_{\R^3}\frac{|v|^2}{1+e^{\frac{|v|^2}{2}-\f}}\d v,\quad p_0^1(t,x)=\int_{\R^3}\frac{e^{\frac{|v|^2}{2}-\f}}{\left(1+e^{\frac{|v|^2}{2}-\f}\right)^2}\d v,\\
  &p_2^1(t,x)=\int_{\R^3}\frac{|v|^2e^{\frac{|v|^2}{2}-\f}}{\left(1+e^{\frac{|v|^2}{2}-\f}\right)^2}\d v,\quad p_4^1(t,x)=\int_{\R^3}\frac{|v|^4e^{\frac{|v|^2}{2}-\f}}{\left(1+e^{\frac{|v|^2}{2}-\f}\right)^2}\d v,\\ &E_0=\int_{\R^3} \mu(1-\mu)\d v, \qquad  E_2=\int_{\R^3}\left|v_1-\u_1\right|^{2}\mu(1-\mu)\d v,\qquad E_{4}=\int_{\R^3}\left|v_1-\u_1\right|^{4}\mu(1-\mu)\d v,\\
  &E_{22}=\int_{\R^3}\left|(v_1-\u_1) (v_2-\u_2)\right|^{2}\mu(1-\mu)\d v, \quad \qquad K_A=\frac{E_4+2E_{22}}{2TE_2},
  \end{aligned}
\end{equation}
while corresponding to $(\f,\u,T)=(1,0,1)$, we use the symbols $\overline{p}_{(\cdot)}^{(\cdot)}$,  $\overline{E}_{(\cdot)}$ and $\overline{K}_A$. In addition, we define $\overline{K}_g=\overline{K}_A-1$.

Finally, we use the symbol $\sim$ to mean that $A \sim B$ is equivalent to
\begin{equation*}
  C_1 B \leqslant A \leqslant C_2 B,
\end{equation*}
for some constant $C_1,\,C_2>0$.

{\it Preliminaries for the Main Results.} In the spirit of Caflisch \cite{Caflisch-1980CPAM}, we introduce a global Fermi-Dirac distribution
\begin{equation}\label{mu-F}
  \mu_F=\frac{1}{1+e^{\frac{|v|^2}{2T_m}-1}},
\end{equation}
where the constant $T_m>0$ satisfies the condition \eqref{T-bd}, i.e.,
\begin{equation*}
  T_m<\min_{(t,x)\in[0,\tau]\times\R^3} T(t,x)\leqslant \max_{(t,x)\in[0,\tau]\times\R^3} T(t,x)<2T_m.
\end{equation*}
Immediately, it follows by inspection that there are some $\frac{1}{2}<\alpha<1$ such that
\begin{equation}\label{mu-mu-F}
  c_1 \mu_F\leqslant \mu\leqslant c_2\mu_F^\alpha.
\end{equation}

On the other hand, given appropriate initial data $(\rho^\i,\u^\i,\E^\i)(x)$ for the compressible Euler system \eqref{Com-Euler}, we can derive the smooth solution $(\rho,\u,\E)$. Furthermore, the parameters $\f$ and $T$ of the local Fermi-Dirac distribution $\mu$ can be determined by Proposition \ref{Prop-Relation-rhoEfT}:
  \begin{equation*}
    \f= \phi\left(\frac{\rho}{\E^{3/5}}\right),\quad T=\left(\frac{\rho}{4\sqrt{2}\pi\Gamma(3/2)\mathcal{F}_{3/2}(\f)}\right)^{2/3} \equiv \psi(\rho,\f),
  \end{equation*}
where $\mathcal{F}_{p}(\f)=\frac{1}{\Gamma(p)}\int_{0}^{\infty}\frac{r^{p-1}}{1+e^{r-\f}}\d r$. Thus we can define the initial local Fermi-Dirac distribution
\begin{equation*}
  \mu^\i(x,v) \equiv \frac{1}{1+e^{\frac{|v-\u^\i|^2}{2T^\i}-\f^\i}},
\end{equation*}
where $\f^\i(x)=\phi\left(\frac{\rho^\i}{(\E^\i)^{3/5}}\right)$ and $T^\i(x)=\psi(\rho^\i,\f^\i)$.

Finally, giving a quick glance of the remainder equation \eqref{Rem-eq}, the left hand side indicates that we would better introduce
\begin{equation}\label{Reminder-f}
  f_{R,\eps}=\frac{F_{R,\eps}}{\sqrt{\mu(1-\mu)}}.
\end{equation}
By \cite{GJJ-2010CPAM}, we further define
\begin{equation}\label{Reminder-h}
  h_{R,\eps}=w(v)\frac{F_{R,\eps}}{\sqrt{\mu_F(1-\mu_F)}},\quad w(v)=(1+|v|)^l,\,l\geqslant 9.
\end{equation}

Now, we state our results on the compressible Euler limit from the scaled BFD equation \eqref{SBFD}.
\begin{theorem}\label{Thm-Euler-limit}
  Let the assumption \eqref{Hypo} hold. Assume the integer $s_1\geqslant 3$, and $s_0\geqslant s_1+2$. Let $(\rho,\u,\E)(t,x)$ be the bounded smooth solution to the compressible Euler system over time interval $[0,\tau]$ constructed in Theorem \ref{Thm-rhouE-Existence}. As a consequence, the parameters $(\f,T)(t,x)$ determined by Proposition \ref{Prop-Relation-rhoEfT} are smooth, which together with $\u$ govern the local Fermi-Dirac distribution $\mu$ and satisfy \eqref{f-bd} and \eqref{T-bd}. Then there is a small $\eps_0>0$ such that if
  \begin{equation}\label{F-eps-in-Exp}
    0 \leqslant F_\eps^\i(x,v)=\mu^\i(x,v)+\sum_{n=1}^{3}\eps^n F_n(0,x,v)+\eps^3 F_{R,\eps}^\i(x,v)\leqslant 1, \quad \forall\,0<\eps<\eps_0,
  \end{equation}
  and
  \begin{equation*}
    \mathbf{E}_R^\i:=\sup_{\eps\in(0,\eps_0)}\left(\|\frac{F_{R,\eps}^\i}{\sqrt{\mu^\i(1-\mu^\i)}}\|_2 +\eps^{\frac{3}{2}}\|(1+|v|)^l\frac{F_{R,\eps}^\i}{\sqrt{\mu_F(1-\mu_F)}}\|_\infty \right)<+\infty,
  \end{equation*}
  then for each $\eps\in(0,\eps_0)$, the scaled BFD equation \eqref{SBFD}-\eqref{In-Data-SBFD} admits a unique solution $F_\eps(t,x,v)$ over the time interval $[0,\tau]$ with the expansion
  \begin{equation*}
    0 \leqslant F_\eps(t,x,v)=\mu(t,x,v)+\sum_{n=1}^{3}\eps^n F_n(t,x,v)+\eps^3 F_{R,\eps}(t,x,v)\leqslant 1.
  \end{equation*}
  Furthermore, the remainders $F_{R,\eps}(t,x,v)$ satisfy the uniform bound
  \begin{equation}\label{Rem-uni-bd}
    \sup_{t\in[0,\tau]}\left(\|\frac{F_{R,\eps}}{\sqrt{\mu(1-\mu)}}\|_2 +\eps^{\frac{3}{2}}\|(1+|v|)^l\frac{F_{R,\eps}}{\sqrt{\mu_F(1-\mu_F)}}\|_\infty \right)\leqslant C(\tau,\mathbf{E}_{0,s_0}^\i,\mathbf{E}_{1,s_1}^\i,\mathbf{E}_{2,s_1}^\i,\mathbf{E}_R^\i ) < +\infty,
  \end{equation}
  where $\mathbf{E}_{0,s_0}^\i:=\|(\rho^\i,\u^\i,\E^\i)\|_{H^{s_0}}$ and $\mathbf{E}_{n,s_1}^\i:=\|(\rho_n^\i,u_n^\i,\theta_n^\i)\|_{H^{s_1}}$,\;n=1,\,2.
\end{theorem}

\begin{remark}
  Since the properties of operator $\L$ given in Proposition \ref{Prop-L-proper} are similar to those of the linearized Boltzmann collision operator and the operators $\Q$ and $\T$ has the estimates \eqref{Q-Es} and \eqref{T-Es}, we declare by the same arguments as Lemma A.1 and Lemma A.2 in \cite{Guo-2006CPAM} that the initial remainder $F_{R,\eps}^\i(x,v)$ in \eqref{F-eps-in-Exp} can be chose so that $F_\eps^\i(x,v) \geqslant 0$. On the other hand, the fact $c\leqslant 1-\mu \leqslant 1$ for some positive constant $c$ implies $F_\eps^\i(x,v) \leqslant 1$.
\end{remark}

\begin{remark}
  Combining with the Proposition \ref{Prop-F-n-Es} where the estimates of $F_n$ (n=1,2,3) are given, Theorem \ref{Thm-Euler-limit} indicates that
  \begin{equation*}
    \sup_{t\in[0,\tau]}\left(\|\frac{F_\eps-\mu}{\sqrt{\mu(1-\mu)}}\|_2 +\|(1+|v|)^l\frac{F_\eps-\mu}{\sqrt{\mu_F(1-\mu_F)}}\|_\infty \right)\leqslant C\eps\to 0,\;\text{ as } \eps\to 0.
  \end{equation*}
  Therefore we have justified the compressible Euler limit from BFD equation.
\end{remark}
Having the results of compressible Euler limit, our main results are established for the acoustic limit from BFD equation:
\begin{theorem}\label{Thm-AC-Limit}
  Let the assumption \eqref{Hypo} hold and $\tau>0$ be any given finite time. Assume
  \begin{equation}\label{In-data-AC}
    (\sigma,u,\theta)(0,x)=(\sigma^\i,u^\i,\theta^\i)(x)\in H^s(\R^3),\quad s\geqslant 4,
  \end{equation}
  is the initial data of the acoustic system \eqref{AC-Sys}. Let $\mu^\delta(0,x,v)$ be the local Fermi-Dirac distribution governed by the initial datum $ 1+\delta(\sigma^\i - \overline{K}_g \theta^\i),\, \delta u^\i$ and $ 1+\delta\theta^\i$:
  \begin{equation*}
    \mu^\delta(0,x,v)\equiv \frac{1}{1+\exp\left\{\frac{|v-\delta u^\i|^2}{2(1+\delta\theta^\i)}-[1+\delta(\sigma^\i - \overline{K}_g \theta^\i)]\right\}},
  \end{equation*}
  where $\overline{K}_g= \overline{K}_A -1$. Let the initial Hilbert expansion be
  \begin{equation*}
    0\leqslant F_\eps^\i(x,v)=\mu^\delta(0,x,v)+\sum_{n=1}^{3}\eps^n F_n(0,x,v)+\eps^3 F_{R,\eps}(0,x,v) \leqslant 1.
  \end{equation*}
  Then there exists an $\eps_0^\prime>0$, a $\delta_0>0$ and a constant $C>0$ such that if
  \begin{equation*}
    \sup_{\eps\in(0,\eps_0^\prime)}\left\{\|F_{R,\eps}(0)\|_2 + \eps^{3/2} \|F_{R,\eps}(0)\|_\infty\right\}< +\infty,
  \end{equation*}
  then for each $0<\eps\leqslant \eps_0^\prime$ and $0<\delta\leqslant\delta_0$, there holds
  \begin{equation}\label{G-eps-G-Es}
    \sup_{t\in[0,\tau]} \|G_\eps-G\|_\infty + \sup_{t\in[0,\tau]} \|G_\eps-G\|_2 \leqslant C\,(\frac{\eps}{\delta}+\delta),
  \end{equation}
  where $G_\eps$ and $G$ are defined by \eqref{Global-Exp} and \eqref{G} and the constant $C$ is depending only on the initial data \eqref{In-data-AC}.
\end{theorem}

\subsection{Sketch of Ideas}
As one of a series of works, the compressible Euler limit from BFD equation is the easiest of all hydrodynamic limits in formal derivation, but also the most difficult one for which giving a rigorous proof. From the mathematical point of view, we expect to utilize the Hilbert expansion with a reminder $F_{R,\eps}$ to certify this limiting process in $L^2-L^\infty$ framework. For this purpose, the interplay between $L^2$ estimates and $L^\infty$ estimates on the remainder $F_{R,\eps}$ play a key role.

In view of the remainder equation \eqref{Rem-eq}, the equation governed by $f_{R,\eps}$ has the form
\begin{equation*}
  \partial_t f_{R,\eps}+v\cdot \nabla_x f_{R,\eps}+\frac{1}{\eps}\L f_{R,\eps}= \text{ some other terms}.
\end{equation*}
The coercivity \eqref{coe-L} of the operator $\L$ provides a kinetic dissipative rate $\frac{\lambda_0}{\eps} \|\h f_{R,\eps}\|_{\nu(\mu)}^2$ for us in $L^2$ estimates. This causes $\|f_{R,\eps}\|_2$ to be bounded by $\eps^2\|h_{R,\eps}\|_\infty$. On the other hand, the equation for $h_{R,\eps}$ is of the form \eqref{Rem-h-eq}:
\begin{equation*}
  \partial_t h_{R,\eps}+ v\cdot\nabla_x h_{R,\eps}+\frac{\nu(\mu)}{\eps}h_{R,\eps}-\frac{1}{\eps} \K_w h_{R,\eps}= \text{ some other terms}.
\end{equation*}
Then, the norm $\eps^{3/2}\|h_{R,\eps}\|_\infty$ can be controlled by $\eps^{3/2}\|h_{R,\eps}(0)\|_\infty+\|f_{R,\eps}\|_2$. However, one can find there exists a disparity of order $\eps^{3/2}$ between $L^2$ norm and $L^\infty$ norm. This is mainly caused by the lack of structure of dissipative rate $\frac{\lambda_0}{\eps} \|\h f_{R,\eps}\|_{\nu(\mu)}^2$ in $L^\infty$ estimates.

Once we acquire the interplay above, we will arrive at the uniform in $\eps$ bound
\begin{equation*}
  \sup_{t\in[0,\tau]}\left\{\|f_{R,\eps}(t)\|_2+\eps^{3/2}\|h_{R,\eps}(t)\|_\infty\right\} \leqslant C,
\end{equation*}
under the assumption on uniform boundedness of $\left\{\|f_{R,\eps}(0)\|_2+\eps^{3/2}\|h_{R,\eps}(0)\|_\infty\right\}$ with respect to small $\eps>0$. In the sequel, we prove that $\eps^3 F_{R,\eps}$ is a infinitesimal as $\eps\to 0$.

As for the acoustic limit, our purpose is to estimate $G_\eps-G$. On one hand, from the formal analysis in subsection \ref{subsec-FormalAnalysis}, we know that as $\eps \to 0$, $F_\eps$ is close to the local Fermi-Dirac distribution $\mu^\delta$ constructed by the smooth solution $(\rho^\delta,\u^\delta,\E^\delta)$ to the compressible Euler equation. Consequently, we first prove that
\begin{equation*}
  F_\eps-\mu^\delta=O(\eps).
\end{equation*}
On the other hand, same as the expansion \eqref{mu-Taylor-Exp}, we then are able to establish that
\begin{equation*}
  \mu^\delta-\mu_0=\delta G+ O(\delta^2).
\end{equation*}
Give a glimpse of \eqref{Global-Exp}, by the two estimates above we conclude our Theorem \ref{Thm-AC-Limit}.

It should be point out that compared with \cite{GJJ-2010CPAM}, the limiting process in the case of BFD equation is more complex. The parameters $(\rho_B,\u_B,T_B)$ of local Maxwellian $M=\frac{\rho_B}{(2\pi T_B)^{3/2}}\exp\{-\frac{|v-\u_B|^2}{2 T_B}\}$ can be determined by the compressible Euler equations deduced from classical Boltzmann equation
\begin{equation*}
\left\{
  \begin{aligned}
  &\partial_t \rho_B+\nabla_x\cdot (\rho_B \u_B)=0,\\
  &\partial_t(\rho_B \u_B)+\nabla_x\cdot(\rho_B \u_B\otimes \u_B)+\nabla_x p=0,\\
  &\partial_t \left(\frac{1}{2}\rho_B|\u_B|^2+\frac{3}{2}\rho_B T_B\right)+\nabla_x\cdot\left[\,\rho_B \u_B\left(\frac{1}{2} |\u_B|^2 +\frac{3}{2}T_B\right)\,\right] + \nabla_x\cdot (p\u_B)=0.
  \end{aligned}
\right. \tag{CEB}
\end{equation*}
However, the compressible Euler equations \eqref{Com-Euler} deduced from BFD equation can only admit $(\rho,\u,\E)$ and the map (Recall \eqref{rho-u-E})
\begin{equation*}
  \Phi \,:\; (\rho,\u,\E)\longrightarrow(\f,\u,T),
\end{equation*}
is implicit, while
\begin{equation*}
  \rho_B=\int_{\R^3} M \d v,\quad \rho_B \u_B=\int_{\R^3}v M \d v,\quad \rho_B T_B=\frac{1}{3}\int_{\R^3}|v-\u_B|^2 M \d v.
\end{equation*}

We tried to solve $\f$, $\u$ and $T$ from the equations \eqref{f-u-T-eq} until we were aware it was impassable since the second equation in \eqref{f-u-T-eq} is indeed a nonlinear integral equation. Nevertheless, Theorem I.2 in \cite{Zakrevskiy} provide us a way to find the relationship between $(\rho,\u,\E)$ and $(\f,\u,T)$. Based on this Theorem, for $(\rho,\u,\E)\in C^k$ ($k \geqslant 1$) we only know $(\f,\u,T)$ is $C^1$ smooth, which is not enough to prove the well-posedness of the linear hyperbolic system \eqref{rho-u-theta-n}. We thereby have to turn to a technical assumption \eqref{Hypo}. We firmly believe that $(\f,\u,T)$ has higher smoothness. The discussion of the hypothesis\eqref{Hypo} will be left elsewhere.

\vskip\baselineskip
The organization of this paper is as follows: the next section contributes to the validity of the compressible Euler equations \eqref{Com-Euler} and the linear hyperbolic system \eqref{rho-u-theta-n}-\eqref{In-data-n} governed by $(\rho_n,u_n,\theta_n)$. Especially, we will obtain the relationship between the solution $(\rho,\u,\E)$ to the compressible Euler equations and the parameters $(\f,\u,T)$ of the local Fermi-Dirac distribution $\mu$. Section \ref{Sec-CE-Lim} focuses on the proof of Theorem \ref{Thm-Euler-limit} via the $L^2-L^\infty$ framework. In detail, the nonlinear energy method is used to acquire the $L^2$ and $L^\infty$ estimates of the reminder $F_{R,\eps}$. The last section establishes the acoustic limit from BFD equation.

\section{Uniform Bounds for the Coefficients of Hilbert Expansion}\label{Sec-Uni-bd}
The material in this section is to derive the uniform bounds for the expansion terms $F_n\;(n=1,2,3)$, and more important, the estimates for the parameters $(\f,\u,T)$ of local Fermi-Dirac distribution $\mu$. For this purpose, $\tfrac{F_n}{\sqrt{\mu(1-\mu)}}$ will be decomposed into the fluid part $\P\tfrac{F_n}{\sqrt{\mu(1-\mu)}}$ and the kinetic part $\h \tfrac{F_n}{\sqrt{\mu(1-\mu)}}$, namely,
\begin{equation*}
  \tfrac{F_n}{\sqrt{\mu(1-\mu)}} = \P\tfrac{F_n}{\sqrt{\mu(1-\mu)}} + \h \tfrac{F_n}{\sqrt{\mu(1-\mu)}},\quad n=1,2, \quad \tfrac{F_3}{\sqrt{\mu(1-\mu)}} = \h \tfrac{F_3}{\sqrt{\mu(1-\mu)}}.
\end{equation*}
For the fluid part, the existence and uniqueness of local-in-time smooth solutions to the compressible Euler system \eqref{Com-Euler} and linear hyperbolic system \eqref{rho-u-theta-n} are studied respectively. For the kinetic part, we give their decay property in $v$.

Now let us forget $(\rho,\u,\E)$ is the solution to the system \eqref{Com-Euler} for a while. Let $\widetilde{\tau}>0$ be any finite time, for any given functions $\f,\,T \in C^1([0,\widetilde{\tau}\,]\times\R^3)$, setting
\begin{equation}\label{Relation-rhoEfT}
  \rho=T^{3/2}\int_{\R^3}\frac{1}{1+e^{\frac{|v|^2}{2}-\f}}\d v, \qquad \mathcal{E}= \frac{1}{3}T^{5/2}\int_{\R^3}\frac{|v|^2}{1+e^{\frac{|v|^2}{2}-\f}}\d v,
\end{equation}
then $(\rho,\,\E)\in C^1([0,\widetilde{\tau}\,]\times\R^3)$ provided $\f$ is bounded, and we have the following proposition:
\begin{proposition}[Theorem I.2, \cite{Zakrevskiy}]\label{Prop-Relation-rhoEfT}
  The ratio $\frac{\rho}{\E^{3/5}}$ depends only on $\f$. Setting $J=\left(\frac{8\sqrt{2}\pi}{3}\right)^{2/5}(\frac{5}{2})^{3/5}\approx 4.65819$, then the map
  \begin{equation*}
    \varphi:\,\f\mapsto \frac{\rho}{\E^{3/5}}
  \end{equation*}
  is $C^\infty(\R;(0,J))$ and strictly monotone. Therefore, there exists $\phi=\varphi^{-1}:\,\frac{\rho}{\E^{3/5}}\mapsto \f$, i.e.
  \begin{equation}\label{f}
    \phi\in C^1((0,J);\mathbb{R}),\qquad \f= \phi\left(\frac{\rho}{\E^{3/5}}\right).
  \end{equation}
  Then the function $T$ is given by
  \begin{equation}\label{T}
    T=\left(\frac{\rho}{4\sqrt{2}\pi\Gamma(3/2)\mathcal{F}_{3/2}(\f)}\right)^{2/3}\stackrel{\triangle}{=}\psi(\rho,\f),
  \end{equation}
  where $\mathcal{F}_{p}(\f)=\frac{1}{\Gamma(p)}\int_{0}^{\infty}\frac{r^{p-1}}{1+e^{r-\f}}\d r$.
\end{proposition}
\begin{remark}
  Noting that the function $\phi$ with argument $\frac{\rho}{\E^{3/5}}$ is $C^1$ smooth at least. We have to give an assumption:
  \begin{equation}\label{Hypo}
    \phi(\cdot) \text{ is sufficiently smooth,} \tag{H}
  \end{equation}
  without proof for technical reason. The discussion of the hypothesis \eqref{Hypo} will be left to another paper.
\end{remark}
\begin{remark}\label{Rmk-Relation-rhoEfT}
  Under the hypothesis \eqref{Hypo}, the significance of proposition above is that: for any integer $k\geqslant 1$, given $(\rho,\,\E)\in C^k([0,\widetilde{\tau}\,]\times\R^3)$ with $0< \frac{\rho}{\E^{3/5}} < J$ for any $(t,x)\in [0,\widetilde{\tau}\,]\times\R^3$, we can determine the unique functions $\f,\,T \in C^k([0,\widetilde{\tau}\,]\times\R^3)$ by relations \eqref{f} and \eqref{T}. Therefore, we acquire the relation \eqref{Relation-rhoEfT}.
\end{remark}

\subsection{The Estimates for Fluid Part}
We would like to consider the compressible Euler system \eqref{Com-Euler}, of which the non-conservative form is
\begin{equation}\label{Com-Euler-NonConser}
  \left\{
  \begin{aligned}
  &\partial_t \rho+\u\cdot\nabla_x\rho+\rho\nabla_x\cdot\u=0,\\
  &\partial_t\u+\u\cdot\nabla_x\u+\frac{\nabla_x\E}{\rho}=0,\\
  &\partial_t \E+\u\cdot\nabla_x \E+\frac{5}{3}\E\nabla_x\cdot\u=0.
  \end{aligned}\right.
\end{equation}
It is the solution $(\rho, \u, \E)$ excluding the vacuum and the non positive temperature that we are looking for. To this end, we consider the solution near the constant state $(\overline{\rho}, \overline{\u}, \overline{\E})$ of \eqref{Com-Euler-NonConser}. More precisely, let
\begin{equation}\label{Cons-State}
  \left\{
  \begin{aligned}
  \overline{\rho}=&\int_{\R^3}\frac{1}{1+e^{\frac{|v|^2}{2}-1}}\d v =4\pi\int_{0}^{+\infty}\frac{r^2}{1+e^{\frac{r^2}{2}-1}} \d r \approx 4\pi\times 1.97477\\
  \overline{\u}=&0\\
  \overline{\E}=&\frac{1}{3}\int_{\R^3}\frac{|v|^2}{1+e^{\frac{|v|^2}{2}-1}}\d v =\frac{4\pi}{3}\int_{0}^{+\infty}\frac{r^4}{1+e^{\frac{r^2}{2}-1}} \d r \approx 4\pi\times 2.509458,
  \end{aligned}
  \right.
\end{equation}
be the constant solution to \eqref{Com-Euler-NonConser}. Going a step further, we can formulate the system \eqref{Com-Euler-NonConser} as
\begin{equation}\label{Symm-form}
  A_0\partial_t U + \sum_{j=1}^{3}A_j\partial_j U = 0,
\end{equation}
where
\begin{equation}\label{A-0}
  U=\left(\begin{array}{c}\E\\ \u \\ \rho\end{array}\right),\quad
  A_0=\left(
  \begin{array}{ccccc}
    \frac{6}{5} & 0 & 0 & 0 & -\frac{\E}{\rho} \\
    0 & \rho\E & 0 & 0 & 0 \\
    0 & 0 & \rho\E & 0 & 0 \\
    0 & 0 & 0 & \rho\E & 0 \\
    -\frac{\E}{\rho} & 0 & 0 & 0 & \frac{5\E^2}{3\rho^2}
  \end{array}
  \right), \quad
  A_1=\left(
  \begin{array}{ccccc}
    \frac{6}{5}\u_1 & \E & 0 & 0 & -\frac{\E}{\rho}\u_1 \\
    \E & \rho\E\u_1 & 0 & 0 & 0 \\
    0 & 0 & \rho\E\u_1 & 0 & 0 \\
    0 & 0 & 0 & \rho\E\u_1 & 0 \\
    -\frac{\E}{\rho}\u_1 & 0 & 0 & 0 & \frac{5\E^2}{3\rho^2}\u_1
  \end{array}
  \right),
\end{equation}
and
\begin{equation*}
  A_2=\left(
  \begin{array}{ccccc}
    \frac{6}{5}\u_2 & 0 & \E & 0 & -\frac{\E}{\rho}\u_2 \\
    0 & \rho\E\u_2 & 0 & 0 & 0 \\
    \E & 0 & \rho\E\u_2 & 0 & 0 \\
    0 & 0 & 0 & \rho\E\u_2 & 0 \\
    -\frac{\E}{\rho}\u_2 & 0 & 0 & 0 & \frac{5\E^2}{3\rho^2}\u_2
  \end{array}
  \right), \qquad
  A_3=\left(
  \begin{array}{ccccc}
    \frac{6}{5}\u_3 & 0 & 0 & \E & -\frac{\E}{\rho}\u_3 \\
    0 & \rho\E\u_3 & 0 & 0 & 0 \\
    0 & 0 & \rho\E\u_3 & 0 & 0 \\
    \E & 0 & 0 & \rho\E\u_3 & 0 \\
    -\frac{\E}{\rho}\u_3 & 0 & 0 & 0 & \frac{5\E^2}{3\rho^2}\u_3
  \end{array}
  \right).
\end{equation*}
We point out that the matrix $A_0$ is positive definite with five real eigenvalues:
\begin{equation*}
     \lambda_1=\lambda_2=\lambda_3=\rho\E,\;\lambda_{4,5}=\frac{3}{5}+\frac{5}{6}\frac{\E^2}{\rho^2} \pm \sqrt{\frac{9}{25}+\frac{25}{36}\frac{\E^4}{\rho^4}}.
   \end{equation*}

Thanks to the Friedrichs’ existence theory of the linear hyperbolic system (c.f. Chapter 2 in \cite{Majda-1984Book}), we can establish that
\begin{theorem}\label{Thm-rhouE-Existence}
  Let $s_0 \geqslant 3$ be an integer, $(\overline{\rho}, \overline{\u}, \overline{\E})$ be as \eqref{Cons-State} and the bound $J$ be as in Proposition \ref{Prop-Relation-rhoEfT}. Assume that the initial data
  \begin{equation}\label{rho-u-E-pertur}
    (\rho,\u,\E)(0,x)=(\rho^\i,\u^\i,\E^\i)(x)=(\widetilde{\rho}^\i,\widetilde{\u}^\i,\widetilde{\E}^\i)(x)+(\overline{\rho}, \overline{\u}, \overline{\E}),\quad (\widetilde{\rho}^\i,\widetilde{\u}^\i,\widetilde{\E}^\i)(x) \in H^{s_0}(\R^3),
  \end{equation}
  satisfies
  \begin{equation}\label{In-bd}
    0<(1-\gamma_0)\overline{\rho} \leqslant \widetilde{\rho}^\i+\overline{\rho} \leqslant (1+\gamma_0)\overline{\rho},\quad 0<(1-\gamma_0)\overline{\E} \leqslant \widetilde{\E}^\i+\overline{\E} \leqslant (1+\gamma_0)\overline{\E},
  \end{equation}
  where $\gamma_0 <1$ small enough is a positive constant such that
  \begin{equation*}
    \frac{1+\gamma_0}{(1-\gamma_0)^{3/5}}\frac{\overline{\rho}}{\overline{\E}^{3/5}}<J.
  \end{equation*}
  Then, there is a finite time $\tau\in (0,+\infty)$ such that, the system \eqref{Symm-form}-\eqref{In-bd} admit a unique smooth solution $(\rho,\u,\E)\in C^1([0,\tau\,]\times\R^3)$. Furthermore,
  \begin{equation*}
    (\rho,\u,\E)(t,x)-(\overline{\rho}, \overline{\u}, \overline{\E}) \in C\left([0,\tau\,];\,H^{s_0}(\R^3)\right)\cap C^1\left([0,\tau\,];\,H^{s_0-1}(\R^3)\right)
  \end{equation*}
  and
  \begin{equation}\label{rho-u-E-Es}
    \begin{array}{c}
      \|(\rho,\u,\E)(t,x)-(\overline{\rho}, \overline{\u}, \overline{\E})\|_{C\left([0,\tau\,];\,H^{s_0}(\R^3)\right)\cap C^1\left([0,\tau\,];\,H^{s_0-1}(\R^3)\right)}\leqslant C_1,\\
      \phantom{h}\\
      (1-\gamma_0)\overline{\rho} \leqslant \rho \leqslant (1+\gamma_0)\overline{\rho},\quad (1-\gamma_0)\overline{\E} \leqslant \E \leqslant (1+\gamma_0)\overline{\E}.
    \end{array}
  \end{equation}
  where $C_1$ is depending only on $\|(\widetilde{\rho}^\i,\widetilde{\u}^\i,\widetilde{\E}^\i)\|_{H^{s_0}(\R^3)}$.
\end{theorem}

\begin{remark}
  It should be declared that the solution $(\rho,\u,\E)$ to the system \eqref{Symm-form}-\eqref{In-bd} constructed above satisfies
  \begin{equation*}
    \frac{\E}{\rho}\geqslant \frac{(1-\gamma_0)}{(1+\gamma_0)}\frac{\overline{\E}}{\overline{\rho}} > \frac{(1-\gamma_0)}{(1+\gamma_0)},
  \end{equation*}
  which guarantees that there is a constant $\widetilde{\lambda}>0$ such that the symmetrizer $A_0$ of the hyperbolic system \eqref{Symm-form} satisfies $A_0>\widetilde{\lambda} I$ for any $(t,x)\in [0,\tau\,]\times\R^3$. One can take $\gamma_0=\frac{1}{100}$ for the sake of intuitiveness.
\end{remark}

According to Remark \ref{Rmk-Relation-rhoEfT}, we can determine $\f,\,T\in  C^1([0,\tau\,]\times\R^3)$ by relations \eqref{f} and \eqref{T}. Furthermore, we can conclude that
\begin{proposition}\label{Prop-fT-Es}
  We have the following assertions of $\f$, $T$ and $\nabla_x T$:
  \begin{itemize}
    \item[(1)] There are some constants $\f_m$ and $\f_M$ such that for any $(t,x)\in [0,\tau\,]\times \R^3$
    \begin{equation}\label{f-bd}
      \f_m \leqslant \f(t,x) \leqslant \f_M;
    \end{equation}
    \item[(2)] There are some constants $T_m>0$ such that for any $(t,x)\in [0,\tau\,]\times \R^3$
    \begin{equation}\label{T-bd}
      T_m < \min_{(t,x)\in[0,\tau]\times\R^3} T(t,x)\leqslant \max_{(t,x)\in[0,\tau]\times\R^3} T(t,x)< 2T_m;
    \end{equation}
    \item[(3)] It holds that for some constant $C>0$
    \begin{equation}\label{T-bd-1}
      \|\nabla_x T\|_{L^\infty([0,\tau\,];L^2(\R^3))}\leqslant C,\quad \|\nabla_x T\|_{L^\infty([0,\tau\,];L^\infty(\R^3))}\leqslant C.
    \end{equation}
  \end{itemize}
\end{proposition}
\begin{proof}
  The Proposition \ref{Prop-Relation-rhoEfT} gives us that
  \begin{equation*}
    \frac{\rho}{\E^{3/5}}=\varphi\left(\f\right)\searrow 0,\quad  \text{as}\quad \f\to -\infty,\qquad \frac{\rho}{\E^{3/5}}=\varphi\left(\f\right)\nearrow J,\quad  \text{as}\quad \f\to +\infty.
  \end{equation*}
  Thus, the statements $(1)$ and $(2)$ follow respectively by
  \begin{equation*}
    \phi(0)<\phi\left(\frac{(1-\gamma_0)}{(1+\gamma_0)^{3/5}}\frac{\overline{\rho}}{\overline{\E}^{3/5}}\approx 3.07961\right)=\f_m \leqslant \f(t,x) \leqslant \f_M= \phi\left(\frac{(1+\gamma_0)}{(1-\gamma_0)^{3/5}}\frac{\overline{\rho}}{\overline{\E}^{3/5}}\approx  3.17975\right)<\phi(J),
  \end{equation*}
  and
  \begin{equation*}
    0<\left(\frac{(1-\gamma_0)\overline{\rho}}{4\sqrt{2}\pi\Gamma(3/2)\mathcal{F}_{3/2}(\f_M)}\right)^{2/3} =T_m \leqslant  T(t,x) \leqslant \left(\frac{(1+\gamma_0)\overline{\rho}}{4\sqrt{2}\pi\Gamma(3/2)\mathcal{F}_{3/2}(\f_m)}\right)^{2/3}<2T_m.
  \end{equation*}
  \begin{equation*}
    \Leftrightarrow\int_{0}^{+\infty}\frac{r^{1/2}}{1+e^{r-\f_M}}\d r\leqslant \frac{\sqrt{8}(1-\gamma_0)}{1+\gamma_0}\int_{0}^{+\infty}\frac{r^{1/2}}{1+e^{r-\f_m}}\d r.
  \end{equation*}
  The only non trivial point is the assertion $(3)$. Computing directly gives us
  \begin{equation*}
    \nabla_x \rho=(T^{3/2} p_0^1)\nabla_x \f+3\rho\times\frac{\nabla_x T}{2T},\quad \nabla_x \E = \left(\frac{T^{5/2}}{3} p_2^1\right)\nabla_x \f+5\E\times\frac{\nabla_x T}{2T}.
  \end{equation*}
  Here $p_0^1(t,x)$ and $p_2^1(t,x)$ are as in \eqref{ints} and satisfy
  \begin{equation*}
    \begin{aligned}
    &\int_{\R^3}\frac{e^{\f_m}}{(1+e^{\f_m})(1+e^{\f_M})}e^{-\frac{|v|^2}{2}} \d v \leqslant p_0^1(t,x) \leqslant \int_{\R^3} e^{\f_M}e^{-\frac{|v|^2}{2}} \d v ,\\
    &\int_{\R^3}\frac{e^{\f_m}}{(1+e^{\f_m})(1+e^{\f_M})}|v|^2 e^{-\frac{|v|^2}{2}} \d v \leqslant p_2^1(t,x) \leqslant \int_{\R^3} e^{\f_M}|v|^2 e^{-\frac{|v|^2}{2}} \d v.
    \end{aligned}
  \end{equation*}
  It follows immediately that
  \begin{equation*}
    \nabla_x T= \frac{2T}{3 p_2^1\rho T -15p_0^1 \E}(p_2^1 T \nabla_x\rho-3 p_0^1 \nabla_x\E),
  \end{equation*}
  where the notations appearing here is given by \eqref{ints} and
  \begin{equation*}
    3 p_2^1\rho T -15p_0^1 \E = T^{5/2}(\,3p_0^0p_2^1-5p_2^0p_0^1\,)=T^{5/2}(\,(p_2^1)^2-p_4^1p_0^1\,)<0,
  \end{equation*}
  the last step is due to the Cauchy-Schwarz's inequality. Therefore, the estimate \eqref{rho-u-E-Es} implies for some constant $C>0$
  \begin{equation*}
    \|\nabla_x T\|_{L^\infty([0,\tau\,];L^2(\R^3))}\leqslant C,\quad \|\nabla_x T\|_{L^\infty([0,\tau\,];L^\infty(\R^3))}\leqslant C.
  \end{equation*}
\end{proof}

Now we turn to prove the local-in-time existence of the hyperbolic system \eqref{rho-u-theta-n} for $n=1,2$, i.e.:
\begin{equation*}
  \left\{
  \begin{aligned}
  &\partial_t\rho_n+\nabla_x\cdot(\rho u_n+\rho_n \u)=0\\
  &\rho(\partial_t u_n+\u\cdot\nabla_x u_n+ u_n\cdot\nabla_x\u)\\
  &\hspace{3cm}-\frac{\nabla_x\E}{\rho}\rho_n+\nabla_x\left(\frac{\rho T}{\varrho}\rho_n+\frac{\theta_n}{6T}\left(5\E - \frac{3\rho^2 T}{\varrho}\right)\right)=\mathcal{F}_u^\bot(F_n)\\
  &\left(\frac{5\E}{2T} - \frac{3\rho^2 }{2\varrho}\right)\Big(\partial_t\theta_n+\u\cdot\nabla_x \theta_n+\frac{2}{3}(\theta_n\nabla_x\cdot\u+3T\nabla_x\cdot u_n)\\
  &\hspace{9cm}+3u_n\cdot\nabla_x T\Big) =\mathcal{F}_\theta^\bot(F_n).
  \end{aligned}
  \right.
\end{equation*}
Firstly, we need to introduce some notations for the sake of simplicity. We define for the multi-index $\beta=(\beta_0,\beta_1,\beta_2,\beta_3)$ that
\begin{equation*}
  \partial^\beta = \partial_t^{\beta_0}\partial_{x_1}^{\beta_1}\partial_{x_2}^{\beta_2}\partial_{x_3}^{\beta_3},
\end{equation*}
and define the functional space $\H^k$ by
\begin{equation*}
  \H^k = \left\{\,f(t,x) \in L^2(\d t\d x)\,: \, \sum_{|\beta|\leqslant k}\left\|\partial^\beta f(t)\right\|_{L^2(\d x)} <+\infty\right\}.
\end{equation*}
The functional space above is endowed with norm $\left\| f(t)\right\|_{\H^k} = \sum_{|\beta|\leqslant k}\left\|\partial^\beta f(t)\right\|_{L^2(\d x)}$. Then we are going to establish that:
\begin{theorem}\label{Thm-rho-u-theta-n}
  Let $s_1\geqslant 3$ be an integer and $s_0 \geqslant s_1+2$. Let $(\rho,\u,\E)$ be the bounded smooth solution of compressible Euler equations \eqref{Com-Euler-NonConser} constructed in Theorem \ref{Thm-rhouE-Existence} over $[0,\tau\,]$. As a consequence, $\f$ and $T$ determined by \eqref{f} and \eqref{T} are at least $C^{s_1}$ smooth under the assumption \eqref{Hypo}. For the system \eqref{rho-u-theta-n}-\eqref{In-data-n} with $n=1,2$, we assume that
  \begin{equation*}
    \mathbb{E}_{n,s_1}=\|(\rho_n^\i,u_n^\i,\theta_n^\i)\|_{\H^{s_1}}^2+ \sup_{t\in(0,\tau)}\|(\mathcal{F}_u^\bot(F_n),\mathcal{F}_\theta^\bot(F_n))(t,\cdot)\|_{\H^{s_1+2}}^2< +\infty.
  \end{equation*}
  Then there exists a unique smooth solution $(\rho_n,u_n,\theta_n)$ to \eqref{rho-u-theta-n}-\eqref{In-data-n} over $t\in [0,\tau\,]$ for each $n=1,2$ such that
  \begin{equation}\label{rho-u-theta-n-Es}
    \sup_{t\in(0,\tau)}\|(\rho_n,u_n,\theta_n)\|_{\H^{s_1}}^2\leqslant C(\tau,\mathbf{E}_{0,s_1+2})\mathbb{E}_{n,s_1},
  \end{equation}
  where $\dis\mathbf{E}_{0,s_1+2}=\sup_{t\in(0,\tau)}\|(\rho-\overline{\rho},\u,\E-\overline{\E})\|_{H^{s_1+2}}$.
\end{theorem}
\begin{proof}
  In what follows, for $n=1,\,2$, the notations
\begin{equation*}
  H \equiv \frac{\rho T}{\varrho},\qquad G \equiv \frac{5\E}{2T} - \frac{3\rho^2 }{2\varrho},\qquad p_n\equiv\frac{\rho T}{\varrho}\rho_n+\frac{G}{3}\theta_n,
\end{equation*}
are used for convenience. Then there holds
\begin{equation*}
  \partial_t G +\nabla\cdot(G\u)=0,\qquad (\partial_t + \u\cdot\nabla)H +\frac{2}{3}H\nabla\cdot\u=0.
\end{equation*}
With the notations and identities above, the hyperbolic system \eqref{rho-u-theta-n} can be written as
\begin{equation*}
  \left\{
  \begin{aligned}
  &\partial_t p_n + \u\cdot \nabla_x p_n + \frac{5}{3}\E\nabla\cdot u_n + \frac{5}{3}(\nabla\cdot\u) p_n + u_n\cdot\nabla\E=\frac{1}{3}\mathcal{F}_\theta^\bot(F_n)\\
  &\rho \partial_t u_n + \rho \u\cdot\nabla u_n +\nabla p_n - \frac{\nabla \E}{\rho^2T/\varrho}p_n + \rho u_n\cdot\nabla \u - \frac{\nabla \E}{3\rho^2T/(G\varrho)}\theta_n = \mathcal{F}_u^\bot(F_n)\\
  & G \partial_t \theta_n +G \u\cdot\nabla \theta_n +2GT\nabla\cdot u_n +3Gu_n\cdot\nabla T + \frac{2}{3}G (\nabla \cdot \u)\theta_n = \mathcal{F}_\theta^\bot(F_n).
  \end{aligned}
  \right.
\end{equation*}
Then a symmetric hyperbolic equation is obtained:
\begin{equation}\label{hyper-Eq-n}
  A_0^n \partial_t U_n + \sum_{j=1}^{3}A_j^n \partial_j U_n + A_4^n U_n = \mathcal{F}_n,
\end{equation}
where
\begin{equation*}
  U_n=\left(\begin{array}{c}p_n \\ u_n \\ \theta_n \end{array}\right),\quad A_0^n=\left(\begin{array}{ccc} \frac{9}{5} & 0 & -\frac{\E}{T} \\ 0 & \rho\E\,\mathbb{I} & 0\\ -\frac{\E}{T} & 0 & \frac{5\E^2}{6T^2} \end{array}\right),\quad A_j^n=\left(\begin{array}{ccc} \frac{9}{5}u_j & \E e_j & -\frac{\E}{T}u_j \\ \E e_j^\mathrm{Tr} & \rho\E u_j\mathbb{I} & 0 \\ -\frac{\E}{T}u_j & 0 & \frac{5\E^2}{6T^2}u_j \end{array}\right)
\end{equation*}
and the matrix $A_4^n$ and column vector $\mathcal{F}_n$ can be written down easily. Here the symbol $\mathbb{I}$ denotes the $3 \times 3$ identity matrix and $e^j,\,j=1,2,3$ the standard (row) base vectors in $\mathbb{R}^3$, the symbol $(\cdot)^\mathrm{Tr}$ means the transpose of row vectors.

Obviously, the matrix $A_0^n$ is positive definite.  For the existence of local-in-time smooth solutions, we refer to the standard Friedrichs' theory (c.f. Chapter 2 in \cite{Majda-1984Book}).

Now we devote ourselves to lifespan of the equations \eqref{rho-u-theta-n} and the estimate \eqref{rho-u-theta-n-Es}. It suffices to establish a priori energy estimates of \eqref{hyper-Eq-n}. Taking $\partial^\beta$ of \eqref{hyper-Eq-n} to get
\begin{equation*}
  \begin{aligned}
  A_0^n \partial_t \partial^\beta U_n + \sum_{j=1}^{3}A_j^n \partial_j \partial^\beta U_n + A_4^n \partial^\beta U_n = & \partial^\beta\mathcal{F}_n - \sum_{|\alpha|\leqslant |\beta|-1}\partial^{\beta-\alpha}A_0^n \partial_t\partial^\alpha U_n\\
  &\qquad- \sum_{j=1}^{3}\sum_{|\alpha|\leqslant |\beta|-1}\partial^{\beta-\alpha}A_j^n \partial_j \partial^\alpha U_n - \sum_{|\alpha|\leqslant |\beta|-1}\partial^{\beta-\alpha}A_4^n \partial^\alpha U_n.
  \end{aligned}
\end{equation*}
Then taking inner product with $\partial^\beta U_n$ in $L^2(\d x)$ and integrating over $[0,t]$, we obtain
\begin{equation*}
  \begin{aligned}
  &\left(A_0^n \partial^\beta U_n(t),\partial^\beta U_n(t)\right)_{L^2(\d x)}-\left(A_0^n \partial^\beta U_n(0),\partial^\beta U_n(0)\right)_{L^2(\d x)}\\
  =&\int_{0}^{t}\left(\left[\partial_t A_0^n+\sum_{j=1}^{3}\partial_j A_j^n-(A_4^n+(A_4^n)^{Tr})\right]\partial^\beta U_n(s),\partial^\beta U_n(s)\right)_{L^2(\d x)}\d s+\int_{0}^{t}2\left(\partial^\beta\mathcal{F}_n,\partial^\beta U_n\right)_{L^2(\d x)}\d s\\
  &- \int_{0}^{t}2\sum_{|\alpha|\leqslant |\beta|-1}\left(\partial^{\beta-\alpha}A_0^n \partial_t\partial^\alpha U_n(s)+\sum_{j=1}^{3}\partial^{\beta-\alpha}A_j^n \partial_j \partial^\alpha U_n(s)+\partial^{\beta-\alpha}A_4^n \partial^\alpha U_n(s),\partial^\beta U_n(s)\right)_{L^2(\d x)}\d s.
  \end{aligned}
\end{equation*}
Since $A_0^n$ is positive definite, summing up over $|\beta|\leqslant s_1$ above, we arrive at
\begin{equation*}
  \left\|U_n(t)\right\|_{\H^{s_1}}^2 \leqslant  \left\|U_n(0)\right\|_{\H^{s_1}}^2 + C(\mathbf{E}_{0,s_1+2}) \int_{0}^{t}\left\|U_n(s)\right\|_{\H^{s_1}}^2 \d s + C\int_{0}^{t}\left\|\mathcal{F}_n(s)\right\|_{\H^{s_1+2}}^2 \d s.
\end{equation*}
A simple Gr\"{o}nwall's in equality implies that
\begin{equation*}
  \left\|U_n(t)\right\|_{\H^{s_1}}^2 \leqslant \left\{\left\|U_n(0)\right\|_{\H^{s_1}}^2+ C\int_{0}^{t}\left\|\mathcal{F}_n(s)\right\|_{\H^{s_1+2}}^2 \d s\right\}\left\{1+C(\mathbf{E}_{0,s_1+2})te^{C(\mathbf{E}_{0,s_1+2})t}\right\},
\end{equation*}
It follows that
\begin{equation*}
  \sup_{t\in(0,\tau)}\left\|U_n(t)\right\|_{\H^{s_1}}^2 \leqslant C(\tau,\mathbf{E}_{0,s_1+2})\mathbb{E}_{n,s_1},
\end{equation*}
which is exactly \eqref{rho-u-theta-n-Es}.

\end{proof}

\subsection{The Properties of Linearized Collision Operator $\L$}
Recall the definition \eqref{L-op} of the linearized BFD collision operator $\L$, we can write
\begin{equation*}
  \L f \equiv \nu(\mu) f -\K f= \nu(\mu) f -(\K_2-\K_1) f.
\end{equation*}
Here since
\begin{equation}\label{mu-Es}
  \frac{e^{\f_m}}{1+e^{\f_m}}e^{-\frac{|v-\u|^2}{2T}}\leqslant \mu\leqslant e^{\f_M}e^{-\frac{|v-\u|^2}{2T}}, \qquad \frac{1}{1+e^{\f_M}}\leqslant 1-\mu\leqslant 1,
\end{equation}
the collision frequency $\nu(\mu)$ satisfies
\begin{equation}\label{nu}
  \begin{aligned}
  \nu\equiv\nu(\mu)=&\iint_{\R^3\times \S^2}|(v_*-v)\cdot\omega|\frac{\N}{\mu(1-\mu)}\d \omega \d v_*\\
  \sim & \iint_{\R^3\times\S^2}|(v-v_*)\cdot\omega| e^{-\frac{|v_*-\u|^2}{2T}} \d \omega \d v_*,
  \end{aligned}
\end{equation}
and $\K = \K_2-\K_1$ is given by
\begin{equation}\label{K-1K-2}
  \begin{aligned}
  \K_1 f = & \iint_{\R^3\times\S^2}|(v_*-v)\cdot\omega|\frac{\N}{\sqrt{\mu(1-\mu)}} \left(\frac{f}{\sqrt{\mu(1-\mu)}}\right)_*\d \omega\d v_*,\\
  \K_2 f = &\iint_{\R^3\times\S^2}|(v_*-v)\cdot\omega|\frac{\N}{\sqrt{\mu(1-\mu)}} \left\{\left(\frac{f}{\sqrt{\mu(1-\mu)}}\right)_*^\prime+\left(\frac{f}{\sqrt{\mu(1-\mu)}}\right)^\prime\right\}\d \omega\d v_*\\
  =&2\iint_{\R^3\times\S^2}|(v_*-v)\cdot\omega|\frac{\N}{\sqrt{\mu\mu'}\sqrt{(1-\mu)(1-\mu')}} f'\d \omega\d v_*,
  \end{aligned}
\end{equation}
where the Proposition 2.2 in \cite{Jiang-Xiong-Zhou-2021} is applied to $\K_2$. Then we are able to acquire that:
\begin{proposition}\label{Prop-L-proper}
    For the linearized collision operator $\L$, we conclude that
  \begin{itemize}
    \item[1)] For the collision frequency $\nu$ defined by \eqref{nu}, we have
        \begin{equation}\label{nu-Es}
            C_1(1+|v|)\leqslant \nu \leqslant C_2(1+|v|).
        \end{equation}
    \item[2)] The linearized collision operator $\L$ is positive, symmetric and the null space of $\L$ in $L^2(\d v)$ is
        \begin{equation}\label{null-L}
            Null(\L)=\mathrm{Span}\Big\{1,\,v-\u,\,|v-\u|^2\Big\}\sqrt{\mu(1-\mu)}.
        \end{equation}
        Consequently, there is a constant $\lambda_0>0$ such that
        \begin{equation}\label{coe-L}
            \langle \L f,\,f \rangle_{L_v^2}\geqslant \lambda_0 |\h f|_{\nu(\mu)}^2,\quad \forall\, f \in L_v^2,
        \end{equation}
        where $\P$ is a projection defined by \eqref{P-op}.
    \item[3)] For any $0<q<1$, there is a constant $C>0$ such that
        \begin{equation*}
          \mu^{-q/2}|\K f| \leqslant C|\mu^{-q/2}f|_{L_v^2}.
        \end{equation*}
    \item[4)] Let $0<q<1$ be as in the last statement. For any $f,\,g \in Null^\bot(\L)$ satisfying
    \begin{equation*}
      \L f =g \quad \text{ and }\quad (1+|v|)^2\{\mu^{-q/2} g\}(v) \in L_v^\infty,
    \end{equation*}
    there is a constant $C>0$ such that
  \begin{equation}\label{L-inv-decay-1}
    |f| \leqslant C\left|(1+|v|)^2\mu^{-q/2}g\right|_{L_v^\infty}\mu^{q/2}.
  \end{equation}
  As a consequence of the estimate above, for any $0<q<\widetilde{q}<1$ we can establish:
    \begin{equation}\label{L-inv-decay-2}
        |D_w f| \leqslant C\Big|(D\f,D\u,D T)\Big|\left\{\left|(1+|v|)^2\mu^{-\widetilde{q}/2}D g\right|_{L_v^\infty}+ \left|(1+|v|)^2\mu^{-\widetilde{q}/2}g\right|_{L_v^\infty}\right\}\mu^{q/2},
    \end{equation}
    where $D=\partial_t$ for $D_w= \partial_t$, and $D=\nabla_x$ for $D_w=v\cdot \nabla_x$.
  \end{itemize}
\end{proposition}

\begin{proof}
  Throughout the proof, we use the estimates \eqref{mu-Es} frequently. The proof of first two statements refer to \cite{Jiang-Xiong-Zhou-2021}. Now in the spirit of Section V in \cite{Grad-1962}, we prove the last assertion provided the third statement holds.

  Let $0<q<1$ be fixed and $\nu_0>0$ be the infimum of $\nu$. Then from $\L f =g$, we deduce that
  \begin{equation*}
    \mu^{-q/2}|f|=\frac{\mu^{-q/2}}{\nu}\left|g+\K f\right|\leqslant \frac{1}{\nu_0}\left(\left|\mu^{-q/2}g\right|+\mu^{-q/2}|\K f|\right)\leqslant \frac{1}{\nu_0}\left(\left|\mu^{-q/2}g\right|+C|\mu^{-q/2}f|_{L_v^2}\right).
  \end{equation*}
  Note for $f \in Null^\bot(\L)$, we have (see Appendix for more details)
  \begin{equation*}
    \left\langle\mu^{-q/2}\L f, \mu^{-q/2} f \right\rangle_{L_v^2} \geqslant \frac{1}{2} |\mu^{-q/2} f|_{\nu(\mu)}^2 -C|f|_{\nu(\mu)}^2,
  \end{equation*}
  by \eqref{coe-L} and the estimate above we obtain
  \begin{equation*}
    \begin{aligned}
    |\mu^{-q/2} f|_{L_v^2}^2 \leqslant &C\left\langle\mu^{-q/2}\L f, \mu^{-q/2} f \right\rangle_{L_v^2}+C\left\langle \L f, f \right\rangle_{L_v^2}\\
    \leqslant & C\int_{\R^3}\left|\{\mu^{-q/2} f\}(v)\right| \left|(1+|v|)^2\{\mu^{-q/2}g\}(v)\right| \frac{1}{(1+|v|)^2} \d v\\
    \leqslant & C \left|(1+|v|)^2 \{\mu^{-q/2}g\}\right|_{L_v^\infty} \left|\mu^{-q/2} f\right|_{L_v^2} \left\{\int_{\R^3} \frac{1}{(1+|v|)^4} \d v\right\}^{1/2}\\
    \leqslant  & C \left|(1+|v|)^2 \{\mu^{-q/2}g\}\right|_{L_v^\infty} \left|\mu^{-q/2} f\right|_{L_v^2}.
    \end{aligned}
  \end{equation*}
  It follows that $|\mu^{-q/2} f|_{L_v^2} \leqslant C \left|(1+|v|)^2 \{\mu^{-q/2}g\}\right|_{L_v^\infty}$ and hence
  \begin{equation*}
    |f| \leqslant \frac{1}{\nu_0}\left(\left|\mu^{-q/2}g\right|+C\left|(1+|v|)^2\{\mu^{-q/2}g\}\right|_{L_v^\infty}\right)\mu^{q/2} \leqslant C \left|(1+|v|)^2 \{\mu^{-q/2}g\}\right|_{L_v^\infty}\mu^{q/2},
  \end{equation*}
  which is the estimate \eqref{L-inv-decay-1}.

Next we turn to the estimate of $D_w f$ for $D_w = v\cdot\nabla_x$ and $D=\nabla_x$. To this end, we define
\begin{equation*}
  \overline{\K}_1 f \equiv D (\K_1 f) - \K_1 (D f),\quad \overline{\K_2} f \equiv D (\K_2 f) - \K_2 (D f),
\end{equation*}
then
\begin{equation*}
  \overline{\K} f \equiv \overline{\K}_2 f-\overline{\K}_1 f  = D (\K f) - \K (D f).
\end{equation*}
Taking $D$ on both sides of $\L f =g$ to yield
\begin{equation}\label{L-D-f}
  \L (D f) = D g -\left[(D \nu)f-\overline{\K} f\right].
\end{equation}
In order to evaluate $Df$, we decompose it into
\begin{equation*}
  Df = \P(Df) + \h(Df).
\end{equation*}
For the fluid part $\P(Df) $, since $f \in Null^\bot(\L)$, we introduce
\begin{equation*}
    \begin{aligned}
     X= & \left\langle D f, \sqrt{\mu(1-\mu)} \right\rangle_{L_v^2} = -\left\langle f, D\sqrt{\mu(1-\mu)}\right\rangle_{L_v^2},\\
     Y= & \left\langle D f, (v-\u)\sqrt{\mu(1-\mu)} \right\rangle_{L_v^2} = -\left\langle f, D\left\{(v-\u)\sqrt{\mu(1-\mu)}\right\}\right\rangle_{L_v^2},\\
     Z= & \left\langle D f, |v-\u|^2\sqrt{\mu(1-\mu)} \right\rangle_{L_v^2} = -\left\langle f, D\left\{|v-\u|^2\sqrt{\mu(1-\mu)}\right\}\right\rangle_{L_v^2}.
    \end{aligned}
  \end{equation*}
Then we can expand $\P(Df) $ as
\begin{equation*}
      \P (Df) = \left(\,\rho_{Df} + u_{Df}\cdot (v-\u) + \theta_{Df}|v-\u|^2\,\right)\sqrt{\mu(1-\mu)},
    \end{equation*}
where $\rho_{Df}(t,x),\,\theta_{Df}(t,x) \in \R$ and $u_{Df}(t,x) \in \R^3$ are given by
  \begin{equation*}
    \rho_{Df}(t,x)= \frac{Z-2K_ATX}{2(3E_2/2-K_AE_0T)},\quad u_{Df}= \frac{Y}{3E_2}, \quad \theta_{Df}= \frac{3E_2X-E_0Y}{6E_2(3E_2/2-K_AE_0T)}.
  \end{equation*}
Therefore, $\P(Df)$ is bounded by
\begin{equation}\label{PD-f-Es}
  C|(D\f,D\u,DT)||f|_{L_v^\infty}(1+|v|)^2\mu^{1/2}.
\end{equation}
Next, for any fixed $\widetilde{q}$ with $q<\widetilde{q}<1$, taking $q_1$ satisfying $q<q_1<\widetilde{q}$, then it follows from the estimate above, the equality \eqref{L-D-f} and \eqref{L-inv-decay-1} (replacing $q$ by $q_1$) that
\begin{equation}\label{D-w-f}
  \begin{aligned}
  &|D_w f| \leqslant |(1+|v|)\left(\P D f +\h D f\right)|\\
  &\quad\leqslant   C(D\f,D\u,DT)\left|(1+|v|)^2\mu^{-q_1/2}g\right|_{L_v^\infty}(1+|v|)^3\mu^{1/2}\\
  &\qquad +C \left|(1+|v|)^2\mu^{-q_1/2}\left\{D g -\left[(D \nu)f-\overline{\K} f\right]\right\}\right|_{L_v^\infty}(1+|v|)\mu^{q_1/2}\\
  &\qquad\leqslant C(D\f,D\u,DT)\left|(1+|v|)^2\mu^{-q_1/2}g\right|_{L_v^\infty}\mu^{q/2}\\
  &\qquad\quad+C\left|(1+|v|)^2\mu^{-q_1/2}\left\{D g -\left[(D \nu)f-\overline{\K} f\right]\right\}\right|_{L_v^\infty}\mu^{q/2}.
  \end{aligned}
\end{equation}
Thus it is sufficient to estimate $(D \nu)f$ and $\overline{\K} f$. A simple calculation give us for some constant $C>0$,
\begin{equation*}
  |D \nu| \leqslant C\Big|(D\f,D\u,D T)\Big|(1+|v|).
\end{equation*}
Hence by using \eqref{L-inv-decay-1} with replacing $q$ by $\widetilde{q}$, we derive
\begin{equation}\label{D-nu-f}
  \begin{aligned}
  (1+|v|)^2\mu^{-q_1/2}|(D \nu) f| \leqslant & C\Big|(D\f,D\u,D T)\Big|\left|(1+|v|)^3\mu^{-q_1/2}f\right|\\
  \leqslant & C\Big|(D\f,D\u,D T)\left|(1+|v|)^2\mu^{-\widetilde{q}/2}g\right|_{L_v^\infty}(1+|v|)^3\mu^{(\widetilde{q}-q_1)/2}\\
  \leqslant & C\Big|(D\f,D\u,D T)\left|(1+|v|)^2\mu^{-\widetilde{q}/2}g\right|_{L_v^\infty}.
  \end{aligned}
\end{equation}
For $\overline{\K}_1 f$, we have
\begin{equation*}
  \begin{aligned}
  \overline{\K}_1 f = & D \iint_{\R^3\times\S^2}|(v_*-v)\cdot\omega| \sqrt{\mu\mu_*}\frac{(1-\mu_*^\prime)(1-\mu^\prime)}{\sqrt{(1-\mu)(1-\mu_*)}}f_*\d \omega\d v_*-\K_1 (D f)\\
  = & \iint_{\R^3\times\S^2}|(v_*-v)\cdot\omega| \left\{ \left( D\sqrt{\mu\mu_*}\right)\frac{(1-\mu_*^\prime)(1-\mu^\prime)}{\sqrt{(1-\mu)(1-\mu_*)}} + \sqrt{\mu\mu_*}D\frac{(1-\mu_*^\prime)(1-\mu^\prime)}{\sqrt{(1-\mu)(1-\mu_*)}}\right\}
  f_*\d \omega\d v_*.
  \end{aligned}
\end{equation*}
For the above $\widetilde{q}$, computing directly gives us
\begin{equation*}
  \left|\left(D\sqrt{\mu\mu_*}\right)\frac{(1-\mu_*^\prime)(1-\mu^\prime)}{\sqrt{(1-\mu)(1-\mu_*)}} + \sqrt{\mu\mu_*}D\frac{(1-\mu_*^\prime)(1-\mu^\prime)}{\sqrt{(1-\mu)(1-\mu_*)}}\right| \leqslant C\Big|(D\f,D\u,D T)\Big| (\mu\mu_*)^{\widetilde{q}/2}.
\end{equation*}
As a result, we obtain
\begin{equation}\label{D-K-1-f}
  \begin{aligned}
  (1+|v|)^2\mu^{-q_1/2}|\overline{\K}_1 f| \leqslant & C\Big|(D\f,D\u,D T)\Big| \left|\mu^{-q_1/2}f\right|_{L_v^\infty} \int_{\R^3}|v_*-v|(1+|v|)^2(\mu\mu_*)^{(\widetilde{q}-q_1)/2}\d v_*\\
  \leqslant & C\Big|(D\f,D\u,D T)\Big|\left|(1+|v|)^2\mu^{-q_1/2}g\right|_{L_v^\infty}.
  \end{aligned}
\end{equation}
On the other hand, for
\begin{equation*}
  \begin{aligned}
  \overline{\K}_2 f =& D \left\{2\iint_{\R^3\times\S^2}|(v_*-v)\cdot\omega|\frac{\N}{\sqrt{\mu\mu'}\sqrt{(1-\mu)(1-\mu')}} f'\d \omega\d v_*\right\}-\K_2 (D f)\\
  = & 2\iint_{\R^3\times\S^2}|(v_*-v)\cdot\omega| \left\{\left(D\sqrt{\mu_*\mu_*^\prime}\right)(1-\mu_*^\prime)(1-\mu_*) + \sqrt{\mu_*\mu_*^\prime}D[(1-\mu_*^\prime)(1-\mu_*) ]\right\}f'\d \omega\d v_*,
  \end{aligned}
\end{equation*}
we can also get
\begin{equation*}
  \left|\left(D\sqrt{\mu_*\mu_*^\prime}\right)(1-\mu_*^\prime)(1-\mu_*) + \sqrt{\mu_*\mu_*^\prime}D[(1-\mu_*^\prime)(1-\mu_*) ]\right| \leqslant C\Big|(D\f,D\u,D T)\Big| \left(\mu_*\mu_*^\prime\right)^{\widetilde{q}/2}.
\end{equation*}
Hence
\begin{equation}\label{D-K-2-f}
  \begin{aligned}
  (1+|v|)^2\mu^{-q_1/2}|\overline{\K}_2 f| \leqslant & C\Big|(D\f,D\u,D T)\Big| \left|\mu^{-\widetilde{q}/2}f\right|_{L_v^\infty} \times\\
  &\qquad\iint_{\R^3}|v_*-v|(1+|v|)^2\left(\mu^\prime\mu_*^\prime\right)^{\widetilde{q}/2} \mu_*^{\widetilde{q}/2}\mu^{-q_1/2}\d\omega\d v_*\\
  \leqslant & C\Big|(D\f,D\u,D T)\Big|\left|(1+|v|)^2\mu^{-\widetilde{q}/2}g\right|_{L_v^\infty}.
  \end{aligned}
\end{equation}

In summary, the above procedure also applies to the case $D_w=D=\partial_t$. Therefore, plugging the estimates \eqref{D-nu-f}, \eqref{D-K-1-f} and \eqref{D-K-2-f} into \eqref{D-w-f}, we can establish that for any $0<q<\widetilde{q}<1$,
\begin{equation*}
  |D_w f| \leqslant C \Big|(D\f,D\u,D T)\Big|\left\{\left|(1+|v|)^2\mu^{-\widetilde{q}/2}D g\right|_{L_v^\infty}+ \left|(1+|v|)^2\mu^{-\widetilde{q}/2}g\right|_{L_v^\infty}\right\}\mu^{q/2}.
\end{equation*}

  Finally, we focus on the third statement. For $\K_1 f$, we deduce from \eqref{K-1K-2} that
    \begin{equation}\label{K-1-Es}
    \begin{aligned}
        \mu^{-q/2}|\K_1 f|\leqslant & C\int_{\R^3}|v_*-v|\sqrt{\mu\mu_*}\mu^{-q/2} |f_*|\d v_*\\
        \leqslant & C \left\{\int_{\R^3}|v_*-v|^2(\mu\mu_*)^{1-q}\d v_*\right\}^{1/2}\left\{\int_{\R^3}|f_*|^2\d v_*\right\}^{1/2} \leqslant C|f|_{L_v^2}.
  \end{aligned}
  \end{equation}
  For $\K_2 f$,  by using the variable changing $v_*-v\to V$ and the classical transformation (see page 43 in \cite{Glassey-1996Book})
  \begin{equation*}
    \d \omega \d V=\frac{2\d V_\bot \d V_{\p}}{|V_{\p}|^2},\quad V_{\p}=(V\cdot\omega)\omega\in \R^3,\,V_\bot=V-V_{\p}\in \R^2,
  \end{equation*}
  we arrive at
  \begin{equation*}
        |\K_2 f| \leqslant C \iint_{\R^3\times \R^2}|V_{\p}||f(v+V_{\p})| \mu^{-1/2}(v+V_{\p})\mu^{1/2}(v)\mu(v+V)\frac{\d V_\bot \d V_{\p}}{|V_{\p}|^2}.
  \end{equation*}
  Let
  \begin{equation*}
    \eta=v+V_{\p},\quad \zeta=\frac{1}{2}(v+\eta),
  \end{equation*}
  then
  \begin{equation*}
    V_\bot\cdot\zeta=V_\bot\cdot v=V_\bot \cdot\eta.
  \end{equation*}
  Thus
  \begin{equation*}
    \begin{aligned}
    &-\frac{|v-\u|^2}{2}+\frac{|\eta-\u|^2}{2}-|\eta+V_\bot-\u|^2 =-\frac{|v-\u|^2}{2}-\frac{|\eta-\u|^2}{2}-2V_\bot\cdot(\zeta-\u)-|V_\bot|^2\\
    &\qquad=-\frac{|v-\u|^2}{2}-2V_\bot\cdot(\zeta-\u)-2|\zeta-\u|^2+2(\zeta-\u)\cdot(v-\u)-\frac{|v-\u|^2}{2}-|V_\bot|^2\\
    &\qquad=(-|V_\bot|^2-2V_\bot\cdot(\zeta-\u)-|\zeta-\u|^2)+(-|\zeta-\u|^2+2(\zeta-\u)\cdot(v-\u)-|v-\u|^2)\\
    &\qquad=-|V_\bot+\zeta-\u|^2-\frac{1}{4}|V_{\p}|^2.
    \end{aligned}
  \end{equation*}
  Note the estimates \eqref{mu-Es}, we acquire
  \begin{equation}\label{K-2-Es}
        \begin{aligned}
        \mu^{-q/2}|\K_2 f| \leqslant & C \iint_{\R^3\times \R^2}|V_{\p}|\left|\{\mu^{-q/2}f\}(v+V_{\p})\right| \mu^{-(1-q)/2}(v+V_{\p})\mu^{(1-q)/2}(v)\mu(v+V)\frac{\d V_\bot \d V_{\p}}{|V_{\p}|^2}\\
        \leqslant & C \int_{\R^3}\frac{1}{|V_{\p}|}|\{\mu^{-q/2}f\}(v+V_{\p})| \exp\left\{-\frac{(1-q)|V_{\p}|^2}{8T}\right\} \d V_{\p}\\
        \leqslant &C \left\{\int_{\R^3} \frac{1}{|V_{\p}|^2}\exp\left\{-\frac{(1-q)|V_{\p}|^2}{4T}\right\} \d V_{\p}\right\}^{1/2}\left\{\int_{\R^3} |\{\mu^{-q/2}f\}(v+V_{\p})|^2 \d V_{\p}\right\}^{1/2}\\
        \leqslant &C |\mu^{-q/2}f|_{L_v^2}.
        \end{aligned}
  \end{equation}
  where we have used that $\int_{\R^2}\exp\left\{-\frac{|V_\bot+\zeta_\bot|^2}{2T}\right\}\d V_\bot$ is bounded for $\zeta_{\p}=[(\zeta-\u)\cdot\omega]\omega$ and $\zeta_\bot=\zeta-\zeta_{\p}$. By the estimates \eqref{K-1-Es} and \eqref{K-2-Es}, we obtain the third statement as announced.
\end{proof}

\begin{remark}
  Comparing with the Section 3 in \cite{Caflisch-1980CPAM}, where ($\L_B$ represents the linearized Boltzmann collision operator) it is pointed out that : $\L_B^{-1}$ preserves the decay property of $v$, we give the explicit estimate \eqref{L-inv-decay-1} for this property.
\end{remark}

\subsection{The Estimates for Kinetic Part}

First of all, we define the nonlinear operators $\widetilde{\Q}$ and $\widetilde{\T}$ by
 \begin{equation}\label{Q-til-op}
  \begin{aligned}
  \widetilde{\Q}(f,g)=&\frac{1}{\sqrt{\mu(1-\mu)}}\Q(\sqrt{\mu(1-\mu)}f,\,\sqrt{\mu(1-\mu)}g)\\
  =&\iint\limits_{\R^3\times\S^2}|(v-v_*)\cdot\omega|\frac{\sqrt{\N}}{\sqrt{\mu(1-\mu)}} \left[\, \frac{\sqrt{(1-\mu_*^\prime)(1-\mu^\prime)}}{\sqrt{(1-\mu_*)(1-\mu)}} f_*^\prime g^\prime- \frac{\sqrt{(1-\mu_*)(1-\mu)}}{\sqrt{(1-\mu_*^\prime)(1-\mu^\prime)}} f_* g\,\right]\d \omega\d v_*,
  \end{aligned}
\end{equation}
and
\begin{equation}\label{T-til-op}
  \begin{aligned}
  \widetilde{\T}(f,g,h)=&\frac{1}{\sqrt{\mu(1-\mu)}}\T(\sqrt{\mu(1-\mu)}f,\,\sqrt{\mu(1-\mu)}g,\,\sqrt{\mu(1-\mu)}h)\\
  =&\iint\limits_{\R^3\times\S^2}|(v-v_*)\cdot\omega|\frac{\sqrt{\N}}{\sqrt{\mu(1-\mu)}} \Big[\, \frac{\sqrt{(1-\mu_*^\prime)(1-\mu^\prime)}}{\sqrt{(1-\mu_*)(1-\mu)}}f_*^\prime g^\prime\left(\,\sqrt{\mu_*(1-\mu_*)}h_*+\sqrt{\mu(1-\mu)}h\,\right)\\
  &\hspace{3cm}- \frac{\sqrt{(1-\mu_*)(1-\mu)}}{\sqrt{(1-\mu_*^\prime)(1-\mu^\prime)}} f_* g \left(\,\sqrt{\mu_*^\prime(1-\mu_*^\prime)}h_*^\prime+\sqrt{\mu^\prime(1-\mu^\prime)}h^\prime\,\right)\,\Big]\d \omega\d v_*.
  \end{aligned}
\end{equation}

Now we are ready to evaluate
\begin{equation*}
  \tfrac{F_n}{\sqrt{\mu(1-\mu)}} \equiv f_n,\quad n=0,\cdots,3.
\end{equation*}
From the expansions \eqref{PF-n-Exp} of $\P f_n$ and the uniform estimates \eqref{rho-u-theta-n-Es}, we deduce that
\begin{equation}\label{PF-n-Es}
    |\P f_n|\leqslant C\, |(\rho_n,u_n,\theta_n)|\,(1+|v|)^2\mu^{1/2},\;n=1,2.
\end{equation}
Consequently, the same proposition as Proposition 3.1 in \cite{Caflisch-1980CPAM} can be established:
\begin{proposition}\label{Prop-F-n-Es}
Let $(\rho,\u,\E)$ be the smooth solution of the compressible Euler equations \eqref{Com-Euler-NonConser}. Hence $\f$ and $T$ given by \eqref{f} and \eqref{T} respectively are smooth under the assumption \eqref{Hypo} and $(\f,\u,T)$ form the local Fermi-Dirac distribution $\mu$. Then the coefficients $F_1,\,F_2,\,F_3$ of Hilbert expansion are smooth in $(t,x)$ and for any $0<q<1$, the coefficients have the decay given by
    \begin{equation}\label{F-n-Es}
    |F_n(t,x,v)|\leqslant C_{\beta^{(n)}}\mu^{(1+q)/2}(t,x,v),\;n=1,\,2,\,3,
    \end{equation}
    where
    \begin{equation*}
      \beta^{(n)}=(\beta_0^{(n)},\beta_1^{(n)},\beta_2^{(n)},\beta_3^{(n)}) \,\text{ and } \,\partial^{\beta^{(n)}}=\partial_t^{\beta_0^{(n)}}\partial_{x_1}^{\beta_1^{(n)}}\partial_{x_2}^{\beta_2^{(n)}}\partial_{x_3}^{\beta_3^{(n)}}
    \end{equation*}
    $C_{\beta^{(n)}}>0$ depends on $|\partial^{\gamma^{(n)}}(\rho_i,u_i,\theta_i)|$, $|D^{n-1}\nabla_x(\u,T)|$, $|D^{n-1}(\f,\u,T)|$ and $C_{\beta^{(j)}}$ for $|\beta^{(n)}|\leqslant n$, $\gamma^{(n)}=n-i$ and $j\leqslant n-1$, here $\gamma^{(n)}=(\gamma_0^{(n)},\gamma_1^{(n)},\gamma_2^{(n)},\gamma_3^{(n)})$.
\end{proposition}

\begin{proof}
  Noting that for $f_n\,(n=0,\cdots,3)$, we already have  $|f_0| \leqslant C \mu^{1/2}$ and the estimates \eqref{PF-n-Es}, thus it suffices to evaluate $\h f_n$ for $n=1,2,3$.

  {\bf Estimate for $\h f_1$}. Recall \eqref{I-P-F-n}, we know
\begin{equation}\label{I-P-f1}
  \h f_1=\L^{-1} \left(-\tfrac{(\partial_t+v\cdot\nabla_x)\mu}{\sqrt{\mu(1-\mu)}}\right)=-\L^{-1}\left(\nabla_x \u : \B\left(\tfrac{v-\u}{\sqrt{T}}\right)+\frac{\nabla_x T}{\sqrt{T}}\cdot\A\left(\tfrac{v-\u}{\sqrt{T}}\right)\right).
\end{equation}
    For any fixed $0<q<\widetilde{q}<1$, we take $q_1$ with $q<q_1<\widetilde{q}$. It follows from \eqref{L-inv-decay-1} that
\begin{equation*}
  \begin{aligned}
  \left|\h f_1\right| \leqslant &C \left|(1+|v|)^2\mu^{-q_1/2}\left(\nabla_x \u : \B\left(\tfrac{v-\u}{\sqrt{T}}\right)+\frac{\nabla_x T}{\sqrt{T}}\cdot\A\left(\tfrac{v-\u}{\sqrt{T}}\right)\right)\right|_{L_v^\infty}\mu^{q_1/2}\\
  \leqslant & C \left|(\nabla_x \u, \nabla_x T)\right|\mu^{q_1/2},
  \end{aligned}
\end{equation*}
and hence combining with \eqref{PF-n-Es} we obtain
\begin{equation}\label{F1-Es}
  |f_1(t,x,v)|\leqslant  \left|\P f_1\right|+\left|\h f_1\right|\leqslant C\Big(|(\rho_1,u_1,\theta_1)|+\left|(\nabla_x \u, \nabla_x T)\right|\Big)\mu^{q_1/2}\stackrel{\triangle}{=}C_{\beta^{(1)}}\mu^{q_1/2}.
\end{equation}

Since the operator $\L$  is acting on the variable $v$, the equalities \eqref{I-P-f1} and \eqref{PF-n-Exp} indicate $f_1(t,x,v)$  is smooth in $(t,x)$.

    {\bf Estimate for $\h f_2$}. The kinetic part of $f_2$ satisfies
    \begin{equation}\label{I-P-f2}
      \begin{aligned}
        \L\left(\h f_2\right)=&-\left((\partial_t+v\cdot\nabla_x)f_1+\frac{(1-2\mu)(\partial_t+v\cdot\nabla_x)\mu}{2\mu(1-\mu)}f_1 \right)\\
        &+\left(\widetilde{\Q}(f_1,f_1) -\widetilde{\T}(f_1,f_1,f_0)-\widetilde{\T}(f_1,f_0,f_1)-\widetilde{\T}(f_0,f_1,f_1)\right).
      \end{aligned}
    \end{equation}
    Based on the estimate \eqref{L-inv-decay-1}, we need to estimate the terms on the right hand side above. Firstly, taking $q_2$ with $q<q_2<q_1$, combining \eqref{I-P-f1} with the estimate \eqref{L-inv-decay-2} (replacing $q$ by $q_1$), we obtain
    \begin{equation*}
      \begin{aligned}
      \left|(1+|v|)^2\mu^{-q_2/2}D_w (\h f_1)\right|\leqslant &C\Big|(D\f,D\u,D T)\Big|(1+|v|)^3\mu^{\frac{q_1-q_2}{2}}\times\\
      &\left\{\left|(1+|v|)^2\mu^{-\widetilde{q}/2}D \left(\nabla_x \u : \B\left(\tfrac{v-\u}{\sqrt{T}}\right)+\frac{\nabla_x T}{\sqrt{T}}\cdot\A\left(\tfrac{v-\u}{\sqrt{T}}\right)\right)\right|_{L_v^\infty}\right.\\ &\qquad\qquad\left.+\left|(1+|v|)^2\mu^{-\widetilde{q}/2}\left(\nabla_x \u : \B\left(\tfrac{v-\u}{\sqrt{T}}\right)+\frac{\nabla_x T}{\sqrt{T}}\cdot\A\left(\tfrac{v-\u}{\sqrt{T}}\right)\right)\right|_{L_v^\infty}\right\}\\
      \leqslant & C|D(\f,\u,T)|\Big\{|D(\nabla_x\u,\nabla_x T)|+\\
      &\qquad\quad|(\nabla_x\u,\nabla_x T)|\big[\,1+|D(\f,\u,T)|\,\big]+|D\u|(|\nabla_x\u|+|\nabla_x T|)\Big\}.
      \end{aligned}
    \end{equation*}
    On the other hand, clearly, we have
    \begin{equation*}
      \left|(1+|v|)^2\mu^{-q_2/2}D(\P f_1)\right|\leqslant C \Big(|D(\rho_1,u_1,\theta_1)|+|D(\f,\u,T)|\Big).
    \end{equation*}
    Therefore, we arrive at
    \begin{equation*}
      \begin{aligned}
      &\left|(1+|v|)^2\mu^{-q_2/2}(\partial_t+v\cdot\nabla_x)f_1\right|\\
      &\qquad\leqslant C\Big\{ |\partial_t(\rho_1,u_1,\theta_1)|+|\nabla_x(\rho_1,u_1,\theta_1)|+|\partial_t(\f,\u,T)|+|\nabla_x(\f,\u,T)|\Big\}\times\\
      &\qquad\qquad\qquad\Big\{\,1+|\partial_t(\nabla_x \u,\nabla_x  T)|+|\nabla_x(\nabla_x \u,\nabla_x  T)|+(|\partial_t\u|+|\nabla_x\u|)(|\nabla_x\u|+|\nabla_x T|)\\
      &\hspace{6cm}+|(\nabla_x \u,\nabla_x  T)|\big[\,1+|\partial_t(\f,\u,T)|+|\nabla_x(\f,\u,T)|\,\big]\Big\}.
      \end{aligned}
    \end{equation*}
    Next, the estimate \eqref{F1-Es} implies that
    \begin{equation*}
      \begin{aligned}
      &\left|(1+|v|)^2\mu^{-q_2/2}\frac{(1-2\mu)(\partial_t+v\cdot\nabla_x)\mu}{2\mu(1-\mu)}f_1\right|\\
      &\qquad= \left|(1+|v|)^2\mu^{-q_2/2}\frac{1-2\mu}{2\sqrt{\mu(1-\mu)}}\left(\nabla_x \u : \B\left(\tfrac{v-\u}{\sqrt{T}}\right)+\frac{\nabla_x T}{\sqrt{T}}\cdot\A\left(\tfrac{v-\u}{\sqrt{T}}\right)\right) f_1\right|\\
      &\qquad\leqslant  C\Big(|(\rho_1,u_1,\theta_1)|+\left|(\nabla_x \u, \nabla_x T)\right|\Big)^2(1+|v|)^5\mu^{(q_1-q_2)/2} \leqslant C\Big(|(\rho_1,u_1,\theta_1)|+\left|(\nabla_x \u, \nabla_x T)\right|\Big)^2.
      \end{aligned}
    \end{equation*}
    For the nonlinear terms, by \eqref{F1-Es} we know
    \begin{equation*}
      \begin{aligned}
        &\left|(1+|v|)^2\mu^{-q_2/2}\widetilde{\Q}(f_1,f_1)\right|\leqslant C(1+|v|)^2\mu^{-q_2/2}\iint\limits_{\R^3\times\S^2}|(v-v_*)\cdot\omega|\sqrt{\mu_*} \left[\, \left|(f_1)_*^\prime f_1^\prime\right| + \left|(f_1)_* f_1\right|\,\right]\d \omega\d v_*\\
        &\qquad\leqslant C\Big(|(\rho_1,u_1,\theta_1)|+\left|(\nabla_x \u, \nabla_x T)\right|\Big)^2\int\limits_{\R^3}(1+|v|)^2|v-v_*|\mu_*^{(1+q_1)/2}\mu^{(q_1-q_2)/2}\d v_*\\
        &\qquad\leqslant C\Big(|(\rho_1,u_1,\theta_1)|+\left|(\nabla_x \u, \nabla_x T)\right|\Big)^2.
      \end{aligned}
    \end{equation*}
    and
    \begin{equation*}
      \begin{aligned}
        &\left|(1+|v|)^2\mu^{-q_2/2}\left\{\widetilde{\T}(f_0,f_1,f_1)+\widetilde{\T}(f_1,f_0,f_1)\right\}\right|\leqslant C(1+|v|)^2\mu^{-q_2/2}\times\\
        &\quad\iint\limits_{\R^3\times\S^2}|(v-v_*)\cdot\omega|\sqrt{\mu_*} \left\{\, \left|(f_0)_*^\prime f_1^\prime\right|\left|\sqrt{\mu_*}(f_1)_*+\sqrt{\mu}f_1\right|+\left|(f_0)_* f_1\right|\left|\sqrt{\mu_*^\prime}(f_1)_*^\prime+\sqrt{\mu^\prime}f_1^\prime\right|\right.\\
        &\hspace{4cm}\left.+\left|(f_1)_*^\prime f_0^\prime\right|\left|\sqrt{\mu_*}(f_1)_*+\sqrt{\mu}f_1\right|  + \left|(f_1)_* f_0\right|\left|\sqrt{\mu_*^\prime}(f_1)_*^\prime+\sqrt{\mu^\prime}f_1^\prime\right|\,\right\}\d \omega\d v_*\\
        \leqslant & C\Big(|(\rho_1,u_1,\theta_1)|+\left|(\nabla_x \u, \nabla_x T)\right|\Big)^2\int\limits_{\R^3}(1+|v|)^2|v-v_*|\mu_*^{1/2}\mu^{-q_2/2}\times\\
        &\qquad\qquad\qquad(\mu\mu_*)^{q_1/2}\left(\mu_*^{(1+q_1)/2} +\mu^{(1+q_1)/2}+(\mu_*^\prime)^{(1+q_1)/2}+(\mu^\prime)^{(1+q_1)/2}\right)\d v_*\\
        \leqslant & C\Big(|(\rho_1,u_1,\theta_1)|+\left|(\nabla_x \u, \nabla_x T)\right|\Big)^2.
      \end{aligned}
    \end{equation*}
    Similarly,
    \begin{equation*}
      \left|(1+|v|)^2\mu^{-q_2/2}\widetilde{\T}(f_1,f_1,f_0)\right| \leqslant C\Big(|(\rho_1,u_1,\theta_1)|+\left|(\nabla_x \u, \nabla_x T)\right|\Big)^2.
    \end{equation*}
    In view of all the estimates above, we conclude by \eqref{L-inv-decay-1} that for multi-index $\beta^{(2)}$
    \begin{equation*}
      |\h f_2| \leqslant C_{\beta^{(2)}} \mu^{q_2/2},
    \end{equation*}
    where $C_{\beta^{(2)}}>0$ is depending on $|\partial^{\beta^{(2)}-\gamma}(\rho_1,u_1,\theta_1)|$, $|\partial^{\beta^{(2)}}(\f,\u,T)|$ for $|\beta^{(2)}| \leqslant 2$ and $|\gamma|=1$.

    {\bf Estimate for $\h f_3$}. The kinetic part of $f_3$ satisfies
    \begin{equation*}
        \begin{aligned}
        \L(\h f_3) = & -\left((\partial_t+v\cdot\nabla_x)f_2+\frac{(1-2\mu)(\partial_t+v\cdot\nabla_x)\mu}{2\mu(1-\mu)}f_2 \right)\\
        &\qquad+\widetilde{\Q}(f_1,f_2)+\widetilde{\Q}(f_2,f_1)-\sum_{\substack{i+j+k=3\\i,j,k<3}}\widetilde{\T}(f_i,f_j,f_k).
        \end{aligned}
    \end{equation*}
    We take $q_3$ satisfying $q<q_3<q_2^{\prime\prime\prime}<q_2^{\prime\prime}<q_2^{\prime}<q_2$ and repeat the procedure above to yield
    \begin{equation*}
      \begin{aligned}
        \left|(1+|v|)^2\mu^{-q_3/2}D_w(\P f_2)\right|\leqslant & C \Big(|D(\rho_2,u_2,\theta_2)|+|D(\f,\u,T)|\Big),\\
        \left|(1+|v|)^2\mu^{-q_3/2}\frac{(1-2\mu)(\partial_t+v\cdot\nabla_x)\mu}{2\mu(1-\mu)}f_2\right|\leqslant & C_{\beta^{(2)}}\left|(\nabla_x \u, \nabla_x T)\right|,\\
        \left|(1+|v|)^2\mu^{-q_3/2}\left[\,\widetilde{\Q}(f_1,f_2)+\widetilde{\Q}(f_2,f_1)\,\right]\right|\leqslant & C_{\beta^{(1)}}C_{\beta^{(2)}},\\
        |(1+|v|)^2\mu^{-q_3/2}\sum_{\substack{i+j+k=3\\i,j,k<3}}\widetilde{\T}(f_i,f_j,f_k)|\leqslant & C_{\beta^{(1)}}^3+C_{\beta^{(1)}}^2C_{\beta^{(2)}}+C_{\beta^{(1)}}C_{\beta^{(2)}}^2.
      \end{aligned}
    \end{equation*}

    We turn to $D_w(\h f_2)$. On a basis of \eqref{L-inv-decay-2} (replacing $(q,\widetilde{q})$ by $(q_2^{\prime\prime\prime},q_2^{\prime\prime})$ ) and \eqref{I-P-f2}, we obtain
    \begin{equation*}
      \begin{aligned}
      &\left|(1+|v|)^2\mu^{-q_3/2}D_w (\h f_2)\right|\\
      &\qquad\leqslant C\Big|(D\f,D\u,D T)\Big|(1+|v|)^3\mu^{\frac{q_2^{\prime\prime\prime}-q_3}{2}}\times\\
      &\qquad\qquad\left\{\left|(1+|v|)^2\mu^{-q_2^{\prime\prime}/2}D \left[-\left((\partial_t+v\cdot\nabla_x)f_1+\tfrac{(1-2\mu)(\partial_t+v\cdot\nabla_x)\mu}{2\mu(1-\mu)}f_1 \right)\right.\right.\right.\\ &\qquad\qquad\qquad\left.\left.+\left(\widetilde{\Q}(f_1,f_1) -\widetilde{\T}(f_1,f_1,f_0)-\widetilde{\T}(f_1,f_0,f_1)-\widetilde{\T}(f_0,f_1,f_1)\right)\right]\right|_{L_v^\infty}\\
      &\qquad\qquad+\left|(1+|v|)^2\mu^{-q_2^{\prime\prime}/2}\left[-\left((\partial_t+v\cdot\nabla_x)f_1+\tfrac{(1-2\mu)(\partial_t+v\cdot\nabla_x)\mu}{2\mu(1-\mu)}f_1 \right)\right.\right.\\
      &\qquad\qquad\qquad\left.\left.\left.+\left(\widetilde{\Q}(f_1,f_1) -\widetilde{\T}(f_1,f_1,f_0)-\widetilde{\T}(f_1,f_0,f_1)-\widetilde{\T}(f_0,f_1,f_1)\right)\right]\right|_{L_v^\infty}\right\}.
      \end{aligned}
    \end{equation*}
    Since $q_2^{\prime\prime}<q_2$, we repeat the process of estimating $\h f_2$ to obtain
    \begin{equation*}
      \begin{aligned} &\left|(1+|v|)^2\mu^{-q_2^{\prime\prime}/2}\left((\partial_t+v\cdot\nabla_x)f_1+\tfrac{(1-2\mu)(\partial_t+v\cdot\nabla_x)\mu}{2\mu(1-\mu)}f_1 \right)\right|_{L_v^\infty}\\
      &\qquad+\left|(1+|v|)^2\mu^{-q_2^{\prime\prime}/2}\left(\widetilde{\Q}(f_1,f_1) -\widetilde{\T}(f_1,f_1,f_0)-\widetilde{\T}(f_1,f_0,f_1)-\widetilde{\T}(f_0,f_1,f_1)\right)\right|_{L_v^\infty} \leqslant C_{\beta^{(2)}} .
      \end{aligned}
    \end{equation*}
    A careful calculation yields
    \begin{equation*}
        \begin{aligned}
        &\left|(1+|v|)^2\mu^{-q_2^{\prime\prime}/2}D\left(\tfrac{(1-2\mu)(\partial_t+v\cdot\nabla_x)\mu}{2\mu(1-\mu)}f_1 \right)\right|_{L_v^\infty}+\left|(1+|v|)^2\mu^{-q_2^{\prime\prime}/2}D\widetilde{\Q}(f_1,f_1)\right|_{L_v^\infty}\\
        &\hspace{3cm}+\left|(1+|v|)^2\mu^{-q_2^{\prime\prime}/2}D\left( \widetilde{\T}(f_1,f_1,f_0) +\widetilde{\T}(f_1,f_0,f_1)+\widetilde{\T}(f_0,f_1,f_1)\right)\right|_{L_v^\infty}\\
        \leqslant & \Big\{C_{\beta^{(2)}} + \Big[|D(\nabla_x\u,\nabla_x T)| + |D\u|(|\nabla_x\u|+|\nabla_x T|)\Big]C_{\beta^{(1)}}\Big\}\\
        &\qquad+ C_{\beta^{(1)}}\left(|D(\f,\u,T)|C_{\beta^{(1)}}+C_{\beta^{(2)}}\right)+\Big\{ C_{\beta^{(1)}}C_{\beta^{(2)}} + C_{\beta^{(1)}}^2 + |D(\f,\u,T)|C_{\beta^{(2)}}\Big\}.
        \end{aligned}
    \end{equation*}
    The remaining term that need to be estimated is $D(\partial_t+v\cdot\nabla_x)f_1$, which we decompose into
    \begin{equation*}
      DDf_1= \P DDf_1+ \h DDf_1.
    \end{equation*}
    Obviously, we can expand $\P DDf_1$ to get
    \begin{equation*}
    \begin{aligned}
        &\big|(1+|v|)^3\mu^{-q_2^{\prime\prime}/2}\P DDf_1\big|\\
        &\qquad\leqslant  C\Big\{|DD(\rho_1,u_1,\theta_1)|_{L_v^\infty}+|D(\f,\u,T)||Df_1|_{L_v^\infty}\\
        &\hspace{4cm}+\left(|D(\f,\u,T)|^2+|DD(\f,\u,T)|\right)|f_1|_{L_v^\infty}\Big\}(1+|v|)^5\mu^{(1-q_2^{\prime\prime})/2}\\
        &\qquad\leqslant C|DD(\rho_1,u_1,\theta_1)|_{L_v^\infty} +|D(\f,\u,T)|C_{\beta^{(2)}}+\left(\,|D(\f,\u,T)|^2+|DD(\f,\u,T)|\,\right)C_{\beta^{(1)}}.
    \end{aligned}
  \end{equation*}
    For $\h DDf_1$, since
    \begin{equation*}
      \L(\h DDf_1)=\L(DDf_1) = DDg - (DD\nu)f_1 - (D\nu)Df_1 + D\overline{\K}f_1 - (D\nu)Df_1 + \overline{\K}[Df_1].
    \end{equation*}
    where $g=\nabla_x \u : \B\left(\tfrac{v-\u}{\sqrt{T}}\right)+\frac{\nabla_x T}{\sqrt{T}}\cdot\A\left(\tfrac{v-\u}{\sqrt{T}}\right)$, we make use of \eqref{L-inv-decay-1} (replacing $q$ by $q_2^\prime$) to get
    \begin{equation*}
      \begin{aligned}
        & \left|(1+|v|)^3\mu^{-q_2^{\prime\prime}/2}\h(DDf_1)\right|\leqslant  C(1+|v|)^3\mu^{(q_2^\prime-q_2^{\prime\prime})/2}\times\\
        &\hspace{2cm}\left|(1+|v|)^2\mu^{-q_2^\prime/2} \left[\, DDg - (DD\nu)f_1 - 2(D\nu)Df_1 + D\overline{\K}f_1+ \overline{\K}Df_1\,\right]\right|_{L_v^\infty}.
      \end{aligned}
    \end{equation*}
    Finally, we arrive at
    \begin{equation*}
      \begin{aligned}
      &\left|(1+|v|)^3\mu^{-q_2^{\prime\prime}/2}\h(DDf_1)\right|\\
      &\qquad\leqslant C \Big\{|DD\nabla_x(\u,T)| + [\,|DD(f,\u,T)|+|D(f,\u,T)|^2\,]|\nabla_x(\u,T)|\\
      &\qquad\quad+[\,|\nabla_x\u|+|D(f,\u,T)|\,][\,|D\nabla_x(\u,T)|+|D\u||\nabla_x(\u,T)|\,]\Big\}\\
      &\qquad\quad+C_{\beta^{(1)}} \left(\left|D(\f,\u,T)\right|^2+\left|DD(\f,\u,T)\right|\right)+C_{\beta^{(2)}}\Big|(D\f,D\u,D T)\Big|+\Big|(D\f,D\u,D T)\Big|C_{\beta^{(2)}} \\ &\qquad\quad+\left\{\left|D(\f,\u,T)\right|^2+\left|DD(\f,\u,T)\right|+\left|D(\u,T)\right|\left|D\u\right|\right\}C_{\beta^{(1)}}.
      \end{aligned}
    \end{equation*}

    All in all, we conclude that
    \begin{equation*}
      |\h f_3| \leqslant C_{\beta^{(3)}}\mu^{q_3/2},
    \end{equation*}
    where $C_{\beta^{(3)}}>0$ is depending on $C_{\beta^{(1)}}$, $C_{\beta^{(2)}}$, $\partial^{\gamma} (\rho_i,u_i,\theta_i)$, and $\partial^{\beta^{(3)}}(\f,\u,T)$ for $|\beta^{(3)}|\leqslant 3$ and $|\gamma|+i=3$.

    Note that $0<q<q_3<q_2<q_1<1$, we complete the proof.
\end{proof}

\section{Compressible Euler Limit: Proof of Theorem \ref{Thm-Euler-limit}}\label{Sec-CE-Lim}
In this section, we focus on the compressible Euler limit from the scaled BFD equation \eqref{SBFD}. On the basis of the $L^2-L^\infty$ approach \cite{Guo-2010ARMA, GJJ-2010CPAM}, it suffices to estimate norms $\|f_{R,\eps}\|_2$ and $\|h_{R,\eps}\|_\infty$ for the equation \eqref{Rem-eq} of the reminder $F_{R,\eps}$. Here, $f_{R,\eps}$ and $h_{R,\eps}$ are given by \eqref{Reminder-f} and \eqref{Reminder-h} respectively, i.e.
\begin{equation*}
  f_{R,\eps}=\frac{F_{R,\eps}}{\sqrt{\mu(1-\mu)}},\quad h_{R,\eps}=w(v)\frac{F_{R,\eps}}{\sqrt{\mu_F(1-\mu_F)}},\quad \mu_F=\frac{1}{1+e^{\frac{|v|^2}{2T_m}-1}},\quad w(v)=(1+|v|)^l,\,l\geqslant 9.
\end{equation*}

Note the relation \eqref{Relation-rhoEfT} between $\rho$ and $(\f, T)$ implies that $\rho$ has a positive lower bound:
\begin{equation*}
  \int_{\R^3}\frac{T_m^{3/2}}{1+e^{\frac{|v|^2}{2}-\f_m}}\d v \leqslant \rho(t,x) \leqslant \int_{\R^3}\frac{T_M^{3/2}}{1+e^{\frac{|v|^2}{2}-\f_M}}\d v,
\end{equation*}
combining with the condition \eqref{T-bd}, we are able to establish the following $L^2$ and $L^\infty$ estimates to certify Theorem \ref{Thm-Euler-limit}.
\begin{lemma}[$L^2$-Estimates]\label{Lem-L2-Es}
  Let $(\rho,\u,\E)(t,x)$ be the smooth solution to the compressible Euler system \eqref{Com-Euler-NonConser} over $[0,\tau]$ obtained in Theorem \ref{Thm-rhouE-Existence}. Let $f_{R,\eps}$ and $h_{R,\eps}$ be defined as in \eqref{Reminder-f} and \eqref{Reminder-h} and $\lambda_0>0$ be as in the coercivity estimate \eqref{coe-L}. Then there exist two constants $\eps_1>0$ and $C=C(\mu,F_1,F_2,F_3)>0$ such that for all $\eps<\eps_1$ and $t\in [0,\tau]$,
  \begin{equation}\label{L2-Es}
    \begin{aligned}
    &\frac{1}{2}\frac{\d}{\d t} \|f_{R,\eps}(t)\|_2^2+\frac{\lambda_0}{2\eps} \|\h f_{R,\eps}(t)\|_{\nu(\mu)}^2\\
    &\qquad\leqslant C\left(1+\sqrt{\eps}\left\|\eps^{3/2} h_{R,\eps}\right\| _\infty\right)\left(\|f_{R,\eps}\|_2+\|f_{R,\eps}\|_2^2\right)+C\eps^2 \left\|\eps^{3/2}h_{R,\eps} \right\|_\infty^2\|f_{R,\eps}\|_2^2.
    \end{aligned}
  \end{equation}
\end{lemma}

\begin{lemma}[$L^\infty$-Estimates]\label{Lem-L-inf-Es}
  Under the same assumptions as in Lemma \ref{Lem-L2-Es}, there are constants $\eps_2>0$ and $C=C(\mu, F_1,F_2,F_3)>0$ such that for all $\eps<\eps_2$ and $t\in [0,\tau]$,
  \begin{equation}\label{L-inf-Es}
    \sup_{s\in [0,t]}\|\eps^{3/2}h_{R,\eps}(s)\|_\infty \leqslant C \big(\|\eps^{3/2}h_{R,\eps}(0)\|_\infty+\sup_{s\in [0,\tau]}\|f_{R,\eps}(s)\|_2+\eps^{5/2}\big).
  \end{equation}
\end{lemma}

\begin{proof}[\bf Proof of Theorem \ref{Thm-Euler-limit}]
  Based on Lemmas \ref{Lem-L2-Es} and \ref{Lem-L-inf-Es}, we acquire for some $C_0>0$
  \begin{equation*}
  \begin{aligned}
  &\frac{1}{2}\frac{\d}{\d t} \|f_{R,\eps}(t)\|_2^2+\frac{\lambda_0}{2\eps} \|\h f_{R,\eps}(t)\|_{\nu(\mu)}^2\\
  &\quad\leqslant  C_0 (1+\|f_{R,\eps}\|_2)\|f_{R,\eps}\|_2\left(1+\sqrt{\eps}\left\|\eps^{3/2} h_{R,\eps}\right\| _\infty+\eps^2\left\|\eps^{3/2} h_{R,\eps}\right\| _\infty^2\right)\\
  &\quad\leqslant C_0(1+\|f_{R,\eps}\|_2)\|f_{R,\eps}\|_2\Bigg\{1+\left(1+\eps\sup_{s\in [0,\tau]} \left\|\eps^{3/2} h_{R,\eps}\right\| _\infty\right)\times\\
  &\qquad\qquad\left(\sqrt{\eps}\|\eps^{3/2}h_{R,\eps}(0)\|_\infty+\sqrt{\eps}\sup_{s\in [0,\tau]}\|f_{R,\eps}(s)\|_2+\eps^3\right)\Bigg\}.
  \end{aligned}
  \end{equation*}
  For bounded $\dis\sup_{0\leqslant s \leqslant \tau}\|f_{R,\eps}(s)\|_2$ and for $\eps$ small enough, the Gr\"{o}nwall's inequality indicates
\begin{equation*}
  1+\|f_{R,\eps}(t)\|_2 \leqslant (1+\|f_{R,\eps}(0)\|_2)\exp\left\{3 C_0 \tau\cdot(1+\sqrt{\eps}\|\eps^{3/2}h_{R,\eps}(0)\|_\infty+\sqrt{\eps}\sup_{s\in [0,\tau]}\|f_{R,\eps}(s)\|_2)\right\}.
\end{equation*}
  Employing the Taylor expansion of the exponential function in the above inequality, we arrive at
  \begin{equation*}
    \sup_{0\leqslant t \leqslant \tau}\|f_{R,\eps}(t)\|_2\leqslant C_\tau (1+\|f_{R,\eps}(0)\|_2 + \|\eps^{3/2}h_{R,\eps}(0)\|_2).
  \end{equation*}
  Combining with \eqref{L-inf-Es}, we obtain
  \begin{equation*}
    \sup_{s\in [0,t]}\|\eps^{3/2}h_{R,\eps}(s)\|_\infty \leqslant C_\tau (1+\|f_{R,\eps}(0)\|_2 + \|\eps^{3/2}h_{R,\eps}(0)\|_2).
  \end{equation*}
  We complete the Theorem as announced.
\end{proof}

\subsection{$L^2$-Estimate: Proof of Lemma \ref{Lem-L2-Es}.}  Inserting \eqref{Reminder-f} into \eqref{Rem-eq}, we get that
\begin{equation}\label{Rem-f-eq}
  \begin{aligned}
    &(\partial_t +v\cdot \nabla_x)f_{R,\eps}+\frac{1}{\eps}\L f_{R,\eps}\\
    &\quad\qquad=-\frac{(\partial_t +v\cdot \nabla_x)(\,\sqrt{\mu(1-\mu)}\,)}{\sqrt{\mu(1-\mu)}}f_{R,\eps}+\sum_{i=1}^{3}\eps^{i-1} \Big(\,\widetilde{\Q}(f_i,f_{R,\eps})+\widetilde{\Q}(f_{R,\eps},f_i)\,\Big) +\eps^2\widetilde{\Q}(f_{R,\eps},f_{R,\eps})\\
    &\qquad\qquad\qquad-\sum_{i+j=1,\,i,j\leqslant 3}^{6}\eps^{i+j-1} \Big(\,\widetilde{\T}(f_{R,\eps},f_i,f_j)+\widetilde{\T}(f_i,f_{R,\eps},f_j)+\widetilde{\T}(f_i,f_j,f_{R,\eps})\,\Big)\\
    &\qquad\qquad\qquad\qquad-\sum_{i=0}^{3}\eps^{i+2}\Big(\,\widetilde{\T}(f_i,f_{R,\eps},f_{R,\eps})+\widetilde{\T}(f_{R,\eps},f_i,f_{R,\eps})+\widetilde{\T}(f_{R,\eps},f_{R,\eps},f_i)\,\Big)\\
    &\qquad\qquad\qquad\qquad\qquad-\eps^5\widetilde{\T}(f_{R,\eps},f_{R,\eps},f_{R,\eps})+\widetilde{R}_\eps,
  \end{aligned}
\end{equation}
where
\begin{equation*}
  \begin{aligned}
  \widetilde{\Q}(f,g)=&\frac{1}{\sqrt{\mu(1-\mu)}}\Q(\sqrt{\mu(1-\mu)}f,\,\sqrt{\mu(1-\mu)}g)\\
  =&\iint\limits_{\R^3\times\S^2}|(v-v_*)\cdot\omega|\frac{\sqrt{\N}}{\sqrt{\mu(1-\mu)}} \left[\, \frac{\sqrt{(1-\mu_*^\prime)(1-\mu^\prime)}}{\sqrt{(1-\mu_*)(1-\mu)}} f_*^\prime g^\prime- \frac{\sqrt{(1-\mu_*)(1-\mu)}}{\sqrt{(1-\mu_*^\prime)(1-\mu^\prime)}} f_* g\,\right]\d \omega\d v_*,
  \end{aligned}
\end{equation*}
\begin{equation*}
  \begin{aligned}
  \widetilde{\T}(f,g,h)=&\frac{1}{\sqrt{\mu(1-\mu)}}\T(\sqrt{\mu(1-\mu)}f,\,\sqrt{\mu(1-\mu)}g,\,\sqrt{\mu(1-\mu)}h)\\
  =&\iint\limits_{\R^3\times\S^2}|(v-v_*)\cdot\omega|\frac{\sqrt{\N}}{\sqrt{\mu(1-\mu)}} \Big[\, \frac{\sqrt{(1-\mu_*^\prime)(1-\mu^\prime)}}{\sqrt{(1-\mu_*)(1-\mu)}}f_*^\prime g^\prime\left(\,\sqrt{\mu_*(1-\mu_*)}h_*+\sqrt{\mu(1-\mu)}h\,\right)\\
  &\hspace{3cm}- \frac{\sqrt{(1-\mu_*)(1-\mu)}}{\sqrt{(1-\mu_*^\prime)(1-\mu^\prime)}} f_* g \left(\,\sqrt{\mu_*^\prime(1-\mu_*^\prime)}h_*^\prime+\sqrt{\mu^\prime(1-\mu^\prime)}h^\prime\,\right)\,\Big]\d \omega\d v_*,
  \end{aligned}
\end{equation*}
and
\begin{equation*}
  \begin{aligned}
  \widetilde{R}_\eps&=\sum_{\substack{i+j\geqslant 4\\i,j\leqslant 3}}\eps^{i+j-4}\widetilde{\Q}(f_i,f_j)-\sum_{\substack{i+j+k\geqslant 4\\i,j,k\leqslant 3}}\eps^{i+j+k-4}\widetilde{\T}(f_i,f_j,f_k)-\frac{(\partial_t +v\cdot \nabla_x)F_3}{\sqrt{\mu(1-\mu)}},\\
  f_n&=\frac{F_n}{\sqrt{\mu(1-\mu)}},\;n=1,\,2,\quad f_3=\h\frac{F_3}{\sqrt{\mu(1-\mu)}}.
  \end{aligned}
\end{equation*}
Computing directly gives us
\begin{equation*}
  \begin{aligned}
  \varsigma:= \frac{(\partial_t +v\cdot \nabla_x)\mu}{\mu(1-\mu)} = & \left[\,(\nabla_x \u):\frac{(v-\u)\otimes(v-\u)}{T}-(\nabla_x\cdot\u)\frac{|v-\u|^2}{T}\,\right] \\ &+\left(\frac{|v-\u|^2}{2T}-\frac{5\E}{2\rho T}\right)\frac{v-\u}{\sqrt{T}}\cdot\frac{\nabla_x T}{\sqrt{T}},
  \end{aligned}
\end{equation*}
which behaves like $|v|^3$ as $|v|\to +\infty$.

Now we are ready to establish $L^2$-estimates on the remainder equation \eqref{Rem-f-eq}. Taking $L_\mu^2$ inner products with $f_{R,\eps}$, the remainder equation \eqref{Rem-f-eq} yields:
\begin{equation}\label{L2-Es-1}
  \begin{aligned}
  \frac{1}{2}\frac{\d}{\d t}\|f_{R,\eps}\|_2^2+ \frac{\lambda_0}{\eps}\|\h f_{R,\eps}\|_{\nu(\mu)}^2\leqslant & -\frac{1}{2}\iint_{\R^3\times\R^3} |f_{R,\eps}|^2(1-2\mu)\,\varsigma\,\d v \d x\\
  &\qquad+\left\langle\,\mathrm{W},\,f_{R,\eps}\,\right\rangle+\left\langle\,\widetilde{R}_\eps,\,f_{R,\eps}\,\right\rangle,
  \end{aligned}
\end{equation}
where we have used
\begin{equation*}
  \frac{1}{\eps}\iint_{\R^3\times\R^3}\L(f)f\d v \d x\geqslant \frac{\lambda_0}{\eps}\|\h f\|_{\nu(\mu)}^2,
\end{equation*}
and the notation
\begin{equation}\label{W}
  \begin{aligned}
  \mathrm{W}=&\sum_{i=1}^{3}\eps^{i-1} \Big(\,\widetilde{\Q}(f_i,f_{R,\eps})+\widetilde{\Q}(f_{R,\eps},f_i)\,\Big) +\eps\widetilde{\Q}(f_{R,\eps},f_{R,\eps})\\
  &-\sum_{i+j=1,i,j\leqslant 3}^{6}\eps^{i+j-1}\Big(\,\widetilde{\T}(f_{R,\eps},f_i,f_j)+\widetilde{\T}(f_i,f_{R,\eps},f_j)+\widetilde{\T}(f_i,f_j,f_{R,\eps})\,\Big)\\
  &\quad-\sum_{i=0}^{3}\eps^{i+2}\Big(\,\widetilde{\T}(f_i,f_{R,\eps},f_{R,\eps})+\widetilde{\T}(f_{R,\eps},f_i,f_{R,\eps})+\widetilde{\T}(f_{R,\eps},f_{R,\eps},f_i)\,\Big) -\eps^5\widetilde{\T}(f_{R,\eps},f_{R,\eps},f_{R,\eps}).
  \end{aligned}
\end{equation}
Noting the definition of $\varsigma$ and the fact $\rho,\,T$ have positive lower bounds, the first term on the right hand side of \eqref{L2-Es-1} can be bounded by
\begin{equation}\label{L2-Es-2}
  \begin{aligned}
    &C\iint_{\R^3\times\R^3} |f_{R,\eps}|^2(1+|v|)^3\,|\nabla_x(\rho,\u,T)|\, (\mathbbm{l}_{|v|\geqslant\frac{\kappa}{\sqrt{\eps}}}+\mathbbm{l}_{|v|\leqslant\frac{\kappa}{\sqrt{\eps}}})\d v \d x\\
    \leqslant &C \left\|f_{R,\eps} {\color{red} (1+|v|)^5 }
    \mathbbm{l}_{|v|\geqslant\frac{\kappa}{\sqrt{\eps}}}\right\|_\infty \cdot \iint_{\R^3\times\R^3}|f_{R,\eps}|\,|\nabla_x(\rho,\u,T)|(1+|v|)^{-2}\d v \d x\\
    &\qquad+C|\nabla_x(\rho,\u,T)|_{L_x^\infty}\iint_{\R^3\times\R^3} \left(|\P f_{R,\eps}|^2+|\h f_{R,\eps}|^2\right)(1+|v|^3) \mathbbm{l}_{|v|\leqslant\frac{\kappa}{\sqrt{\eps}}}\d v \d x\\
    \leqslant &C \left\|f_{R,\eps} {\color{red} (1+|v|)^5 } \mathbbm{l}_{|v|\geqslant\frac{\kappa}{\sqrt{\eps}}}\right\|_\infty \cdot \|f_{R,\eps}\|_2\left(\int_{\R^3_x}|\nabla_x(\rho,\u,T)|^2\int_{\R^3_v}(1+|v|)^{-4}\d v \d x\right)^{1/2}\\
    &\qquad+C|\nabla_x(\rho,\u,T)|_{L_x^\infty}\left(\|f_{R,\eps}\|_2^2+\frac{\kappa^2}{\eps}\|\h f_{R,\eps}\|_{\nu(\mu)}^2\right)\\
    \leqslant & C \left\| {\color{red} (1+|v|)^{5-l} } \frac{\mu_F}{\mu}h_{R,\eps}\mathbbm{l}_{|v|\geqslant\frac{\kappa}{\sqrt{\eps}}}\right\|_\infty \|f_{R,\eps}\|_2+C\left(\|f_{R,\eps}\|_2^2+\frac{\kappa^2}{\eps}\|\h f_{R,\eps}\|_{\nu(\mu)}^2\right)\\
    \leqslant& \frac{C}{\kappa^4}\eps^2 \left\|h_{R,\eps}\right\| _\infty \|f_{R,\eps}\|_2+C\|f_{R,\eps}\|_2^2+\frac{C\kappa^2}{\eps}\|\h f_{R,\eps}\|_{\nu(\mu)}^2,
  \end{aligned}
\end{equation}
for any $\kappa>0$.

Now we treat the term $\left\langle\,\mathrm{W},\,f_{R,\eps}\,\right\rangle$. From Jiang-Xiong-Zhou \cite{Jiang-Xiong-Zhou-2021}, we know that (see the estimates for $\Q_1$, $\Q_2$, $\T_7$, $\T_8$ in Lemma 2.1)
\begin{equation}\label{Q-Es}
  \begin{aligned}
  \left|\big\langle \widetilde{\Q}(f,g),\widetilde{h} \big\rangle\right| \leqslant & C\int_{\R^3_x}\left(|\nu^{1/2}(\mu)f|_{L_v^2}|g|_{L_v^2}+|f|_{L_v^2}|\nu^{1/2}(\mu)g|_{L_v^2}\right) |\nu^{1/2}(\mu)\widetilde{h}|_{L_v^2}\d x,\\
  \left|\big\langle \widetilde{\Q}(f,g),\widetilde{h} \big\rangle\right| \leqslant & C\int_{\R^3_x}\left(|\nu(\mu) f|_{L_v^2}|g|_{L_v^2}+|f|_{L_v^2}|\nu(\mu) g|_{L_v^2}\right) |\widetilde{h}|_{L_v^2}\d x,
  \end{aligned}
\end{equation}
and
\begin{equation}\label{T-Es}
  \begin{aligned}
  \left|\big\langle \widetilde{\T}(f,g,h),\widetilde{h} \big\rangle\right|\leqslant  & C\int_{\R^3_x}\left(|\nu^{1/2}(\mu)f|_{L_v^2}|g|_{L_v^2}+|f|_{L_v^2}|\nu^{1/2}(\mu)g|_{L_v^2}\right) |h|_{L_v^2}|\nu^{1/2}(\mu)\widetilde{h}|_{L_v^2}\d x,\\
  \left|\big\langle \widetilde{\T}(f,g,h),\widetilde{h} \big\rangle\right|\leqslant  & \int_{\R^3_x}\left(|\nu(\mu) f|_{L_v^2}|g|_{L_v^2}+|f|_{L_v^2}|\nu(\mu)g|_{L_v^2}\right)|h|_{L_v^2} |\widetilde{h}|_{L_v^2}\d x.
  \end{aligned}
\end{equation}
It should be pointed out that
\begin{equation*}
  \left\langle \widetilde{\Q}(f,g)+\widetilde{\Q}(g,f),[1,\,v,\,|v|^2]\sqrt{\mu(1-\mu)} \right\rangle=0,\quad \left\langle \widetilde{\T}(f,g,h)+\widetilde{\T}(g,f,h),[1,\,v,\,|v|^2]\sqrt{\mu(1-\mu)} \right\rangle=0,
\end{equation*}
then it follows that
\begin{equation*}
  \begin{aligned}
  &\left\langle \widetilde{\Q}(f,g)+\widetilde{\Q}(g,f),\widetilde{h}\right\rangle=\left\langle \widetilde{\Q}(f,g)+\widetilde{\Q}(g,f),\h \widetilde{h}\right\rangle,\\
  &\left\langle \widetilde{\T}(f,g,h)+\widetilde{\T}(g,f,h),\widetilde{h}\right\rangle=\left\langle \widetilde{\T}(f,g,h)+\widetilde{\T}(g,f,h),\h \widetilde{h}\right\rangle.
  \end{aligned}
\end{equation*}
For each $i=1,2$, we obtain by \eqref{Q-Es} and the statements above that
\begin{equation}\label{L2-Es-3}
  \begin{aligned}
  &\left\langle \,\widetilde{\Q}(f_i,f_{R,\eps})+\widetilde{\Q}(f_{R,\eps},f_i),\, f_{R,\eps} \, \right\rangle\\
  \leqslant & C\int_{\R^3_x}\left( \, \left|f_i\right|_{\nu(\mu)}\left|f_{R,\eps}\right|_{L_v^2} + \left|f_i\right|_{L_v^2}\left|f_{R,\eps}\right|_{\nu(\mu)}\, \right) \left|\h f_{R,\eps}\right|_{\nu(\mu)} \d x\\
  \leqslant & C \left( \, \left\|(1+|v|)^{5/2} f_i \right\|_\infty \left\|f_{R,\eps}\right\|_2  + \left\|(1+|v|)^2 f_i \right\|_\infty \left\|f_{R,\eps}\right\|_{\nu(\mu)} \, \right) \left\|\h f_{R,\eps}\right\|_{\nu(\mu)}\\
  \leqslant & C \left\|(1+|v|)^{5/2} f_i \right\|_\infty \left( \,\left\|\P f_{R,\eps}\right\|_{\nu(\mu)} + \left\|\h f_{R,\eps}\right\|_{\nu(\mu)}\, \right) \left\|\h f_{R,\eps}\right\|_{\nu(\mu)}\\
  \leqslant & \frac{\lambda_0}{n\eps} \left\|\h f_{R,\eps}\right\|_{\nu(\mu)}^2 + C\eps\left\|(1+|v|)^{5/2} f_i \right\|_\infty^2 \left\| f_{R,\eps}\right\|_2^2 + C \left\|(1+|v|)^{5/2} f_i \right\|_\infty \left\|\h f_{R,\eps}\right\|_{\nu(\mu)}^2\\
  \leqslant & (1+\frac{Cn}{\lambda_0}\eps)\frac{\lambda_0}{n\eps}\left\|\h f_{R,\eps}\right\|_{\nu(\mu)}^2+ C \eps\left\| f_{R,\eps}\right\|_2^2,
  \end{aligned}
\end{equation}
where we have used the Cauchy-Schwarz's inequality with $\eps$ and the positive constant $n$ is to be determined. Similarly,
\begin{equation}\label{L2-Es-4}
   \eps^2\left\langle \,\widetilde{\Q}(f_{R,\eps},f_{R,\eps}) ,\, f_{R,\eps} \, \right\rangle \leqslant
   C\eps^2\|\nu(\mu)(1+|v|)^2 f_{R,\eps}\|_\infty\|f_{R,\eps}\|_2^2=C\sqrt{\eps}\|\eps^{3/2} h_{R,\eps}\|_\infty\|f_{R,\eps}\|_2^2.
\end{equation}
For each integers $i,\,j$ with $1 \leqslant i+j \leqslant 4$, we obtain by \eqref{T-Es} that
\begin{equation}\label{L2-Es-5}
  \begin{aligned}
   &\left\langle\, \widetilde{\T}(f_{R,\eps},f_i,f_j)+\widetilde{\T}(f_i,f_{R,\eps},f_j) ,\, f_{R,\eps} \,\right\rangle\\
  \leqslant & C\int_{\R^3_x}\left( \, \left|f_i\right|_{\nu(\mu)}\left|f_{R,\eps}\right|_{L_v^2} + \left|f_i\right|_{L^2_\mu }\left|f_{R,\eps}\right|_{\nu(\mu)}\, \right) \left|f_j\right|_{L_v^2} \left|\h f_{R,\eps}\right|_{\nu(\mu)} \d x\\
  \leqslant & C \left\|(1+|v|)^{5/2} f_i \right\|_\infty \left|(1+|v|)^2 f_j \right\|_\infty  \left( \,\left\|\P f_{R,\eps}\right\|_{\nu(\mu)} + \left\|\h f_{R,\eps}\right\|_{\nu(\mu)}\, \right) \left\|\h f_{R,\eps}\right\|_{\nu(\mu)}\\
  \leqslant & \frac{\lambda_0}{n\eps} \left\|\h f_{R,\eps}\right\|_{\nu(\mu)}^2 + C \eps\left\|(1+|v|)^{5/2} f_i \right\|_\infty^2 \left\|(1+|v|)^2 f_j \right\|_\infty^2 \left\| f_{R,\eps}\right\|_2^2\\
  &\qquad+ C \left\|(1+|v|)^{5/2} f_i \right\|_\infty \left\|(1+|v|)^2 f_j \right\|_\infty \left\|\h f_{R,\eps}\right\|_{\nu(\mu)}^2\\
  \leqslant & (1+\frac{Cn}{\lambda_0}\eps)\frac{\lambda_0}{n\eps}\left\|\h f_{R,\eps}\right\|_{\nu(\mu)}^2+ C \eps\left\| f_{R,\eps}\right\|_2^2.
  \end{aligned}
\end{equation}
Similarly, we can also get
\begin{equation}\label{L2-Es-6}
  \begin{aligned}
   &\left\langle\, \widetilde{\T}(f_i,f_j,f_{R,\eps}) ,\, f_{R,\eps} \,\right\rangle\\
  \leqslant & C\left\|(1+|v|)^{5/2} f_i \right\|_\infty\left\|(1+|v|)^{5/2} f_j \right\|_\infty \iint_{\R^3\times\R^3}|f_{R,\eps}|_{L_v^2}|\h f_{R,\eps}|_{\nu(\mu)}\d v\d x\\
  \leqslant & \frac{\lambda_0}{n\eps}\left\|\h f_{R,\eps}\right\|_{\nu(\mu)}^2+ C \eps\left\| f_{R,\eps}\right\|_2^2.
  \end{aligned}
\end{equation}

\begin{equation}\label{L2-Es-7}
  \begin{aligned}
   &\eps^2\left\langle\,\widetilde{\T}(f_i,f_{R,\eps},f_{R,\eps}) +\widetilde{\T}(f_{R,\eps},f_i,f_{R,\eps})+\widetilde{\T}(f_{R,\eps},f_{R,\eps},f_i) ,\, f_{R,\eps} \,\right\rangle\\
  \leqslant & C\eps^2 \left\|(1+|v|)^3 f_i \right\|_\infty \left\|(1+|v|)^3 f_{R,\eps} \right\|_\infty \|f_{R,\eps}\|_2^2\\
  \leqslant & C\sqrt{\eps}\|\eps^{3/2}h_{R,\eps}\|_\infty\|f_{R,\eps}\|_2^2.
  \end{aligned}
\end{equation}
and
\begin{equation}\label{L2-Es-8}
   \begin{aligned}
   &\eps^5\left\langle\,\widetilde{\T}(f_{R,\eps},f_{R,\eps},f_{R,\eps}) ,\, f_{R,\eps} \,\right\rangle \leqslant C\eps^5 \iint_{\R^3\times\R^3}|\nu(\mu)f_{R,\eps}|_{L_v^2}|f_{R,\eps}|_{L_v^2}^3\d v\d x\\
   &\qquad\quad\leqslant C\eps^5 \left\|(1+|v|)^3 f_{R,\eps} \right\|_\infty^2 \|f_{R,\eps}\|_2^2 \leqslant C\eps^2 \left\|\eps^{3/2}h_{R,\eps} \right\|_\infty^2 \|f_{R,\eps}\|_2^2.
   \end{aligned}
\end{equation}
Noting \eqref{W}, we collect the estimates \eqref{L2-Es-3}-\eqref{L2-Es-8} to obtain
\begin{equation}\label{L2-W-f}
  \begin{aligned}
  \left|\left\langle \mathrm{W},f_{R,\eps} \right\rangle\right|\leqslant & 3(1+\frac{Cn}{\lambda_0}\eps)\frac{\lambda_0}{n\eps}\|\h f_{R,\eps}\|_{\nu(\mu)}^2\\
  &+\left( C+3C\eps+2C\sqrt{\eps}\|\eps^{3/2} h_{R,\eps}\|_\infty+C\eps^2 \left\|\eps^{3/2}h_{R,\eps} \right\|_\infty^2\right)\|f_{R,\eps}\|_2^2
  \end{aligned}
\end{equation}

For $\dis\widetilde{R}_\eps=\sum_{\substack{i+j\geqslant 4\\i,j\leqslant 3}}\eps^{i+j-4}\widetilde{\Q}(f_i,f_j)-\sum_{\substack{i+j+k\geqslant 4\\i,j,k\leqslant 3}}\eps^{i+j+k-4}\widetilde{\T}(f_i,f_j,f_k)-\frac{(\partial_t +v\cdot \nabla_x)F_3}{\sqrt{\mu(1-\mu)}}$, using \eqref{Q-Es} and \eqref{T-Es} again to get
\begin{equation*}
    \begin{aligned}
    \left\langle\, \widetilde{\Q}(f_i,f_j) ,\, f_{R,\eps} \,\right\rangle \leqslant & C\|f_{R,\eps}\|_2, \qquad \left\langle\, \widetilde{\T}(f_i,f_j,f_k) ,\, f_{R,\eps} \,\right\rangle \leqslant C\|f_{R,\eps}\|_2\\
    \left\langle\, \frac{(\partial_t +v\cdot \nabla_x)F_2}{\sqrt{\mu(1-\mu)}} ,\, f_{R,\eps} \,\right\rangle =& \iint_{\R^3 \times \R^3} (\partial_t +v\cdot \nabla_x)\Big[ \sqrt{\mu(1-\mu)} \h f_2 \Big] f_{R,\eps} \d v \d x\\
    &\leqslant C\left\| f_{R,\eps}\right\|_2.
    \end{aligned}
\end{equation*}
Therefore,
\begin{equation}\label{L2-R-f}
  \left|\langle \widetilde{R}_\eps, f_{R,\eps} \rangle\right|\leqslant C\|f_{R,\eps}\|_2.
\end{equation}

In summary, plugging \eqref{L2-Es-2}, \eqref{L2-W-f} and \eqref{L2-R-f} into \eqref{L2-Es-1} and taking $n=5$, $\kappa=\sqrt{\frac{\lambda_0}{10C}}$ and $\eps\leqslant \min\{\frac{\lambda_0}{20C},\,1\}=\eps_1$, we conclude that
\begin{equation*}
  \begin{aligned}
  \frac{1}{2}\frac{\d}{\d t}\|f_{R,\eps}\|_2^2 & + \frac{\lambda_0}{2\eps}\|\h f_{R,\eps}\|_{\nu(\mu)}^2\leqslant C \sqrt{\eps}\left\|\eps^{3/2} h_{R,\eps}\right\| _\infty \|f_{R,\eps}\|_2\\
  &+C\left(\eps+\sqrt{\eps}\|\eps^{3/2} h_{R,\eps}\|_\infty+\eps^2 \left\|\eps^{3/2}h_{R,\eps} \right\|_\infty^2\right)\|f_{R,\eps}\|_2^2+C\left(\|f_{R,\eps}\|_2+\|f_{R,\eps}\|_2^2\right)\\
  &\qquad\leqslant C\left(1+\sqrt{\eps}\left\|\eps^{3/2} h_{R,\eps}\right\| _\infty\right)\left(\|f_{R,\eps}\|_2+\|f_{R,\eps}\|_2^2\right)+C\eps^2 \left\|\eps^{3/2}h_{R,\eps} \right\|_\infty^2\|f_{R,\eps}\|_2^2.
  \end{aligned}
\end{equation*}

\rightline{$\qed$}

\subsection{$L^\infty$-Estimate: Proof of Lemma \ref{Lem-L-inf-Es}.}
Currently, we are going to establish the $L^\infty$-estimate for $h_{R,\eps}$. In terms of $\mu_F$ defined by \eqref{mu-F}, we define
\begin{equation*}
  \begin{aligned}
  \L_F g = &-\frac{1}{\sqrt{\mu_F(1-\mu_F)}}\Big[\,\Q(\mu,\sqrt{\mu_F(1-\mu_F)}g)+\Q(\sqrt{\mu_F(1-\mu_F)}g,\mu)\\ &\qquad-\T(\sqrt{\mu_F(1-\mu_F)}g,\mu,\mu)-\T(\mu,\sqrt{\mu_F(1-\mu_F)}g,\mu)-\T(\mu,\mu,\sqrt{\mu_F(1-\mu_F)}g)\,\Big]\\
  =&\nu(\mu) g-\K^F g,
  \end{aligned}
\end{equation*}
and $\K^F g=\K_2^F g-\K_1^F g$ with
\begin{equation*}
  \begin{aligned}
  \K_1^F g=&\iint_{\R^3\times\S^2} \frac{|(v-v_*)\cdot\omega|\N}{\sqrt{\mu_F(1-\mu_F)}}\frac{\sqrt{(\mu_F)_*(1-(\mu_F)_*)}}{\mu_*(1-\mu_*)}g_* \d \omega \d v_*,\\
  \K_2^F g=&\iint_{\R^3\times\S^2} \frac{|(v-v_*)\cdot\omega|\N}{\sqrt{\mu_F(1-\mu_F)}}\left(\frac{\sqrt{(\mu_F)_*^\prime(1-(\mu_F)_*^\prime)}} {\mu_*^\prime(1-\mu_*^\prime)}g_*^\prime +\frac{\sqrt{\mu_F^\prime(1-\mu_F^\prime)}}{\mu^\prime(1-\mu^\prime)}g^\prime
  \right) \d \omega \d v_*.
  \end{aligned}
\end{equation*}
\begin{lemma}
  It holds that $\dis(\K^F g)(v)=\int_{\R^3}l(v,v_1)g(v_1)\d v_1$, where the kernel $l(v,v_1)$ satisfies
  \begin{equation}\label{l-Es}
    |l(v,v_1)|\leqslant \frac{C}{|v-v_1|}e^{-c|v-v_1|^2},
  \end{equation}
  for some $c>0$.
\end{lemma}
\begin{proof}
  Obviously, $\dis \K_1^F g=\int_{\R^3}l_1(v,v_1)g(v_1)\d v_1$ with
  \begin{equation*}
    l_1(v,v_1)=\int_{\S^2}\frac{|(v-v_1)\cdot\omega|\N}{\sqrt{\mu_F(1-\mu_F)}} \frac{\sqrt{(\mu_F)_1(1-(\mu_F)_1)}}{\mu_1(1-\mu_1)}\d \omega,
  \end{equation*}
  where $(\mu_F)_1=\mu_F(v_1)$. It follows immediately by \eqref{mu-mu-F} and \eqref{mu-Es} that
  \begin{equation}\label{l-1}
    |l_1(v,v_1)|\,\leqslant\, C|v-v_1|\mu_F^{\alpha-1/2}(v)\mu_F(v_1)\,\leqslant\, \frac{C}{|v-v_1|}e^{-c|v-v_1|^2}.
  \end{equation}

  For $\K_2^F$, the Proposition 2.2 in \cite{Jiang-Xiong-Zhou-2021} indicates that
  \begin{equation*}
    \begin{aligned}
    \K_2^F g(v)=&2\iint_{\R^3\times\S^2} \frac{|(v-v_*)\cdot\omega|\N}{\sqrt{\mu_F(1-\mu_F)}}\frac{\sqrt{\mu_F^\prime(1-\mu_F^\prime)}} {\mu^\prime(1-\mu^\prime)}g^\prime \d \omega \d v_*\\
    =&2\iint_{\R^3\times\S^2} |(v-v_*)\cdot\omega|\mu_*^\prime \frac{\sqrt{\mu_F^\prime}}{\sqrt{\mu_F}}\frac{\sqrt{1-\mu_F^\prime}(1-\mu)(1-\mu_*)}{\sqrt{1-\mu_F}(1-\mu')}g' \d \omega \d v_*,
    \end{aligned}
  \end{equation*}
  Noting that $1-\mu$ and $1-\mu_F$ are both bounded and
  \begin{equation*}
    \mu(v)\sim e^{-\frac{|v-\u|^2}{2T}},\quad \mu_F(v) \sim e^{-\frac{|v|^2}{2T_m}},
  \end{equation*}
  then by a similar argument as Lemma 2.3 in \cite{GJJ-2010CPAM}, we can write  $\K_2^F g$ as
  \begin{equation*}
    \K_2^F g(v)=\int_{\R^3}l_2(v,v_1)g(v_1)\d v_1,
  \end{equation*}
  with
  \begin{equation}\label{l-2}
    \begin{aligned}
    |l_2(v,v_1)|\leqslant\frac{C}{|v_1-v|}e^{-c|v_1-v|^2}.
    \end{aligned}
  \end{equation}
  The estimates \eqref{l-1} and \eqref{l-2} implies \eqref{l-Es}.
\end{proof}

\noindent{\bf Proof of Lemma \ref{Lem-L-inf-Es}.} Setting $\K^F_wg=w\K^F(\frac{g}{w})$. Plugging \eqref{Reminder-h} into \eqref{Rem-eq} to yields
\begin{equation}\label{Rem-h-eq}
  \begin{aligned}
  \partial_t h_{R,\eps}+ & v\cdot\nabla_x h_{R,\eps}+\frac{\nu(\mu)}{\eps}h_{R,\eps}-\frac{1}{\eps} \K^F_w h_{R,\eps}\\ &=\frac{\eps^2 w}{\sqrt{\mu_F(1-\mu_F)}} \mathbb{Q}_1 +\sum_{i=1}^{3}\frac{\eps^{i-1}w}{\sqrt{\mu_F(1-\mu_F)}}\mathbb{Q}_2^i \\
  &\qquad-\sum_{\substack{i+j=1\\i,j \leqslant 3}}^{6}\frac{\eps^{i+j-1} w}{\sqrt{\mu_F(1-\mu_F)}}\mathbb{T}_1^{ij} -\sum_{i=0}^{3}\frac{\eps^{i+2} w}{\sqrt{\mu_F(1-\mu_F)}} \mathbb{T}_2^{i} -\frac{\eps^5 w}{\sqrt{\mu_F(1-\mu_F)}}\mathbb{T}_3 +\overline{R}_\eps,
  \end{aligned}
\end{equation}
where
\begin{equation*}
  \begin{aligned}
  \mathbb{Q}_1=&\Q\left(\tfrac{\sqrt{\mu_F(1-\mu_F)}}{w}h_{R,\eps},\tfrac{\sqrt{\mu_F(1-\mu_F)}}{w}h_{R,\eps}\right)\\
  \mathbb{Q}_2^i=&\left\{\,\Q\left(F_i,\tfrac{\sqrt{\mu_F(1-\mu_F)}}{w}h_{R,\eps}\right)+\Q\left(\tfrac{\sqrt{\mu_F(1-\mu_F)}}{w}h_{R,\eps},F_i\right)\,\right\},\quad i=1,2,3\\
  \mathbb{T}_1^{ij}=&\bigg\{\,\T\left(\tfrac{\sqrt{\mu_F(1-\mu_F)}}{w}h_{R,\eps},F_i,F_j\right) +\T\left(F_i,\tfrac{\sqrt{\mu_F(1-\mu_F)}}{w}h_{R,\eps},F_j\right)+\T\left(F_i,F_j,\tfrac{\sqrt{\mu_F(1-\mu_F)}}{w}h_{R,\eps}\right)\,\bigg\}\\
  &\hspace{10cm}i,j=0,\cdots,3, \, 1\leqslant i+j\leqslant 6\\
  \mathbb{T}_2^{i}=&\bigg\{\,\T\left(F_i,\tfrac{\sqrt{\mu_F(1-\mu_F)}}{w}h_{R,\eps},\tfrac{\sqrt{\mu_F(1-\mu_F)}}{w}h_{R,\eps}\right) +\T\left(\tfrac{\sqrt{\mu_F(1-\mu_F)}}{w}h_{R,\eps},F_i,\tfrac{\sqrt{\mu_F(1-\mu_F)}}{w}h_{R,\eps}\right)\\ &\hspace{6cm}+\T\left(\tfrac{\sqrt{\mu_F(1-\mu_F)}}{w}h_{R,\eps},\tfrac{\sqrt{\mu_F(1-\mu_F)}}{w}h_{R,\eps},F_i\right)\,\bigg\},\quad i=0,\cdots,3\\
  \mathbb{T}_3=&\T\left(\tfrac{\sqrt{\mu_F(1-\mu_F)}}{w}h_{R,\eps},\tfrac{\sqrt{\mu_F(1-\mu_F)}}{w}h_{R,\eps},\tfrac{\sqrt{\mu_F(1-\mu_F)}}{w}h_{R,\eps}\right) \\
  \overline{R}_\eps=&\sum_{\substack{i+j\geqslant 4\\i,j\leqslant 3}}\frac{\eps^{i+j-4} w}{\sqrt{\mu_F(1-\mu_F)}} \Q(F_i,F_j)-\sum_{\substack{i+j+k\geqslant 4\\i,j,k\leqslant 3}}\eps^{i+j+k-4}\T(F_i,F_j,F_k)-\frac{w(\partial_t +v\cdot \nabla_x)F_2}{\sqrt{\mu_F(1-\mu_F)}}.
  \end{aligned}
\end{equation*}
By Duhamel's principle, we can formulate the solution of \eqref{Rem-h-eq} as
\begin{equation}\label{Rem-h-sol}
  \begin{aligned}
  h_{R,\eps}&(t,x,v)\\
  &=\exp\left\{-\frac{1}{\eps}\int_{0}^{t}\nu(\tau)\d\tau\right\}h_{R,\eps}(0,x-vt,v)\\
  &\quad+\int_{0}^{t}\exp\left\{-\frac{1}{\eps}\int_{s}^{t}\nu(\tau)\d\tau\right\}(\tfrac{\K^F_w h_{R,\eps}}{\eps})(s,x-v(t-s),v)\d s\\
  &\qquad + \int_{0}^{t}\exp\left\{-\frac{1}{\eps}\int_{s}^{t}\nu(\tau)\d\tau\right\}\left(\tfrac{\eps^2 w}{\sqrt{\mu_F(1-\mu_F)}} \mathbb{Q}_1\right)(s,x-v(t-s),v)\d s\\
  &\qquad\quad +\int_{0}^{t}\exp\left\{-\frac{1}{\eps}\int_{s}^{t}\nu(\tau)\d\tau\right\} \left(\sum_{i=1}^{3}\tfrac{\eps^{i-1}w}{\sqrt{\mu_F(1-\mu_F)}}\mathbb{Q}_2^i\right)(s,x-v(t-s),v)\d s\\
  &\qquad\qquad -\int_{0}^{t}\exp\left\{-\frac{1}{\eps}\int_{s}^{t}\nu(\tau)\d\tau\right\} \left(\sum_{i+j=1}^{6}\tfrac{\eps^{i+j-1} w}{\sqrt{\mu_F(1-\mu_F)}}\mathbb{T}_1^{ij}\right)(s,x-v(t-s),v)\d s\\
  &\qquad\qquad\quad -\int_{0}^{t}\exp\left\{-\frac{1}{\eps}\int_{s}^{t}\nu(\tau)\d\tau\right\} \left(\sum_{i=0}^{3}\tfrac{\eps^{i+2} w}{\sqrt{\mu_F(1-\mu_F)}} \mathbb{T}_2^{i}\right)(s,x-v(t-s),v)\d s\\
  &\qquad\qquad\qquad -\int_{0}^{t}\exp\left\{-\frac{1}{\eps}\int_{s}^{t}\nu(\tau)\d\tau\right\}           \left(\tfrac{\eps^5 w}{\sqrt{\mu_F(1-\mu_F)}}\mathbb{T}_3\right)(s,x-v(t-s),v)\d s\\
  &\qquad\qquad\qquad\qquad -\int_{0}^{t}\exp\left\{-\frac{1}{\eps}\int_{s}^{t}\nu(\tau)\d\tau\right\} \overline{R}_\eps(s,x-v(t-s),v)\d s.
  \end{aligned}
\end{equation}
First, we deduce from \eqref{nu-Es} that
\begin{equation}\label{nu-int-Es}
  \int_{0}^{t}\exp\left\{-\frac{1}{\eps}\int_{s}^{t}\nu(\mu)(\tau)\d\tau\right\}\nu(\mu) \d s\leqslant c \int_{0}^{t}\exp\left\{-\frac{c(1+|v|)(t-s)}{\eps}\right\}(1+|v|) \d s=O(\eps).
\end{equation}
Since $\dis\mu_F(v_*)\leqslant \mu(v_*)$ and
\begin{equation*}
  \left|\tfrac{ w}{\sqrt{\mu_F(1-\mu_F)}} \Q\left(\tfrac{\sqrt{\mu_F(1-\mu_F)}}{w}h_{R,\eps},\tfrac{\sqrt{\mu_F(1-\mu_F)}}{w}h_{R,\eps}\right)\right|\leqslant C \nu(\mu)\|h_{R,\eps}\|_\infty^2,
\end{equation*}
the third line on the right hand side of \eqref{Rem-h-sol} is bounded by
\begin{equation}\label{term-3}
  C\eps^2 \int_{0}^{t}\exp\left\{-\frac{1}{\eps}\int_{s}^{t}\nu(\mu)(\tau)\d\tau\right\}\nu(\mu) \|h_{R,\eps}(s)\|_\infty^2\d s \leqslant C\eps^3 \sup_{0\leqslant s\leqslant t}\|h_{R,\eps}(s)\|_\infty^2.
\end{equation}
Since $F_i=\tfrac{\sqrt{\mu_F(1-\mu_F)}}{w}\left(\tfrac{w}{\sqrt{\mu_F(1-\mu_F)}}F_i\right)$, we have that
\begin{equation*}
  \begin{aligned}
  &\sum_{i=1}^{3}\tfrac{\eps^{i-1}w}{\sqrt{\mu_F(1-\mu_F)}}\left\{\,\Q\left(F_i,\tfrac{\sqrt{\mu_F(1-\mu_F)}}{w}h_{R,\eps}\right)+\Q\left(\tfrac{\sqrt{\mu_F(1-\mu_F)}}{w}h_{R,\eps},F_i\right)\,\right\} \\
  &\hspace{6cm}\leqslant C \nu(\mu)\|h_{R,\eps}(s)\|_\infty \left\|\sum_{i=1}^{3}\tfrac{\eps^{i-1}w}{\sqrt{\mu_F(1-\mu_F)}}F_i\right\|_\infty.
  \end{aligned}
\end{equation*}
Then the fourth line on the right hand side of \eqref{Rem-h-sol} is bounded by
\begin{equation}\label{term-4}
  C\int_{0}^{t}\exp\left\{-\frac{1}{\eps}\int_{s}^{t}\nu(\mu)(\tau)\d\tau\right\}\nu(\mu)\|h_{R,\eps}(s)\|_\infty \left\|\sum_{i=1}^{3}\tfrac{\eps^{i-1}w}{\sqrt{\mu_F(1-\mu_F)}}F_i\right\|_\infty\d s \leqslant C\eps \sup_{0\leqslant s\leqslant t}\|h_{R,\eps}(s)\|_\infty.
\end{equation}
Similarly, the fifth to seventh lines on the right hand side of \eqref{Rem-h-sol} are respectively bounded by
\begin{equation}\label{term-5}
  \begin{aligned}
  &\textit{the fifth line on RHS }\\
  \leqslant & C \sum_{\substack{i+j=1\\i,j\leqslant 3}}^{6}\eps^{i+j-1}\int_{0}^{t}\exp\left\{-\frac{1}{\eps}\int_{s}^{t}\nu(\mu)(\tau)\d\tau\right\}\nu(\mu) \|h_{R,\eps}(s)\|_\infty\times\\
  &\hspace{3cm}\left\{\left\|\tfrac{w}{\sqrt{\mu_F(1-\mu_F)}}F_i\right\|_\infty\|F_j\|_\infty +\left\|\tfrac{w}{\sqrt{\mu_F(1-\mu_F)}}F_i\right\|_\infty\|\tfrac{w}{\sqrt{\mu_F(1-\mu_F)}}F_j\|_\infty\right\}\d s\\
  \leqslant & C \eps \sup_{0\leqslant s\leqslant t}\|h_{R,\eps}(s)\|_\infty,
  \end{aligned}
\end{equation}
\begin{equation}\label{term-6}
  \begin{aligned}
  &\textit{the sixth line on RHS }\\
  \leqslant & C \sum_{i=0}^{3}\eps^{i+2}\int_{0}^{t}\exp\left\{-\frac{1}{\eps}\int_{s}^{t}\nu(\mu)(\tau)\d\tau\right\}\nu(\mu) \|h_{R,\eps}(s)\|_\infty^2\times\\
  &\hspace{8cm}\left\{\left\|\tfrac{w}{\sqrt{\mu_F(1-\mu_F)}}F_i\right\|_\infty+\|F_i\|_\infty\right\}\d s\\
  \leqslant & C \eps^3 \sup_{0\leqslant s\leqslant t}\|h_{R,\eps}(s)\|_\infty^2,
  \end{aligned}
\end{equation}
\begin{equation}\label{term-7}
  \begin{aligned}
  \textit{the seventh line on RHS } \leqslant & C \eps^5\int_{0}^{t}\exp\left\{-\frac{1}{\eps}\int_{s}^{t}\nu(\mu)(\tau)\d\tau\right\}\nu(\mu) \|h_{R,\eps}(s)\|_\infty^3\d s\\
  \leqslant & C \eps^6 \sup_{0\leqslant s\leqslant t}\|h_{R,\eps}(s)\|_\infty^3.
  \end{aligned}
\end{equation}
The last term on the right hand side of \eqref{Rem-h-sol} is obviously bounded by $C\eps$.

Now we devote ourselves to the second term on the right hand side of \eqref{Rem-h-sol}. Let $l_w(v,v_1)$ be the corresponding kernel associated with $\K^F_w$. Recall \eqref{l-Es}, we have
\begin{equation*}
  |l_w(v,v_1)|\leqslant C\frac{w(v_1)\exp\{-c|v-v_1|^2\}}{w(v)|v-v_1|} \leqslant C\frac{\exp\{-\frac{3c}{4}|v-v_1|^2\}}{|v-v_1|},
\end{equation*}
then it follows that
\begin{equation}\label{l-w-Es}
  \int_{\R^3}|l_w(v,v_1)|\d v_1\leqslant C.
\end{equation}
The second term on the right hand side of \eqref{Rem-h-sol} is bounded  by
\begin{equation*}
  \frac{1}{\eps}\int_{0}^{t}\exp\left\{-\frac{1}{\eps}\int_{s}^{t}\nu(\tau)\d\tau\right\}
  \int_{\R^3}\left|l_w(v,v_1)h_{R,\eps}(s,x-v(t-s),v_1)\right|\d v_1 \d s.
\end{equation*}
We use \eqref{Rem-h-sol} again to evaluate $h_{R,\eps}(s,x-v(t-s),v_1)$ and hence bound the above by
\begin{equation}\label{term-2}
  \begin{aligned}
  &\frac{1}{\eps}\int_{0}^{t}\exp\left\{-\frac{1}{\eps}\int_{s}^{t}\nu(\mu)(\tau)\d\tau -\frac{1}{\eps}\int_{0}^{s}\nu(\mu)(\tau)\d\tau\right\}\times\\
  &\hspace{3cm}\left(\int_{\R^3}|l_w(v,v_1)h_{R,\eps}(0,x-v(t-s)-v_1 s,v_1)|\d v_1\right)\d s\\
  &\quad+\frac{1}{\eps^2}\int_{0}^{t}\exp\left\{-\frac{1}{\eps}\int_{s}^{t}\nu(\mu)(\tau)\d\tau\right\}\iint_{\R^3\times \R^3}|l_w(v,v_1)| |l_w(v_1,v_2)|\d v_1\d v_2 \times\\
  &\hspace{2cm}\left(\int_{0}^{s}\exp\left\{-\frac{1}{\eps}\int_{s_1}^{s}\nu(\mu)(\tau,v_1)\d\tau\right\} |h_{R,\eps}(s_1,x-v(t-s)-v_1(s-s_1),v_2)|\d s_1\right)\d s\\
  &\qquad\quad+\frac{C}{\eps}\int_{0}^{t}\exp\left\{-\frac{1}{\eps}\int_{s}^{t}\nu(\mu)(\tau)\d\tau\right\}\d s \int_{\R^3}|l_w(v,v_1)|\d v_1\times\\
  &\hspace{3.5cm}\left\{2\eps \sup_{0\leqslant s\leqslant t}\|h_{R,\eps}(s)\|_\infty+
  2\eps^3 \sup_{0\leqslant s\leqslant t}\|h_{R,\eps}(s)\|_\infty^2 +\eps^6 \sup_{0\leqslant s\leqslant t}\|h_{R,\eps}(s)\|_\infty^3+\eps \right\}
  \end{aligned}
\end{equation}

Noting the estimate \eqref{l-w-Es}, the sum of the first and the third term in \eqref{term-2} is bounded by
\begin{equation}\label{term-2-13}
  C\left(\|h_{R,\eps}(0)\|_\infty+\eps \sup_{0\leqslant s\leqslant t}\|h_{R,\eps}(s)\|_\infty+
  \eps^3 \sup_{0\leqslant s\leqslant t}\|h_{R,\eps}(s)\|_\infty^2 +\eps^6 \sup_{0\leqslant s\leqslant t}\|h_{R,\eps}(s)\|_\infty^3+\eps\right)
\end{equation}
The estimate for the second term in \eqref{term-2} need more care. For some constant $N>0$ to be determined later, we divide this term into the following several cases.

{\bf Case 1.} $|v|\geqslant N$. By \eqref{l-w-Es},
\begin{equation*}
  \int_{\R^3}|l_w(v,v_1)|\d v_1\leqslant C\frac{1+|v|}{N},\qquad \int_{\R^3}|l_w(v_1,v_2)|\d v_2\leqslant C(1+|v_1|).
\end{equation*}
Therefore in this case we have the bound
\begin{equation}\label{Case-1}
  \begin{aligned}
  &\frac{C}{\eps^2}\sup_{0\leqslant s\leqslant t}\|h_{R,\eps}(s)\|_\infty\int_{0}^{t}\exp\left\{-\frac{c\nu(\mu)(t-s)}{\eps}\right\}\frac{1+|v|}{N}\d s\\
  &\qquad\qquad\times\int_{0}^{s}\exp\left\{-\frac{c(1+|v_1|)(s-s_1)}{\eps}\right\}(1+|v_1|)\d s_1\\
  \leqslant& \frac{C}{N}\sup_{0\leqslant s\leqslant t}\|h_{R,\eps}(s)\|_\infty.
  \end{aligned}
\end{equation}

{\bf Case 2.} For either $|v|\leqslant N$, $|v_1|\geqslant 2N$ or $|v_1|\leqslant 2N$, $|v_2|\geqslant 3N$. Under the assumptions, we have either $|v_1-v|\geqslant N$ or $|v_1-v_2|\geqslant N$. We then deduce that either one of the following is valid for some small $\eta>0$:
\begin{equation}\label{}
  |l_w(v,v_1)|\leqslant e^{-\frac{\eta}{8}N^2}\left|e^{\frac{\eta}{8}|v-v_1|^2} l_w(v,v_1)\right|,\quad|l_w(v_1,v_2)|\leqslant e^{-\frac{\eta}{8}N^2}\left|e^{\frac{\eta}{8}|v_1-v_2|^2} l_w(v_1,v_2)\right|.
\end{equation}
It immediately follows by \eqref{l-w-Es} that
\begin{equation*}
  \int_{|v_1-v|\geqslant N}|l_w(v,v_1)|\d v_1\leqslant C e^{-\frac{\eta}{8}N^2},\quad \int_{|v_1-v_2|\geqslant N}|l_w(v_1,v_2)|\d v_2\leqslant C e^{-\frac{\eta}{8}N^2}.
\end{equation*}
Consequently, this case is bounded by
\begin{equation}\label{Case-2}
  \begin{aligned}
  &\frac{C}{\eps^2}\sup_{0\leqslant s\leqslant t}\|h_{R,\eps}(s)\|_\infty \int_{0}^{t}\exp\left\{-\frac{c\nu(\mu)(t-s)}{\eps}\right\} \left(\int_{0}^{s}\exp\left\{-\frac{c(1+|v_1|)(s-s_1)}{\eps}\right\}\d s_1\right)\d s\\
  &\hspace{4cm}\times\left(\iint_{|v|\leqslant N, |v_1|\geqslant 2N}+\iint_{|v_1|\leqslant 2N,|v_2|\geqslant 3N}\right)|l_w(v,v_1)| |l_w(v_1,v_2)|\d v_1\d v_2\\
  &\qquad\leqslant  C e^{-\frac{\eta}{4}N^2}\sup_{0\leqslant s\leqslant t}\|h_{R,\eps}(s)\|_\infty.
  \end{aligned}
\end{equation}

{\bf Case 3a.} $|v|\leqslant N$, $|v_1|\leqslant 2N$ and $|v_2|\leqslant 3N$. Then there hold that $\nu(v)\leqslant C_N$ and $\nu(v_1)\leqslant C_N$. We further assume that $s-s_1\leqslant \kappa_*\eps$ for $\kappa_*>0$ small enough. In this case, we have the bound
\begin{equation}\label{Case-3a}
  \begin{aligned}
  &\frac{C}{\eps^2}\sup_{0\leqslant s\leqslant t}\|h_{R,\eps}(s)\|_\infty \int_{0}^{t}\exp\left\{-\frac{C_N(t-s)}{\eps}\right\} \left(\int_{s-\kappa_*\eps}^{s}\exp\left\{-\frac{C_N(s-s_1)}{\eps}\right\}\d s_1\right)\d s\\
  &\hspace{6cm}\times\iint_{|v|\leqslant N,|v_1|\leqslant 2N,|v_2|\leqslant 3N}|l_w(v,v_1)| |l_w(v_1,v_2)|\d v_1\d v_2\\
  &\qquad\leqslant  C_N\kappa_* \sup_{0\leqslant s\leqslant t}\|h_{R,\eps}(s)\|_\infty.
  \end{aligned}
\end{equation}

{\bf Case 3b.} $|v|\leqslant N$, $|v_1|\leqslant 2N$, $|v_2|\leqslant 3N$ and $s-s_1\geqslant \kappa_*\eps$. Setting $D_N=\{\,|v|\leqslant N,\,|v_1|\leqslant 2N,\,|v_2|\leqslant 3N\,\}$. Since $\nu(\mu)(v)\leqslant C\nu(\mu)$ and
\begin{equation*}
  \int_{\R^3}|l_w(v,v_1)|^2\d v_1\leqslant C,\quad\int_{\R^3}|l_w(v_1,v_2)|^2\d v_2\leqslant C,
\end{equation*}
the second term in \eqref{term-2} can be bounded by
\begin{equation*}
  \begin{aligned}
  &\frac{1}{\eps^2}\int_{0}^{t}\int_{0}^{s-\kappa_*\eps}\exp\left\{-\frac{c(t-s)}{\eps}\right\} \exp\left\{-\frac{c(s-s_1)}{\eps}\right\}\d s_1\d s\\
  &\qquad\times\iint_{D_N}|l_w(v,v_1)| |l_w(v_1,v_2)||h_{R,\eps}(s_1,x-v(t-s)-v_1(s-s_1),v_2)|\d v_1\d v_2\\
  \leqslant& \frac{C}{\eps^2}\int_{0}^{t}\int_{0}^{s-\kappa_*\eps}\exp\left\{-\frac{c(t-s)}{\eps}\right\} \exp\left\{-\frac{c(s-s_1)}{\eps}\right\}\d s_1\d s\\
  &\qquad\qquad\times\left(\iint_{D_N}|h_{R,\eps}(s_1,x-v(t-s)-v_1(s-s_1),v_2)|^2\d v_1\d v_2\right)^{1/2},
  \end{aligned}
\end{equation*}
where we have used the Cauchy-Schwarz inequality. For $s-s_1\geqslant \kappa_*\eps$, we can obtain by changing the variable $y=x-v(t-s)-v_1(s-s_1)$,
\begin{equation*}
  \begin{aligned}
  &\left(\iint_{D_N}|h_{R,\eps}(s_1,x-v(t-s)-v_1(s-s_1),v_2)|^2\d v_1\d v_2\right)^{1/2}\\
  =&\frac{1}{\kappa_*^{3/2}\eps^{3/2}}\left(\iint_{\widetilde{D}_N}\frac{1}{\mu(v_2)(1-\mu(v_2))}|h_{R,\eps}(s_1,y,v_2)|^2\mu(v_2)(1-\mu(v_2))\d y\d v_2\right)^{1/2}\\
  \leqslant & \frac{C_N}{\kappa_*^{3/2}\eps^{3/2}}\sup_{0\leqslant s\leqslant t}\|f_{R,\eps}(s)\|_2,
  \end{aligned}
\end{equation*}
where $\widetilde{D}_N=\{\,|v|\leqslant N,\,|y-[x-v(t-s)]|\leqslant 2N(s-s_1),\,|v_2|\leqslant 3N\,\}$. Thus, in this case we can bound the second term in \eqref{term-2} by
\begin{equation}\label{Case-3b}
  \frac{C_N}{\kappa_*^{3/2}\eps^{3/2}}\sup_{0\leqslant s\leqslant t}\|f_{R,\eps}(s)\|_2.
\end{equation}

In summary, we collect \eqref{term-2-13}and all the cases \eqref{Case-1}, \eqref{Case-2}-\eqref{Case-3b} to establish that for any $\kappa_*>0$ and large $N>0$ the second term in \eqref{Rem-h-sol} is bounded by
\begin{equation*}
  \begin{aligned}
  &C\left(\|h_{R,\eps}(0)\|_\infty+\eps \sup_{0\leqslant s\leqslant t}\|h_{R,\eps}(s)\|_\infty+
  \eps^3 \sup_{0\leqslant s\leqslant t}\|h_{R,\eps}(s)\|_\infty^2 +\eps^6 \sup_{0\leqslant s\leqslant t}\|h_{R,\eps}(s)\|_\infty^3+\eps\right)\\
  &\qquad+\left(\frac{C}{N}+C e^{-\frac{\eta}{4}N^2} +C_N\kappa_*\right)\sup_{0\leqslant s\leqslant t}\|h_{R,\eps}(s)\|_\infty+\frac{C_N}{\kappa_*^{3/2}\eps^{3/2}}\sup_{0\leqslant s\leqslant t}\|f_{R,\eps}(s)\|_2.
  \end{aligned}
\end{equation*}

Finally, by taking $N=6C$, $\eta=\frac{\ln(6C)}{9C^2}$ and $\kappa_*=\frac{1}{6C_N}$, we conclude that there exists a suitable small $\eps_2 \leqslant \min\{\frac{1}{6C},\,1\} $ such that for all $\eps<\eps_2$ and $t\in [0,\tau]$,
  \begin{equation*}
    \sup_{s\in [0,t]}\|\eps^{3/2}h_{R,\eps}(s)\|_\infty \leqslant C \big(\|\eps^{3/2}h_{R,\eps}(0)\|_\infty+\sup_{s\in [0,\tau]}\|f_{R,\eps}(s)\|_2+\eps^{5/2}\big).
  \end{equation*}
\rightline{$\qed$}

\section{Acoustic Limit: Proof of Theorem \ref{Thm-AC-Limit}}
\subsection{Refined Estimates for Compressible Euler and Acoustic System}
Now we focus on the acoustic limit from the scaled BFD equation \eqref{SBFD}. From the a priori estimates giving in the Theorem \ref{Thm-Euler-limit}, we obtian the existence of solutions to the BFD equation. Furthermore, the compressible Euler limit can be immediately derived by the arguments in \cite{Caflisch-1980CPAM}.

Essentially, the acoustic system \eqref{AC-Sys} we derived formally is the linear wave equations. Therefore, for any given initial data $(\sigma^\i,u^\i,\theta^\i)\in H^s,\,s\geqslant 4$, there is a unique global-in-time classical solution $(\sigma,u,\theta)\in C\left([0,\infty);\,H^s\right)$ to the system \eqref{AC-Sys}. Moreover, the solution satisfies the following energy estimates:
\begin{equation}\label{Ac-sys-Es}
  \left\|\left(\sqrt{3/(2\overline{K}_g)}\,\sigma,u,\sqrt{3/2}\,\theta\right)\right\|_{H^s}^2(t) =\left\|\left(\sqrt{3/(2\overline{K}_g)}\,\sigma^\i,u^\i,\sqrt{3/2}\,\theta^\i\right)\right\|_{H^s}^2,
\end{equation}
where $\overline{K}_g=\overline{K}_A-1$ and $\overline{K}_A$ is given by the notations \eqref{ints}.

On the other hand, we already have the existence of the solution to the compressible Euler system \eqref{Com-Euler-NonConser}. Based on the classical result of the lifespan of symmetric hyperbolic system, we conclude that
\begin{lemma}\label{Lem-lifespan}
  For any given $(\sigma^\i,u^\i,\theta^\i)\in H^s$ with $s\geqslant 4$, let
  \begin{equation*}
    \widehat{\rho}^\i=\overline{E}_0 \sigma^\i+\left(\frac{3\overline{E}_2}{2}-\overline{K}_g \overline{E}_0\right)\theta^\i,\qquad\e^\i=\overline{E}_2(\sigma^\i+\theta^\i).
  \end{equation*}
  Consider the non-conservative compressible Euler system \eqref{Com-Euler-NonConser} with the initial data
  \begin{equation}\label{In-data-Euler}
    \rho^\i=\overline{\rho}+\delta \widehat{\rho}^\i,\quad \u^\i=\delta u^\i, \quad \E^\i=\overline{\E}+\delta\e^\i,
  \end{equation}
  choose $\delta_1>0$ small enough so that for each $0< \delta \leqslant \delta_1$, $\rho^\i$ and $\E^\i$ are positive. Then there is a finite time $\tau^\delta>0$ for each $\delta$ such that the system \eqref{Com-Euler-NonConser} admits a unique classical solution $(\rho^\delta-\overline{\rho},\u^\delta,\E^\delta-\overline{\E})\in C\left([0,\tau^\delta];\,H^s\right)\cap C^1\left([0,\tau^\delta];\,H^{s-1}\right)$ with $\rho^\delta>0$, $\E^\delta>0$ and the following estimates:
  \begin{equation}\label{}
    \|(\rho^\delta-\overline{\rho},\u^\delta,\E^\delta-\overline{\E})\|_{C\left([0,\tau^\delta];\,H^s\right)\cap C^1\left([0,\tau^\delta];\,H^{s-1}\right)}\leqslant C_0.
  \end{equation}
  Furthermore, the lifespan $\tau^\delta$ has the lower bound:
  \begin{equation*}
    \tau^\delta>\frac{C_1}{\delta}.
  \end{equation*}
  Here, the constants $C_0>0$ and $C_1>0$ are independent of $\delta$, depending only on the $H^s$ norm of $(\sigma^\i,u^\i,\theta^\i)$.
\end{lemma}
We refer to the literatures \cite{Friedrichs-1954CPAM,Kato-1975ARMA,Majda-1984Book} for the proof of Lemma \ref{Lem-lifespan}.

\begin{remark}
  Note that the lifespan of the acoustic system \eqref{AC-Sys} is $[0,+\infty)$, while those of the compressible Euler system \eqref{Com-Euler-NonConser} is finite. However, the lifespan of the latter can be understood to be arbitrary in the following sense. For any given $\tau>0$ and $(\sigma^\i,u^\i,\theta^\i)\in H^s$, we define
  \begin{equation}\label{delta-1}
    \delta_1=\frac{C_1}{\tau},
  \end{equation}
  then for any $0<\delta\leqslant \delta_1$, the lifespan $\tau^\delta$ obtained in Lemma \ref{Lem-lifespan} satisfies
  \begin{equation*}
    \tau^\delta> \frac{C_1}{\delta}>\frac{C_1}{\delta_1}=\tau.
  \end{equation*}
\end{remark}

From right now, all our statements are on an arbitrary time interval $[0,\tau]$ and we fix $\delta_1>0$ as in \eqref{delta-1}. In order to derive a refined estimate of two solutions to compressible Euler system and acoustic system, we shall introduce the following difference variables $(\rho_d^\delta,u_d^\delta,\E_d^\delta)$ that are given by the second-order perturbation in $\delta$ of Euler solutions:
\begin{equation}\label{diff-2-order}\left\{
  \begin{aligned}
  \delta^2 \rho_d^\delta=&\rho^\delta-\overline{\rho}-\delta\left(\overline{E}_0 \sigma+\left(\frac{3\overline{E}_2}{2}-\overline{K}_g \overline{E}_0\right)\theta\right)\\
  \delta^2 u_d^\delta =& \u^\delta-\delta u\\
   \delta^2 \E_d^\delta=& \E^\delta-\overline{\E}-\delta \overline{E}_2(\sigma+\theta).
  \end{aligned}\right.
\end{equation}

\begin{lemma}\label{Lem-diff-2-order}
  Let $\tau>0$ and $s \geqslant 3$. Assume $(\sigma,u,\theta)$ is the solution to the acoustic system \eqref{AC-Sys} with initial data $(\sigma^\i,u^\i,\theta^\i)\in H^{s+1}$ and $(\rho^\delta,\u^\delta,\E^\delta)$ is the solution to the compressible Euler system \eqref{Com-Euler-NonConser} constructed in Lemma \ref{Lem-lifespan}. Then for all $0<\delta\leqslant \delta_1$, there exists a constant $C_2>0$ depending on $\tau$ and $\|(\sigma^\i,u^\i,\theta^\i)\|_{H^{s+1}}$ such that
  \begin{equation}\label{diff-2-order-Es}
    \left\|(\rho_d^\delta,u_d^\delta,\E_d^\delta)\right\|_{H^s}\leqslant C_2.
  \end{equation}
\end{lemma}

\begin{remark}
  From Lemma \ref{Lem-diff-2-order}, we deduce immediately that the acoustic system is the linearization about the constant state $(\overline{\rho},0,\overline{\E})$ of the compressible Euler system:
  \begin{equation}\label{diff-2-order-Es-1}
    \sup_{t\in[0,\tau]}\left\|\left(\rho^\delta-\overline{\rho}-\delta\left(\overline{E}_0 \sigma+\left(3\overline{E}_2/2-\overline{K}_g \overline{E}_0\right)\theta\right),\u^\delta-\delta u, \E^\delta-\overline{\E}-\delta \overline{E}_2(\sigma+\theta)\right)\right\|_{H^s}\leqslant C_2\delta^2.
  \end{equation}
  Moreover, the Sobolev embedding $H^3(\R^3)\hookrightarrow C^1(\R^3)$ indicates that we can acquire the uniform pointwise estimates of the difference variables $(\rho_d^\delta,u_d^\delta,\E_d^\delta)$.
\end{remark}

\begin{proof}[{\bf Proof of Lemma \ref{Lem-diff-2-order}}]
  Inserting \eqref{diff-2-order} into the non-conservative compressible Euler system \eqref{Com-Euler-NonConser} to yield
  \begin{equation*}
    \begin{aligned}
    &\partial_t\left\{\delta\left(\overline{E}_0 \sigma+\left(3\overline{E}_2/2-\overline{K}_g \overline{E}_0\right)\theta\right)+\delta^2 \rho_d^\delta \right\}\\
    &\qquad+(\delta u+\delta^2 u_d^\delta)\cdot\nabla_x\left\{\delta\left(\overline{E}_0 \sigma+\left(3\overline{E}_2/2-\overline{K}_g \overline{E}_0\right)\theta\right)+\delta^2 \rho_d^\delta \right\}\\
    &\qquad\qquad+\left\{\overline{\rho}+\delta\left(\overline{E}_0 \sigma+\left(3\overline{E}_2/2-\overline{K}_g \overline{E}_0\right)\theta\right)+\delta^2 \rho_d^\delta \right\}\nabla_x\cdot(\delta u+\delta^2 u_d^\delta)=0\\
    &\left\{\overline{\rho}+\delta\left(\overline{E}_0 \sigma+\left(3\overline{E}_2/2-\overline{K}_g \overline{E}_0\right)\theta\right)+\delta^2 \rho_d^\delta \right\}\partial_t(\delta u+\delta^2 u_d^\delta)\\
    &\qquad+\left\{\overline{\rho}+\delta\left(\overline{E}_0 \sigma+\left(3\overline{E}_2/2-\overline{K}_g \overline{E}_0\right)\theta\right)+\delta^2 \rho_d^\delta \right\}(\delta u+\delta^2 u_d^\delta)\cdot\nabla_x(\delta u+\delta^2 u_d^\delta)\\
    &\qquad\qquad+ \nabla_x(\delta \overline{E}_2(\sigma+\theta)+\delta^2 \E_d^\delta)=0\\
    &\partial_t (\delta \overline{E}_2(\sigma+\theta)+\delta^2 \E_d^\delta)+(\delta u+\delta^2 u_d^\delta)\cdot\nabla_x(\delta \overline{E}_2(\sigma+\theta)+\delta^2 \E_d^\delta)\\
    &\qquad\qquad+ \frac{5}{3}(\overline{\E}+\delta \overline{E}_2(\sigma+\theta)+\delta^2 \E_d^\delta)\nabla_x\cdot(\delta u+\delta^2 u_d^\delta)=0.
    \end{aligned}
  \end{equation*}
  Noting that $(\sigma, u, \theta)$ satisfy the acoustic system \eqref{AC-Sys}. We can deduce that the coefficients of $\delta$ in the equations above are all equal to $0$:
   \begin{equation*}\left\{
     \begin{aligned}
     &\partial_t\left\{\left(\overline{E}_0 \sigma+\left(3\overline{E}_2/2-\overline{K}_g \overline{E}_0\right)\theta\right)\right\}+\overline{\rho}\nabla_x\cdot u=0,\\
     &\overline{\rho}\partial_t u +\overline{E}_2\nabla_x (\sigma+\theta)=0,\\
     &\overline{E}_2\partial_t (\sigma+\theta)+\frac{5}{3}\overline{\E}\nabla_x\cdot u=0,
     \end{aligned}\right.
   \end{equation*}
   since we have $\overline{\rho}=\overline{E}_2$ and $\overline{K}_A=\frac{5\overline{\E}}{2\overline{\rho}}$. Therefore the remaining terms which are the coefficients of $\delta^2$ form the equations:
   \begin{equation}\label{diff-2-order-Eq}
     \begin{aligned}
   &\partial_t\rho_d^\delta+\u^\delta\cdot\nabla_x \rho_d^\delta+\rho^\delta\nabla_x\cdot u_d^\delta+\delta\left\{u_d^\delta\cdot\nabla_x\left[\overline{E}_0 \sigma+\left(3\overline{E}_2/2-\overline{K}_g \overline{E}_0\right)\theta \right]+\rho_d^\delta\nabla_x\cdot u\right\}\\
   &\hspace{7cm}=-\nabla_x\cdot\Big\{\left[\,\overline{E}_0 \sigma+\left(3\overline{E}_2/2-\overline{K}_g \overline{E}_0\right)\theta\,\right]u\Big\}\\
   &\rho^\delta\partial_t u_d^\delta+\rho^\delta \u^\delta\cdot\nabla_x u_d^\delta+\nabla_x\E_d^\delta +\delta\left( \rho_d^\delta\partial_t  u+\rho^\delta u_d^\delta\cdot\nabla_x u \right) \\
   &\hspace{7cm}=-\left[\overline{E}_0 \sigma+\left(3\overline{E}_2/2-\overline{K}_g \overline{E}_0\right)\theta\right]\partial_t  u-\rho^\delta u\cdot\nabla_x u\\
   &\partial_t \E_d^\delta+\u^\delta\cdot\nabla_x\E_d^\delta+\frac{5}{3}\E^\delta\nabla_x\cdot u_d^\delta +\delta\left[\overline{E}_2 u_d^\delta\cdot\nabla_x(\sigma+\theta)+\frac{5}{3}\E_d^\delta\nabla_x\cdot u\right]\\
   &\hspace{7cm}=-\overline{E}_2 u\cdot(\sigma+\theta)-\frac{5}{3}\overline{E}_2\nabla_x(\sigma+\theta)\nabla_x\cdot u.
   \end{aligned}
   \end{equation}

   Indeed, the equations above are the system governed by the difference variables $(\rho_d^\delta,u_d^\delta,\E_d^\delta)$ with the coefficients depending on the smooth Euler solution $(\rho^\delta,\u^\delta,\E^\delta)$ and smooth acoustic solution $(\sigma,u,\theta)$ over time integral $[0,\tau]$. Then going a step further, we can formulate \eqref{diff-2-order-Eq} as a symmetric hyperbolic system:
   \begin{equation}\label{U-d}
     A_0^\delta \partial_t U_d +\sum_{i=1}^{3}A_i^\delta \partial_i U_d +B^\delta U_d = F,
   \end{equation}
   where $A_0^\delta$, $A_i^\delta$, $B^\delta$ and $F$ are given by
   \begin{equation*}
     U_d=\left(\begin{aligned}\E_d^\delta \\ \u_d^\delta \\ \rho_d^\delta\end{aligned}\right),\quad
     A_0^\delta=\left(\begin{array}{ccc} \frac{6}{5}& 0 &-\frac{\E^\delta}{\rho^\delta}\\ 0 & \rho^\delta\E^\delta \mathbb{I} & 0 \\ -\frac{\E^\delta}{\rho^\delta} & 0 & \frac{5(\E^\delta)^2}{3(\rho^\delta)^2}\end{array}\right),\quad
     A_i^\delta=\left(\begin{array}{ccc} \frac{6}{5}\u_i^\delta & \E^\delta e_i &  -\frac{\E^\delta}{\rho^\delta}\u_i^\delta \\ \E^\delta (e_i)^\mathrm{T} &  \rho^\delta \E^\delta \u_i^\delta \mathbb{I} & 0 \\ -\frac{\E^\delta}{\rho^\delta}\u_i^\delta & 0 & \frac{5(\E^\delta)^2}{3(\rho^\delta)^2}\u_i^\delta \end{array}\right),
   \end{equation*}
   \begin{equation*}
     B^\delta=\delta A_0^\delta\left(\begin{array}{ccc}\frac{5}{3}\nabla_x\cdot u & \overline{E}_2 \nabla_x(\sigma+\theta) & 0 \\ 0 & \rho^\delta\nabla_x u & \partial_t  u\\ 0 & \nabla_x\left[\overline{E}_0 \sigma+\left(3\overline{E}_2/2-\overline{K}_g \overline{E}_0\right)\theta \right] & \nabla_x\cdot u \end{array}\right),
   \end{equation*}
   \begin{equation*}
     F=-A_0^\delta\left(\begin{aligned}\overline{E}_2 u\cdot(\sigma+\theta)+\frac{5}{3}\overline{E}_2\nabla_x(\sigma+\theta)\nabla_x\cdot u\\
     \left[\overline{E}_0 \sigma+\left(3\overline{E}_2/2-\overline{K}_g \overline{E}_0\right)\theta\right]\partial_t  u+\rho^\delta u\cdot\nabla_x u \\
     \nabla_x\cdot\Big\{\left[\,\overline{E}_0 \sigma+\left(3\overline{E}_2/2-\overline{K}_g \overline{E}_0\right)\theta\,\right]u\Big\}
     \end{aligned}\right).
   \end{equation*}
   Here, $\mathbb{I}$ is the $3\times 3$ identity matrix and $e_i$'s $(i=1,2,3)$ are the standard unit row base vectors in $\R^3$. $(\cdot)^\mathrm{T}$ denotes the transpose of row vectors. Similar to the arguments in section \ref{Sec-Uni-bd}, the system \eqref{U-d} is strictly hyperbolic. Hence by the Friedrichs’ existence theory of the linear hyperbolic system (c.f. Chapter 2 in \cite{Majda-1984Book}), the existence of smooth solution to the system \eqref{U-d} is guaranteed and in addition for $s\geqslant 3$, we have the following energy inequality:
   \begin{equation}\label{U-d-Es}
     \frac{\d}{\d t}\|U_d\|_{H^s}^2 \leqslant C_3\|U_d\|_{H^s}^2+C_4 \|U_d\|_{H^s},
   \end{equation}
   where $C_3$ and $C_4$ are depending on the $H^{s+1}$ norms of $(\rho^\delta,\u^\delta,\E^\delta)$ and $(\sigma,u,\theta)$. A simple Gr\"{o}nwall inequality combining with the estimates \eqref{Ac-sys-Es} implies that $\|(\rho_d^\delta,u_d^\delta,\E_d^\delta)\|_{H^s}$ is bounded by a constant that depends only on $\tau$ and $\|(\sigma^\i,u^\i,\theta^\i)\|_{H^{s+1}}$. We obtain the conclusion as announced.
\end{proof}

\subsection{Local Fermi-Dirac Distribution and Proof of Theorem \ref{Thm-AC-Limit}}
Recall the Proposition \ref{Prop-Relation-rhoEfT}. Since $0<\frac{\overline{\rho}}{\overline{\E}^{3/5}}< J$, we can choose $\delta_2>0$ sufficient small such that for any $0<\delta \leqslant \delta_2$, there holds $0<\frac{\rho^\delta}{(\E^\delta)^{3/5}}\leqslant J$ because of the estimate \eqref{diff-2-order-Es-1}. Therefore, by \eqref{f} and \eqref{T} we can define
\begin{equation}\label{f-T-del}
  \f^\delta=\phi\left(\frac{\rho^\delta}{(\E^\delta)^{3/5}}\right),\quad T^\delta=\psi(\rho^\delta,\f^\delta).
\end{equation}
In addition, from the Proposition \ref{Prop-fT-Es}, we can deduce that for each $0<\delta \leqslant \delta_2$, there exist constants $\f_m^\delta$, $\f_M^\delta$ and $T_m^\delta>0$ such that
\begin{equation}\label{fT-del-Es}
  \f_m^\delta \leqslant \f(t,x) \leqslant \f_M^\delta,\qquad T_m^\delta< \min_{(t,x)\in[0,\tau]\times\R^3} T(t,x)\leqslant \max_{(t,x)\in[0,\tau]\times\R^3} T(t,x) <2 T_m^\delta.
\end{equation}
Obviously, we should take
\begin{equation}\label{delta-0}
  \delta_0=\min\{\,\delta_1,\,\delta_2 \,\}.
\end{equation}

Now we introduce a local Fermi-Dirac distribution governed by $(\f^\delta,\u^\delta,T^\delta)$ as following
\begin{equation}\label{}
  \mu^\delta(t,x,v)=\frac{1}{1+e^{\frac{|v-\u^\delta(t,x)|^2}{2T^\delta(t,x)}-\f^\delta(t,x)}}.
\end{equation}
Similar to the process of proving compressible Euler limit, Hilbert expansion around local Fermi-Dirac distribution $\mu^\delta$ takes the form
\begin{equation*}
  F_\eps=\mu^\delta+\sum_{n=1}^{3}\eps^n F_n^\delta+\eps^3 F_{R,\eps}^\delta.
\end{equation*}
Here the coefficients $F_n^\delta$ ($n=1,2,3$) and the remainder $F_{R,\eps}^\delta$ are determined by a same argument in section \ref{Sec-Uni-bd}. As a consequence of Theorem \ref{Thm-Euler-limit}, we can obtain
\begin{equation}\label{uni-bd}
  \sup_{0\leqslant t \leqslant \tau} \|F_\eps-\mu^\delta\|_2+\sup_{0\leqslant t \leqslant \tau} \|F_\eps-\mu^\delta\|_\infty\leqslant C_\tau \eps,
\end{equation}
where the constant $C_\tau>0$ is depending on $\tau$, $\mu^\delta$ and $F_n^\delta$, $n=1,2,3$.

In fact, noting the definition of $G$ \eqref{G}, by Taylor expansion we see that $\mu^\delta$ is close to $\mu_0+\delta G$. That is
\begin{lemma}
  Let $\delta_0$, $\mu_0$ and $G$ be as in \eqref{delta-0}, \eqref{mu-0} and \eqref{G} respectively and the assumption \eqref{Hypo} hold. Assume $(\rho^\delta,\u^\delta,\E^\delta)$ is the smooth solution to the compressible Euler system constructed in Lemma \ref{Lem-lifespan} and $(\sigma,u,\theta)$ is the smooth acoustic solution. Then for each $0<\delta\leqslant \delta_0$, there holds
  \begin{equation}\label{mu-del-2-order-Es}
    \sup_{0\leqslant t \leqslant \tau} \|\mu^\delta(t)-\mu_0-\delta G(t)\|_2+\sup_{0\leqslant t \leqslant \tau} \|\mu^\delta(t)-\mu_0-\delta G(t)\|_\infty\leqslant C_5 \delta^2,
  \end{equation}
  where the constant $C_5$ depends on $\tau$ and $(\sigma^\i,u^\i,\theta^\i)$.
\end{lemma}
\begin{proof}
  Firstly, we should expand $\f^\delta$ and $T^\delta$ given by \eqref{f-T-del}. Applying Taylor expansion to
  \begin{equation*}
    \f^\delta=\phi\left(\frac{\rho^\delta}{(\E^\delta)^{3/5}}\right)\equiv \phi(\rho^\delta,\E^\delta),
  \end{equation*}
  to acquire
  \begin{equation}\label{f-del-Exp}
    \begin{aligned}
    \f^\delta=& \phi(\overline{\rho},\overline{\E}) + \frac{\partial\phi}{\partial\rho^\delta}(\overline{\rho},\overline{\E})(\rho^\delta-\overline{\rho}) +\frac{\partial\phi}{\partial\E^\delta}(\overline{\rho},\overline{\E})(\E^\delta-\overline{\E})\\
    &\qquad+\frac{1}{2}\left((\rho^\delta-\overline{\rho})\frac{\partial}{\partial\rho^\delta}+ (\E^\delta-\overline{\E})\frac{\partial}{\partial\E^\delta}\right)^2 \phi(\overline{\rho}+r_1(\rho^\delta-\overline{\rho}),\overline{\E}+r_1(\E^\delta-\overline{\E})),
    \end{aligned}
  \end{equation}
  where $0\leqslant r_1\leqslant 1$ and
  \begin{equation*}
    \phi(\overline{\rho},\overline{\E})=1,\quad \frac{\partial\phi}{\partial\rho^\delta}(\overline{\rho},\overline{\E}) =\frac{\phi'(\overline{\rho},\overline{\E})}{\overline{\E}^{3/5}}, \quad \frac{\partial\phi}{\partial\E^\delta}(\overline{\rho},\overline{\E})=\phi'(\overline{\rho},\overline{\E}) \left(-\frac{3\overline{\rho}}{5\overline{\E}^{8/5}}\right).
  \end{equation*}
  Since $(\sigma,u,\theta)$ is smooth with the energy estimate \eqref{Ac-sys-Es} and $(\rho_d^\delta,\u_d^\delta,\E_d^\delta)$ is smooth with the estimate \eqref{diff-2-order-Es}, the last term on the right hand side of \eqref{f-del-Exp} is $O(\delta^2)$ under the assumption \eqref{Hypo}. We denote it by $\f_d^\delta\delta^2$ with
  \begin{equation}\label{diff-2-order-f-Es}
    \|\f_d^\delta\|_{H^s} \leqslant C.
  \end{equation}
  Thus, utilizing the relation $\phi'(\overline{\rho},\overline{\E})=\frac{1}{\varphi'(1)}$, we obtain
  \begin{equation}\label{f-del-Exp-1}
    \f^\delta\equiv 1+ (\sigma - \overline{K}_g \theta) \delta +\f_d^\delta\delta^2.
  \end{equation}
  Similarly, the expansion of $T^\delta=\psi(\rho^\delta,\f^\delta)$ is
  \begin{equation}\label{T-del-Exp}
    \begin{aligned}
    T^\delta=&\psi(\overline{\rho},1) + \frac{\partial\psi}{\partial\rho^\delta}(\overline{\rho},1)(\rho^\delta-\overline{\rho}) +\frac{\partial\psi}{\partial\f^\delta}(\overline{\rho},1)(\f^\delta-1)\\
    &\qquad+\frac{1}{2}\left((\rho^\delta-\overline{\rho})\frac{\partial}{\partial\rho^\delta}+ (\f^\delta-1)\frac{\partial}{\partial\f^\delta}\right)^2 \psi(\overline{\rho}+r_2(\rho^\delta-\overline{\rho}),1+r_2(\f^\delta-1))\\
    \equiv&1 + \theta \delta +T_d^\delta\delta^2,
    \end{aligned}
  \end{equation}
  with $0\leqslant r_2\leqslant 1$ and
  \begin{equation}\label{diff-2-order-T-Es}
    \|T_d^\delta\|_{H^s} \leqslant C.
  \end{equation}

  From \eqref{diff-2-order-f-Es} and \eqref{diff-2-order-T-Es}, we know the difference variables $(\f_d^\delta,u_d^\delta,T_d^\delta)$ are smooth. Consequently, we can define the smooth parameters $\f^{\delta,z}$, $\u^{\delta,z}$ and $T^{\delta,z}$ by
  \begin{equation*}
    \f^{\delta,z}=  1+ (\sigma - \overline{K}_g \theta) z + \f_d^\delta z^2, \;\; \u^{\delta,z}=  u z + u_d^\delta z^2, \;\; T^{\delta,z}=  1 + \theta z + T_d^\delta z^2,
  \end{equation*}
  and the auxiliary local Fermi-Dirac distribution by
  \begin{equation*}
    \mu(z) \equiv \mu^{\delta,z} = \frac{1}{1+e^{\frac{|v-\u^{\delta,z}(t,x)|^2}{2T^{\delta,z}(t,x)}-\f^{\delta,z}(t,x)}}.
  \end{equation*}
  $(\f^{\delta,\delta},\u^{\delta,\delta},T^{\delta,\delta})=(\f^\delta,\u^\delta,T^\delta)$ implies that $\mu^{\delta,\delta}=\mu^\delta$. We expand $\mu(z)$ as a function of $z$. By Taylor’s formula, $\mu(z)$ can be written as
  \begin{equation}\label{mu-z-Exp}
    \mu(z)=\mu(0) + \mu'(0)z + \frac{\mu^{\prime\prime}(z_*)}{2} z^2,
  \end{equation}
  for some $0 \leqslant z_* \leqslant z$, which may depend on $(t,x,v)$ and $\delta$. By computing directly, we acquire that $\mu(0)=\mu_0$ and
  \begin{equation*}
    \begin{aligned}
    \frac{\d}{\d z} \mu(z) =& \left\{ (\f^{\delta,z})' + (v-\u^{\delta,z})\cdot \frac{(\u^{\delta,z})'}{T^{\delta,z}} + \frac{|v-\u^{\delta,z}|^2}{T^{\delta,z}} \frac{(T^{\delta,z})'}{2T^{\delta,z}} \right\} \mu(z)(\,1-\mu(z)\,)\\
    \equiv & \; D^{\delta,z}\mu(z)(\,1-\mu(z)\,),
    \end{aligned}
  \end{equation*}
  where the symbol $'$ means $\frac{\d}{\d z}$ and
  \begin{equation*}
     (\f^{\delta,z})' = (\sigma - \overline{K}_g \theta) + 2\f_d^\delta  z, \quad (\u^{\delta,z})' = u + 2 u_d^\delta z,\quad (T^{\delta,z})' =\theta + 2T_d^\delta  z.
  \end{equation*}
  Thus $\big[\, (\f^{\delta,z})', \, (\u^{\delta,z})', \, (T^{\delta,z})' \,\big](0)=\big[\, (\sigma - \overline{K}_g \theta), \, u, \, \theta \,\big]$ and hence we obtain
  \begin{equation*}
    \mu'(0) = \left\{\, \sigma + v\cdot u + (\frac{|v|^2}{2} - \overline{K}_g )\theta \,\right\} \mu_0(1-\mu_0)= G(t,x,v).
  \end{equation*}
  We further take the second derivative of $\mu(z)$ to get
  \begin{equation*}
    \mu^{\prime\prime}(z)= (D^{\delta,z})' \mu(z) (\,1-\mu(z)\,)+ (D^{\delta,z})^2\mu(z)(\,1-\mu(z)\,)(\,1-2\mu(z)\,),
  \end{equation*}
  where
  \begin{equation*}
    \begin{aligned}
    (D^{\delta,z})' = (\f^{\delta,z})^{\prime\prime} - \frac{|(\u^{\delta,z})'|^2}{T^{\delta,z}} +& (v-\u^{\delta,z})\cdot\left(  \frac{ (\u^{\delta,z})^{\prime\prime} }{T^{\delta,z}} -2\frac{(\u^{\delta,z})'(T^{\delta,z})'}{(T^{\delta,z})^2}  \right) \\
     &\qquad+ |v-\u^{\delta,z}|^2\left( \frac{(T^{\delta,z})^{\prime\prime}}{2(T^{\delta,z})^2}  -  \frac{((T^{\delta,z})')^2}{(T^{\delta,z})^3} \right).
    \end{aligned}
  \end{equation*}
  The terms $(\f^{\delta,z})^{\prime\prime}$, $(\u^{\delta,z})^{\prime\prime}$ and $(T^{\delta,z})^{\prime\prime}$ are given by
  \begin{equation*}
    (\f^{\delta,z})^{\prime\prime} = 2\f_d^\delta ,\quad (\u^{\delta,z})^{\prime\prime} = 2 u_d^\delta,\quad (T^{\delta,z})^{\prime\prime} = 2T_d^\delta.
  \end{equation*}
  Thus by taking $z=\delta$ we obtain
  \begin{equation*}
    \mu^{\delta,\delta}= \mu_0 + G \delta+ \frac{\mu^{\prime\prime}(\delta_*)}{2} \delta^2, \quad \text{ for some } 0\leqslant \delta_* \leqslant \delta.
  \end{equation*}

  To certify the estimate \eqref{mu-del-2-order-Es}, it is sufficient to show that $\|\mu^{\prime\prime}(\delta_*)\|_2+\|\mu^{\prime\prime}(\delta_*)\|_\infty$ is bounded uniformly in $\delta_*$. To this end, from the uniform estimates \eqref{diff-2-order-Es}, \eqref{diff-2-order-f-Es} and \eqref{diff-2-order-T-Es}, we deduce the uniform pointwise estimates of
  \begin{equation*}
    \f^{\delta,z},\; \u^{\delta,z},\; T^{\delta,z},\; (\f^{\delta,z})',\; (\u^{\delta,z})',\; (T^{\delta,z})',\; (\f^{\delta,z})^{\prime\prime},\; (\u^{\delta,z})^{\prime\prime},\; (T^{\delta,z})^{\prime\prime},
  \end{equation*}
  for every $0 \leqslant z=\delta_* \leqslant \delta_0$ and any $t\leqslant \tau$.

  The estimate \eqref{mu-del-2-order-Es} is established as announced.
\end{proof}

\begin{proof}[{\bf Proof of Theorem \ref{Thm-AC-Limit}}]
  Recall the expansion \eqref{Global-Exp} and the definition \eqref{G} of $G$, we can obtain
  \begin{equation*}
    \begin{aligned}
    &\sup_{0\leqslant t\leqslant \tau} \|G_\eps-G\|_\infty + \sup_{0\leqslant t\leqslant \tau} \|G_\eps-G\|_2\\
    &\qquad=\sup_{0\leqslant t\leqslant \tau} \|\frac{F_\eps-\mu_0}{\delta}-G\|_\infty +\sup_{0\leqslant t\leqslant \tau} \|\frac{F_\eps-\mu_0}{\delta}-G\|_2\\
    &\qquad\qquad =\sup_{0\leqslant t\leqslant \tau} \|\frac{F_\eps-\mu^\delta}{\delta}+\frac{\mu^\delta-\mu_0-\delta G}{\delta}\|_\infty +\sup_{0\leqslant t\leqslant \tau} \|\frac{F_\eps-\mu^\delta}{\delta}+\frac{\mu^\delta-\mu_0-\delta G}{\delta}\|_2.
    \end{aligned}
  \end{equation*}
  Then from \eqref{uni-bd} and \eqref{mu-del-2-order-Es}, we deduce that
  \begin{equation*}
    \sup_{0\leqslant t\leqslant \tau} \|G_\eps-G\|_\infty + \sup_{0\leqslant t\leqslant \tau} \|G_\eps-G\|_2 \leqslant C\,(\frac{\eps}{\delta}+\delta).
  \end{equation*}
\end{proof}

\section{Appendix}
  \begin{lemma}
  For any $0<q<1$, there exists $C>0$ such that
  \begin{equation*}
    \left\langle\mu^{-q/2}\L f, \mu^{-q/2} f \right\rangle_{L_v^2} \geqslant \frac{1}{2} |\mu^{-q/2} f|_{\nu(\mu)}^2- C|f|_{\nu(\mu)}^2.
  \end{equation*}
\end{lemma}

\begin{proof}
  Throughout the proof, we always use the estimates \eqref{mu-Es}. For some $r>0$, we define the smooth cutoff function $\chi(s)$ by
  \begin{equation*}
    \chi(s)=\left\{
    \begin{aligned}
    &\quad 1,\qquad s\geqslant 2r,\\
    &\in [0,1],\quad r \leqslant s \leqslant 2r,\\
    &\quad 0,\qquad 0 \leqslant s \leqslant r.
    \end{aligned}
    \right.
  \end{equation*}
  Then we split $\K$ as $\K=\K^{1-\chi}+\K^\chi$ with
  \begin{equation*}
  \begin{aligned}
    &\K^{1-\chi}f = \iint_{\R^3\times\S^2}|(v_*-v)\cdot\omega|\frac{\N}{\sqrt{\mu(1-\mu)}}(1-\chi(|v_*-v|)) \left(\frac{f}{\sqrt{\mu(1-\mu)}}\right)_*\d \omega\d v_*\\
    &\quad-\iint_{\R^3\times\S^2}|(v_*-v)\cdot\omega|\frac{\N}{\sqrt{\mu(1-\mu)}}(1-\chi(|v_*-v|)) \left\{\left(\frac{f}{\sqrt{\mu(1-\mu)}}\right)_*^\prime+\left(\frac{f}{\sqrt{\mu(1-\mu)}}\right)^\prime\right\}\d \omega\d v_*\\
    &\qquad \,=\K_1^{1-\chi}-\K_2^{1-\chi}.
    \end{aligned}
  \end{equation*}
  Thus
  \begin{equation*}
    \left\langle\mu^{-q/2}\K^{1-\chi} f, \mu^{-q/2}f \right\rangle_{L_v^2} = \left\langle\K_1^{1-\chi} f, \mu^{-q}f \right\rangle_{L_v^2} - \left\langle\K_2^{1-\chi} f, \mu^{-q}f \right\rangle_{L_v^2}.
  \end{equation*}
  Recall the definition of $\N$ \eqref{N}, since $1-\mu$ is bounded, we employ Cauchy-Schwarz's inequality to obtain for any $0< q<1$
  \begin{equation*}
    \begin{aligned}
    \left|\left\langle\K_1^{1-\chi} f, \mu^{-q}f \right\rangle_{L_v^2}\right| \leqslant & C\iiint_{|v_*-v|\leqslant 2r}|(v_*-v)\cdot\omega||ff_*|\sqrt{\mu\mu_*}\mu^{-q} \d \omega\d v_*\d v\\
    \leqslant &  C\left\{\iint_{\R^3\times\S^2}|f|^2\mu^{\frac{1-q}{2}}\left(\int_{\R^3}|(v_*-v)\cdot\omega|\mu_*^{\frac{1-q}{4}}\d v_*\right)\d \omega\d v\right\}^{1/2}\times\\
    &\qquad\left\{\iint_{\R^3\times\S^2}|(v_*-v)\cdot\omega||f_*|^2\mu_*^{\frac{1-q}{4}}\left(\int_{\R^3}|(v_*-v)\cdot\omega|\mu^{\frac{1-q}{2}}\d v_*\right)\d \omega\d v\right\}^{1/2}\\
    \leqslant  & C |f|_{\nu(\mu)}^2,
    \end{aligned}
  \end{equation*}
  where we have used $\mu^{-q}(v) \leqslant C\mu^{-\tfrac{q}{2}}(v)\mu^{-(\tfrac{1}{2}-\tfrac{1-q}{4})}(v_*)$, because $|v-\u|^2 \leqslant \frac{1+q}{2q}|v_*-\u|^2+Cr^2$ provided $|v_*-v|\leqslant 2r$.
  On the other hand, for the part $\K_2^{1-\chi}$, $|v_*^\prime-v'|=|v_*-v|\leqslant 2r$ implies that
  \begin{equation*}
    |v_*^\prime-\u|\geqslant |v'-\u|-2r,\quad |v'-\u|\geqslant |v-\u|-2r.
  \end{equation*}
  It immediately follows that for any $0 <q'<1$
  \begin{equation*}
    \sqrt{\mu_*\mu_*^\prime}+\sqrt{\mu_*\mu^\prime}\leqslant C\mu_*^{(1+q')/4}\mu^{(1+q')/4},
  \end{equation*}
  and for any $q<q'<1$
  \begin{equation*}
    \mu^{-q}(v)\leqslant C \mu^{-q'/2}(v_*^\prime)\mu^{-q'/2}(v^\prime).
  \end{equation*}
  As a result, since $\d \omega\d v_*^\prime\d v'=\d \omega\d v_*\d v$, we get
  \begin{equation*}
    \begin{aligned}
    \left|\left\langle\K_2^{1-\chi} f, \mu^{-q}f \right\rangle_{L_v^2}\right| \leqslant & C\iiint_{|v_*^\prime-v'|\leqslant 2r}|v_*^\prime-v'|\left(\,|f'|\sqrt{\mu_*^\prime}+|f_*^\prime|\sqrt{\mu'}\,\right)\sqrt{\mu_*}\mu^{-q}|f| \d \omega\d v_*^\prime\d v'\\
    \leqslant &  C\left\{\iint_{\R^3\times\S^2}|f'|^2(\mu')^{\frac{1-q'}{4}}\left(\int_{\R^3}|v_*^\prime-v'|(\mu_*^\prime)^{\frac{1-q'}{4}}\d v_*^\prime\right)\d \omega\d v'\right\}^{1/2}\times\\
    &\qquad\left\{\iint_{\R^3\times\S^2}|f|^2\mu_*^{\frac{1-q'}{4}} \left(\int_{\R^3}|v_*-v|\mu^{\frac{1-q'}{4}}\d v\right)\d \omega\d v_*\right\}^{1/2}\\
    \leqslant  & C |f|_{\nu(\mu)}^2,
    \end{aligned}
  \end{equation*}

  Now we treat the part
  \begin{equation*}
    \left\langle\mu^{-q/2}\K^\chi f, \mu^{-q/2}f \right\rangle_{L_v^2} = \left\langle\K_1^\chi f, \mu^{-q}f \right\rangle_{L_v^2} - \left\langle\K_2^\chi f, \mu^{-q}f \right\rangle_{L_v^2}.
  \end{equation*}
  For the first term on the right hand side above, since for $|v_*-v|\geqslant r$,
  \begin{equation*}
    |v_*-v|(\mu_*\mu)^{\frac{1-q}{2}}\leqslant C(\mu_*\mu)^{\frac{1-q}{4}}\exp\left\{ -\frac{1-q}{32T}r^2 \right\},
  \end{equation*}
  we obtain by using Cauchy-Schwarz's inequality that
  \begin{equation*}
    \begin{aligned}
    &\left|\left\langle\K_1^\chi f, \mu^{-q}f \right\rangle_{L_v^2}\right| \leqslant C \iiint_{|v_*-v|\geqslant r}|v_*-v||f_*||f| \mu_*^{\frac{1}{2}} \mu^{\frac{1-q}{2}}\mu^{-\frac{q}{2}}\d\omega\d v_*\d v\\
    &\quad\leqslant C\left(\iiint_{|v_*-v|\geqslant r}|v_*-v||\mu^{-q/2}f|^2(\mu_*\mu)^{\frac{1-q}{2}}\d\omega\d v_*\d v\right)^{1/2} \left(\iiint_{|v_*-v|\geqslant r}|v_*-v||f_*|^2\mu_*^{\frac{1}{2}}\mu^{\frac{1-q}{2}}\d\omega\d v_*\d v\right)^{1/2}\\
    &\quad\leqslant C \exp\left\{ -\frac{1-q}{32T}r^2 \right\}|\mu^{-q/2}f|_{L_v^2}^2.
    \end{aligned}
  \end{equation*}
  Next, we turn to the part
  \begin{equation*}
    \K_2^\chi f =2\iint_{\R^3\times\S^2}\chi(|v_*-v|)|(v_*-v)\cdot\omega|\frac{\N}{\sqrt{\mu(1-\mu)}} \left(\frac{f}{\sqrt{\mu(1-\mu)}}\right)^\prime\d \omega\d v_*.
  \end{equation*}
  By using the variable changing $v_*-v=V$ and the classical transformation (see page 43 in \cite{Glassey-1996Book})
  \begin{equation}\label{Vel-Dec}
    \d \omega \d V=\frac{2\d V_\bot \d V_{\p}}{|V_{\p}|^2},\quad V_{\p}=(V\cdot\omega)\omega\in \R^3,\,V_\bot=V-V_{\p}\in \R^2,
  \end{equation}
  we arrive at
  \begin{equation*}
        |\K_2^\chi f| \leqslant C \iint_{|V|\geqslant r}|V_{\p}||f(v+V_{\p})| \mu^{-1/2}(v+V_{\p})\mu^{1/2}(v)\mu(v+V)\frac{\d V_\bot \d V_{\p}}{|V_{\p}|^2}.
  \end{equation*}
  Let
  \begin{equation*}
    \eta=v+V_{\p},\quad \zeta=\frac{1}{2}(v+\eta),
  \end{equation*}
  then
  \begin{equation*}
    V_\bot\cdot\zeta=V_\bot\cdot v=V_\bot \cdot\eta.
  \end{equation*}
  Thus
  \begin{equation*}
    \begin{aligned}
    &-\frac{|v-\u|^2}{2}+\frac{|\eta-\u|^2}{2}-|\eta+V_\bot-\u|^2 =-\frac{|v-\u|^2}{2}-\frac{|\eta-\u|^2}{2}-2V_\bot\cdot(\zeta-\u)-|V_\bot|^2\\
    &\qquad=-\frac{|v-\u|^2}{2}-2V_\bot\cdot(\zeta-\u)-2|\zeta-\u|^2+2(\zeta-\u)\cdot(v-\u)-\frac{|v-\u|^2}{2}-|V_\bot|^2\\
    &\qquad=(-|V_\bot|^2-2V_\bot\cdot(\zeta-\u)-|\zeta-\u|^2)+(-|\zeta-\u|^2+2(\zeta-\u)\cdot(v-\u)-|v-\u|^2)\\
    &\qquad=-|V_\bot+\zeta-\u|^2-\frac{1}{4}|V_{\p}|^2.
    \end{aligned}
  \end{equation*}
  Note the estimates \eqref{mu-Es}, we acquire
  \begin{equation*}
    \begin{aligned}
        |\K_2^\chi f| \leqslant &C \iint_{\R^3\times\R^2}\frac{\chi(|V|)}{|V_{\p}|}|f(v+V_{\p})| \exp\left\{-\frac{|V_\bot+\zeta-\u|^2}{2T}-\frac{1}{8T}|V_{\p}|^2\right\}\d V_\bot \d V_{\p}\\
        \leqslant &C \int_{\R^3}\frac{1}{|V_{\p}|}|f(v+V_{\p})| \exp\left\{-\frac{|V_{\p}|^2}{8T}-\frac{|\zeta_{\p}|^2}{2T}\right\} \d V_{\p},
    \end{aligned}
  \end{equation*}
  where we have used that $\int_{\R^2}\exp\left\{-\frac{|V_\bot+\zeta_\bot|^2}{2T}\right\}\d V_\bot$ is bounded for $\zeta_{\p}=[(\zeta-\u)\cdot\omega]\omega$ and $\zeta_\bot=\zeta-\zeta_{\p}$.
  As a result, we evaluate that
  \begin{equation*}
    \begin{aligned}
    \left|\left\langle\K_2^\chi f, \mu^{-q}f \right\rangle_{L_v^2}\right|\leqslant & C\iint_{\R^3\times\R^3}\frac{1}{|V_{\p}|} \left\{\mu^{-q/2}(v+V_{\p})|f(v+V_{\p})|\right\} \left\{\mu^{-q/2}(v)|f(v)|\right\}\times\\
    &\qquad \exp\left\{\frac{q|v-\u|^2}{4T}-\frac{q|v+V_{\p}-\u|^2}{4T}-\frac{|V_{\p}|^2}{8T}-\frac{|\zeta_{\p}|^2}{2T}\right\} \d V_{\p}\d v.
    \end{aligned}
  \end{equation*}
  Because
  \begin{equation*}
    \begin{aligned}
    &\frac{q|v-\u|^2}{4T}-\frac{q|v+V_{\p}-\u|^2}{4T}-\frac{|V_{\p}|^2}{8T}-\frac{|\zeta_{\p}|^2}{2T} \\ =&-(1-q)\left\{\frac{|V_{\p}|^2}{8T}+\frac{|\zeta_{\p}|^2}{2T}\right\}+\frac{q}{2T} \left(\frac{|v-\u|^2}{2}-\frac{|v+V_{\p}-\u|^2}{2}-\frac{|V_{\p}|^2}{4}-|\zeta_{\p}|^2\right)\\
    =&-(1-q)\left\{\frac{|V_{\p}|^2}{8T}+\frac{|\zeta_{\p}|^2}{2T}\right\}+\frac{q}{2T} \left(-\frac{|V_{\p}|^2}{2}-(\zeta-\u-\tfrac{1}{2}V_{\p})\cdot V_{\p}-\frac{|V_{\p}|^2}{4}-|\zeta_{\p}|^2\right)\\
    =&-(1-q)\left\{\frac{|V_{\p}|^2}{8T}+\frac{|\zeta_{\p}|^2}{2T}\right\}-\frac{q}{2T} \left(\frac{|V_{\p}|^2}{4}+|\zeta_{\p}|^2+V_{\p}\cdot\zeta_{\p}\right)-\frac{q}{2T} \frac{|V_{\p}|^2}{2}\\
    \leqslant &-(1-q)\left\{\frac{|V_{\p}|^2}{8T}+\frac{|\zeta_{\p}|^2}{2T}\right\},
    \end{aligned}
  \end{equation*}
  it follows that
  \begin{equation*}
    \begin{aligned}
    &\left|\left\langle\K_2^\chi f, \mu^{-q}f \right\rangle_{L_v^2}\right|\\
    &\quad\leqslant  C\int_{\R^3}\frac{1}{|V_{\p}|}\exp\left\{-(1-q)\left\{\frac{|V_{\p}|^2}{8T}+\frac{|\zeta_{\p}|^2}{2T}\right\}\right\}\d V_{\p}\\
    &\qquad\times\int_{\R^3} \left\{\mu^{-q/2}(v+V_{\p})|f(v+V_{\p})|\frac{\nu^{1/2}(v+V_{\p})}{(1+|v+V_{\p}|)^{1/2}}\right\} \left\{\mu^{-q/2}(v)|f(v)|\frac{\nu^{1/2}(v)}{(1+|v|)^{1/2}}\right\} \d v\\
    &\quad\leqslant C|(1+|v|)^{-1/2}\mu^{-q/2}f|_{\nu(\mu)}^2\\
    &\quad\leqslant C\int_{\R^3}(\mathbbm{l}_{|v|\leqslant r_1}+\mathbbm{l}_{|v|\geqslant r_1})\frac{1}{1+|v|}\mu^{-q}|f|^2\nu(\mu)\d v\\
    &\quad \leqslant C|f|_{\nu(\mu)}^2+\frac{C}{1+r_1}|\mu^{-q/2}f|_{\nu(\mu)}^2.
    \end{aligned}
  \end{equation*}
  In summary, by taking $r$ and $r_1$ large enough, we deduce that
  \begin{equation}
    \left\langle\mu^{-q/2}\L f, \mu^{-q/2} f \right\rangle_{L_v^2} \geqslant \frac{1}{2} |\mu^{-q/2} f|_{\nu(\mu)}^2 -C|f|_{\nu(\mu)}^2.
  \end{equation}
\end{proof}

\end{document}